\documentclass[12pt, a4paper]{amsart}
\usepackage{amsmath}
\usepackage{geometry,amsthm,graphics,tabularx,amssymb,shapepar}
\usepackage[cmyk]{xcolor}
\usepackage{appendix}
\usepackage{amscd}
\usepackage[all]{xypic}
\usepackage{ulem}
\usepackage{bm}

\makeatletter
\newcommand*{\rom}[1]{\expandafter\@slowromancap\romannumeral #1@}
\makeatother

\newcommand{\BC}{{\mathbb {C}}}

\newcommand{\BN}{{\mathbb {N}}}

\newcommand{\BR}{{\mathbb {R}}}

\newcommand{\CF}{{\mathcal {F}}}
\newcommand{\CG}{{\mathcal {G}}}

\newcommand{\CL}{{\mathcal {L}}}

\newcommand{\CO}{{\mathcal {O}}}

\newcommand{\RC}{{\mathrm {C}}}
\newcommand{\RD}{{\mathrm {D}}}

\newcommand{\wt}{\widetilde}
\newcommand{\wh}{\widehat}

\newcommand{\q}{\mathfrak q}

\newcommand{\m}{\mathfrak m}

\newcommand{\C}{\mathbb{C}}
\newcommand{\R}{\mathbb R}

\newcommand{\abs}[1]{\lvert#1\rvert}

\newcommand{\cf}{\textit{cf}.~}
\newcommand{\la}{\langle}
\newcommand{\ra}{\rangle}

\newcommand{\be}{\begin {equation}}
\newcommand{\ee}{\end {equation}}
\newcommand{\bee}{\begin {equation*}}
\newcommand{\eee}{\end {equation*}}

\newcommand{\qaq}{\quad\textrm{and}\quad}

\renewcommand{\mid}{\,:\,}

\theoremstyle{Theorem}

\theoremstyle{Theorem}

\theoremstyle{Theorem}

\theoremstyle{Theorem}

\theoremstyle{Plain}

\theoremstyle{remark}

\theoremstyle{remark}

\theoremstyle{Definition}
\newtheorem{dfn}{Definition}[section]

\newtheorem{cord}[dfn]{Corollary}
\newtheorem{prpd}[dfn]{Proposition}
\newtheorem{thmd}[dfn]{Theorem}
\newtheorem{lemd}[dfn]{Lemma}
\newtheorem{remarkd}[dfn]{Remark}
\newtheorem{exampled}[dfn]{Example}
\numberwithin{equation}{section}

\begin{document}

	\title[Formal manifolds]{Function spaces on formal manifolds}
	
	\author[F. Chen]{Fulin Chen}
	\address{School of Mathematical Sciences, Xiamen University,
		Xiamen, 361005, China} \email{chenf@xmu.edu.cn}

	\author[B. Sun]{Binyong Sun}
	\address{Institute for Advanced Study in Mathematics \& New Cornerstone Science Laboratory, Zhejiang University,  Hangzhou, 310058, China}
	\email{sunbinyong@zju.edu.cn}
	
	\author[C. Wang]{Chuyun Wang}
	\address{Institute for Theoretical
		Sciences/ Institute of Natural Sciences, Westlake University/ Westlake Institute for Advanced Study,
		Hangzhou, 310030, China}
	\email{wangchuyun@westlake.edu.cn}
	
	
	\begin{abstract}
		This is a paper in a series that studies smooth relative Lie
		algebra homologies and cohomologies based on the theory of formal manifolds
		and formal Lie groups. In a previous paper, we introduce the notion of formal manifolds and develop  the foundational framework
		of formal manifolds.
		In this paper, we study various function spaces on formal manifolds, including generalizations of vector-valued generalized functions and vector-valued distributions on smooth manifolds to the setting of formal manifolds.  
	\end{abstract}
	
	\maketitle
	
	\tableofcontents

	\section{Introduction and the main results}
	
	This is a sequel to \cite{CSW} in a series to study smooth (co)homogogy theory for representations of general Lie pairs. 
	
	\subsection{Motivation}
	For the construction of Harish-Chandra modules and unitary representations, 
	authors are often interested in (co)homologies of representations of certain Lie pairs $(\q,L)$. Here $\q$ is a finite-dimensional complex Lie
	algebra and $L$ is a compact  Lie group, with additional structures that satisfy specific compatibility conditions (see \cite[(1.64)]{KV} for details). 
	Motivated by the recent development in Lie group representations and  automorphic forms, it is essential to  establish a smooth
	(co)homology theory for representations of
	general Lie pairs $(\q,L)$, with $L$ not necessarily compact. We refer to this theory as smooth relative Lie algebra (co)homology theory, which can be viewed as a generalization of the smooth (co)homology theory for Lie groups (see  \cite{HM} and \cite{BW}) in a relative setting.

	In \cite{CSW}, we formulate and study a notion of what we call formal manifolds, inspired by the notion of formal schemes in algebraic geometry. 
	As one of the main results therein, we prove that finite products  exist in the category of formal manifolds. Then we have a notion of formal Lie gorups, which are the group objects in the category of formal manifolds. We will prove in a subsequent paper that there is an equivalence of categories between the category of formal Lie groups  and that of  Lie pairs. Using this equivalence, we will establish the smooth relative Lie algebra (co)homology theory by studying (co)homologies of representations of formal Lie groups.

	It is well-known that smooth representations of Lie groups are often 
	realized as various function spaces. 
	For example,  the maximal globalization of a Harish-Chandra module for a real reductive group $G$
	can be realized by using the space $\RC^{\infty}(G)$ or $\RC^{-\infty}(G)$,
	and the minimal globalization can be realized by using the space $\RD^{\infty}_c(G)$ or $\RD^{-\infty}_c(G)$ (see \cite{KS,Schm}). Here and below, for a smooth manifold $N$, we use the notations 
	\be\label{intro:funspaces}
	\RC^\infty(N),\quad \RC^\infty_c(N),\quad \RC^{-\infty}(N),\quad \RD_c^{\infty}(N),\quad \RD_c^{-\infty}(N)\quad\text{and}\quad \RD^{-\infty}(N)
	\ee
	to denote the spaces of (complex-valued) smooth functions, compactly supported smooth functions, generalized functions, compactly supported smooth densities, compactly supported distributions and distributions on $N$, respectively. 
	On the other hand, in the  smooth (co)homology theory of Lie groups,  the standard (projective or injective) resolutions of a smooth representation $V$
	can be constructed by using the spaces of suitable $V$-valued smooth functions, distributions or generalized functions (see \cite{HM,BW,KS}).
	
	Thus, to establish the  smooth (co)homology theory for general Lie pairs (or equivalently, formal Lie groups), it is desirable 
	to generalize Schwartz's (vector-valued) distribution theory  (see \cite{Sc2,Sc3,Sc4}) to the formal manifold setting. 
	% As a stone step towards the smooth relative Lie algebra (co)homology theory, we  need to  construct standard resolutions for formal Lie groups with coefficients in formal functions, formal generalized functions, compactly supported formal densities,and compactly supported formal distributions. To  achieve this goal, in a forthcoming paper, we will define the de Rham complexes for formal manifolds with coefficients in these spaces and prove the corresponding Poincar\'e's lemma.
	
	\subsection{Notations and conventions} Before presenting the main results of this paper, we provide some notations and conventions used throughout this paper. 
	For every smooth manifold $N$ and $k\in \BN$, denote by $N^{(k)}$ the locally ringed space $(N, \CO_N^{(k)})$ over $\mathrm{Spec}(\BC)$, where 
	\[
	\CO_N^{(k)}(U):=\RC^\infty(U)[[y_1, y_2, \dots, y_k]]\quad (U\ \text{is an open subset of}\ N).
	\]
	Recall from \cite{CSW} that a formal manifold  is a locally ringed space $(M, \CO)$ over $\mathrm{Spec}(\BC)$ such that
	\begin{itemize}
		\item the topological space $M$ is paracompact and Hausdorff; and
		\item for every $a\in M$, there is an open  neighborhood $U$ of $a$ in $M$ and $n,k\in \BN$ such that $(U, \CO|_U)$ is isomorphic to $(\R^n)^{(k)}$ as locally ringed spaces over $\mathrm{Spec}(\BC)$.
	\end{itemize} For every $a\in M$, the numbers $n$ and $k$  are respectively called the dimension and the degree of $M$ at $a$,  denoted by $\dim_a(M)$ and $\deg_a(M)$. We denote by $\pi_0(M)$ the set of all connected components in $M$, which may or may not be countable. An element in $\CO(M)$ is called a formal function on $M$.
	
	In this paper,  by an LCS, 
	we mean a 
	locally convex topological vector
	space over $\C$, which may or may not be Hausdorff. 
	However, a complete (or quasi-complete) LCS is always assumed to be Hausdorff. 
	All function spaces in \eqref{intro:funspaces} are equipped with the usual topologies so that they are naturally complete LCS. Additionally, we equip the spaces of formal power series and polynomials with the
	term-wise convergence topology and the usual inductive limit topology, respectively. 
	
	For two LCS $E_1$ and $E_2$, let $\CL(E_1,E_2)$ denote the space of all continuous linear maps from $E_1$ to $E_2$. When $E_2=\C$, we also set $E_1':=\CL(E_1,E_2)$. 
	Unless otherwise mentioned,  $\CL(E_1,E_2)$ is equipped with the strong topology.
	
	As stated in \cite{Gr}, there are three useful topological tensor products on LCS:
	the inductive tensor product $\otimes_{\mathrm{i}}$, the projective tensor product $\otimes_\pi$, and the epsilon tensor product $\otimes_\varepsilon$.
	We denote the quasi-completions and completions of these topological tensor products by
	\[\widetilde\otimes_{\mathrm{i}},\,\, \widetilde\otimes_\pi,\,\, \widetilde\otimes_{\varepsilon}\quad \text{and}\quad 
	\widehat\otimes_{\mathrm{i}}, \,\,\widehat\otimes_\pi,\,\, \widehat\otimes_{\varepsilon},\] respectively.
	Some basics on these topological tensor products are reviewed in Appendixes \ref{appendixA} and \ref{appendixB}.

	Throughout this paper, let $(M,\CO)$ be a formal manifold,
	and let $\CF$ be a sheaf of $\CO$-modules. Additionally, except in the appendix, let $E$ be a quasi-complete LCS.
	For every open subset $U$  of $M$,  we equip $\CF(U)$ with the smooth topology (see \cite[Definition 4.1]{CSW}). When $M=N^{(k)}$ with $N$ a smooth manifold and $k\in \BN$, we have the LCS identifications (see Example \ref{ex:E-valuedpowerseries})
	\[
	\CO(M)=\RC^\infty(N)\widetilde\otimes_\pi \C[[y_1,y_2,\dots,y_k]]=
	\RC^\infty(N)\widehat\otimes_\pi \C[[y_1,y_2,\dots,y_k]].
	\]

	\subsection{Main results}
	In Section \ref{sec:formalden},  for every open subset $U$  of $M$,  
	we introduce a subspace $\RD_c^{\infty}(U;\CF)$ of $\CF(U)'$, and equip it with a certain inductive limit topology (in the category of LCS). Set \[\RC^{-\infty}(U;\CF;E):=\CL(\RD_c^{\infty}(U;\CF),E).\]  
	Then we prove that the assignments (see Proposition \ref{prop:F'smisacosheaf} and Lemma \ref{lem:GFsheaf})
	\[\RD_c^{\infty}(\CF): U\mapsto \RD_c^{\infty}(U;\CF)\quad \text{and}\quad
	\RC^{-\infty}(\CF;E):U\mapsto \RC^{-\infty}(U;\CF;E)\]
	respectively form a cosheaf and a sheaf of $\CO$-modules on $M$. As usual, by a cosheaf    of 
	$\CO$-modules on $M$, we mean a cosheaf $\mathcal{G}$ (see \cite[Chapter V, Definition 1.1]{Br2}) on $M$ together with  an $\CO(U)$-module structure on $\mathcal{G}(U)$  for every open
	subset $U$ of $M$, with the following property: for all open subsets $V\subset U$ of $M$,
	the extension map $\mathcal G(V)\rightarrow \mathcal G(U)$ is an $\CO(U)$-module homomorphism. Here the $\CO(U)$-module structure on $\CG(V)$ is
	obtained from its $\CO(V)$-module structure through the homomorphism $\CO(U)\rightarrow \CO(V)$.

	On the other hand, for an open subset $U$  of $M$, denote by  $\CF_c(U)$  the space of compactly supported  sections 
	of  $\CF(U)$, equipped with the usual inductive limit topology.  In Sections \ref{sec:formaldis} and \ref{sec:comformaldis}, 
	we prove  that the assignments
	\[\RD^{-\infty}(\CF;E): U\mapsto \RD^{-\infty}(U;\CF;E)\quad\text{and}\quad
	\RD^{-\infty}_c(\CF;E): U\mapsto \RD^{-\infty}_c(U;\CF;E)
	\] respectively form  a sheaf and a cosheaf of  $\CO$-modules on $M$, where 
	\[\RD^{-\infty}(U;\CF;E):=\CL(\CF_c(U),E)\] and $\RD^{-\infty}_c(U;\CF;E)$ is the subspace of  $\RD^{-\infty}(U;\CF;E)$ consisting of the compactly
	supported sections.
	Write \[\RC^{-\infty}(\CF)=\RC^{-\infty}(\CF;\C) \qaq \RC^{-\infty}(U;\CF)=\RC^{-\infty}(U;\CF;\C).\]
	Similarly, we use the notations $\RD^{-\infty}(\CF), \RD^{-\infty}(U;\CF)$, $\RD_c^{-\infty}(\CF)$, and $\RD^{-\infty}_c(U;\CF)$.
	
	Let $\mathcal{S}$ be one of the sheaves
	\[
	\CF,\quad \RC^{-\infty}(\CF;E),\quad \RD^{-\infty}(\CF;E),
	\]
	and let $\mathcal{C}$ be one of the  cosheaves 
	\[
	\RD_c^{\infty}(\CF),\quad \CF_c,\quad \RD^{-\infty}_c(\CF;E).
	\]
	We prove that the restriction maps in the sheaf $\mathcal{S}$ and the extension maps in the cosheaf $\mathcal{C}$ are automatically continuous.  
	Furthermore, let $\{U_\gamma\}_{\gamma\in \Gamma}$ be an open cover of $M$. 
	Then we prove that  the continuous linear map 
	\[
	\mathcal{S}(M)\rightarrow \prod_{\gamma\in \Gamma} \mathcal{S}(U_\gamma)
	\]
	induced by restriction maps is a  closed topological embedding under certain natural conditions, and that the continuous linear map 
	\[
	\bigoplus_{\gamma\in \Gamma} \mathcal{C}(U_\gamma)\rightarrow \mathcal{C}(M)
	\]
	induced by extension maps is open and surjective. 
	See \cite[Lemmas 4.5 and 4.6]{CSW}, and Propositions \ref{prop:GFclosedembedding}, \ref{prop:D-inftyclosdtopim}, \ref{prop:extconD},  \ref{prop:FcMopen} and \ref{prop:DFcEopenonto} for details. 
	
	Thus, the structure of LCS $\mathcal{S}(M)$ and $\mathcal{C}(M)$ is in general locally determined. 
	On the other hand, we have the following local result (see Propositions \ref{prop:charlcsosm'}, \ref{prop:charGO}, \ref{prop:localD-infty}, and \ref{prop:desDcMO}).

	\begin{thmd}\label{thm:introlocally}
		Assume that $M=N^{(k)}$ with $N$ a smooth manifold and $k\in \BN$. Then we have the following LCS identifications: 
		\begin{eqnarray*}
			&&\RD_c^{\infty}(M;\CO)=\RD_c^\infty(N)\widetilde\otimes_{\mathrm{i}}\C[y_1^*,y_2^*,\dots,y_k^*]=\RD_c^\infty(N)\widehat\otimes_{\mathrm{i}}\C[y_1^*,y_2^*,\dots,y_k^*],\\
			&&\RC^{-\infty}(M;\CO)=\RC^{-\infty}(N)\widetilde\otimes_\pi \C[[y_1,y_2,\dots,y_k]]=\RC^{-\infty}(N)\widehat\otimes_\pi \C[[y_1,y_2,\dots,y_k]],\\
			%\CO_c(M)&=& \RC_c^\infty(N)\widehat\otimes_{\mathrm{i}}\C[[y_1,y_2,\dots,y_k]],\\
			&&\RD^{-\infty}(M;\CO)= \RD^{-\infty}(N)\widetilde\otimes_\pi \C[y_1^*,y_2^*,\dots,y_k^*]= \RD^{-\infty}(N)\widehat\otimes_\pi \C[y_1^*,y_2^*,\dots,y_k^*],\\
			&&\RD^{-\infty}_c(M;\CO)= \RD^{-\infty}_c(N)\widetilde\otimes_{\mathrm{i}} \C[y_1^*,y_2^*,\dots,y_k^*]=\RD^{-\infty}_c(N)\widehat\otimes_{\mathrm{i}} \C[y_1^*,y_2^*,\dots,y_k^*].
		\end{eqnarray*} 
	\end{thmd}

	In view of the above LCS identifications, we refer to  an element in $\RD_c^\infty(M;\CO)$
	(resp.\,$\RC^{-\infty}(M;\CO)$; $\RD^{-\infty}(M;\CO)$; $\RD^{-\infty}_c(M;\CO)$) as a compactly supported formal density (resp.\,a formal generalized function;
	a formal distribution; a compactly supported formal distribution) on $M$.

	Recall that the sheaf $\CF$ of $\CO$-modules is said to be locally free of finite rank if every point of $M$ has an open neighborhood $U$ such that $\CF|_U$ is free of finite rank.  
	By generalizing the tensor product characterization of (compactly supported) vector-valued distributions on $\BR^n$ (see \cite[Proposition 12]{Sc3} and \cite[Page 63]{Sc3}), we have the following result (see Propositions \ref{prop:GFMEten}, \ref{prop:charDFME}, and \ref{prop:charDFcME}).
	
	\begin{thmd} Assume that $\CF$ is locally free of finite rank. Then we have the  following LCS identifications:
		\begin{itemize}
			\item $\RC^{-\infty}(M;\CF;E)=\RC^{-\infty}(M;\CF)\widetilde\otimes_\pi E$, and 
			$\RC^{-\infty}(M;\CF;E)=\RC^{-\infty}(M;\CF)\widehat\otimes_\pi E$ provided that $E$ is complete; 
			\item 
			$\RD^{-\infty}(M;\CF;E)=\RD^{-\infty}(M;\CF)\widetilde\otimes_\pi E$, and 
			$\RD^{-\infty}(M;\CF;E)=\RD^{-\infty}(M;\CF)\widehat\otimes_\pi E$ provided that $E$ is complete; 
			\item $\RD^{-\infty}_c(M;\CF;E)=\RD^{-\infty}_c(M;\CF)\widetilde\otimes_\mathrm{i} E=\RD^{-\infty}_c(M;\CF)\widehat\otimes_\mathrm{i} E$ provided that $E$ is a barreled DF space.
		\end{itemize}  
	\end{thmd}
	
	It is a classical result of Schwartz that the space of compactly supported distributions on $\R^n$ coincides with the strong dual of the space of smooth functions on $\R^n$.
	More general, if $E$ is a DF space, then for every $n\in \BN$, 
	\be \label{eq:charcsEdN} 
	\mathrm D_c^{-\infty}(\R^n)\widetilde\otimes_\pi E=
	\mathrm D_c^{-\infty}(\R^n)\widehat\otimes_\pi E=
	\mathrm D_c^{-\infty}(\R^n;E)=\CL(\RC^\infty(\R^n),E)
	\ee
	as LCS (see \cite[Pages 62-63]{Sc3}). 
	The following result (see Propositions \ref{prop:DFcM=FM'} and \ref{prop:DFE=LME})  generalizes 
	these LCS identifications to the setting of formal manifolds.
	
	\begin{thmd} Assume that $\CF$ is locally free of finite rank.
		Then 
		\[\CF(M)'=\RD^{-\infty}_c(M)\]
		as LCS. Assume further that $M$ is secondly countable and $E$ is a DF space. Then 
		\[\mathrm{D}^{-\infty}_c(M;\CF)\wt\otimes_\pi E=\mathrm{D}^{-\infty}_c(M;\CF)\wh\otimes_\pi E=\mathrm{D}^{-\infty}_c(M;\CF;E)=\CL(\CF(M),E)\]
		as LCS.
	\end{thmd}

	\section{Compactly supported formal densities}\label{sec:formalden}
	
	In this section, we introduce a cosheaf 
	$\mathrm{D}^\infty_c(\CF)$ of $\CO$-modules over $M$, which generalizes the concept of cosheaves of compactly supported smooth densities on smooth manifolds.
	
	\subsection{The cosheaf $\mathrm{D}^\infty_c(\CF)$}
	Let $\CF_1$ and $\CF_2$ be two sheaves of $\CO$-modules.
	Recall from \cite[Section 3.2]{CSW} that there is a sheaf 
	\[
	\mathcal{D}\mathrm{iff}(\CF_1,\CF_2): U\mapsto \mathrm{Diff}(\CF_1|_U,\CF_2|_U)\quad (U\ \text{is an open subset of}\ M)
	\]
	on $M$ consisting of differential operators, as well as a cosheaf 
	\[
	\mathcal{D}\mathrm{iff}_c(\CF_1,\CF_2): U\mapsto \mathrm{Diff}_c(\CF_1|_U,\CF_2|_U)\quad (U\ \text{is an open subset of}\ M)
	\]
	on $M$ consisting of compactly supported differential operators. Note that for every  $f\in \CO(M)$ and $D\in \mathrm{Diff}(\CF_1,\CF_2)$, we define  two differential operators $f\circ D$ and $D\circ f$ in $\mathrm{Diff}(\CF_1,\CF_2)$  by setting
	\be\label{eq:defofDf} (f\circ D)_U:\ u\mapsto f|_U(D_U(u))\quad\text{and}\quad (D\circ f)_U:\ u\mapsto D_U (f|_Uu ),
	\ee 
	where $u\in \CF_1(U)$ and $U$ is an open subset of $M$.
	%Then we have two natural  $\CO(M)$-module structures on $\mathrm{Diff}(\CF_1,\CF_2)$  with\begin{eqnarray}\label{eq:leftactiononD}  &&\CO(M)\times  \mathrm{Diff}(\CF_1,\CF_2)\rightarrow  \mathrm{Diff}(\CF_1,\CF_2),\quad (f,D)\mapsto f\circ D,\\\label{eq:rightactiononD} &&\CO(M)\times  \mathrm{Diff}(\CF_1,\CF_2)\rightarrow  \mathrm{Diff}(\CF_1,\CF_2),\quad (f,D)\mapsto D\circ f.\end{eqnarray}Each of the  two actions mentioned above makes$\mathcal D\mathrm{iff}(\CF_1,\CF_2)$ a sheaf of $\CO$-modules.
	
	Recall  the reduction $\underline{M}=(M,\underline{\CO})$ 
	of $M$ introduced in \cite[Definition 1.2]{CSW}, which is a smooth manifold. 
	Write $\underline{\mathcal{D}}$ for the sheaf of all (complex-valued) smooth densities on $\underline{M}$ (\cf \cite[Chapter 16, Page 430]{L}).
	Via the reduction map $\underline{M}\rightarrow M$ (see \cite[(1.4)]{CSW}), $\underline{\mathcal{D}}$ is naturally a sheaf of $\CO$-modules.
	Let $U$ be an open subset of $M$. Then  we have a linear  map
	\begin{equation}\begin{split}\label{eq:defcomsuppden}
			\rho_U:\ \mathrm{Diff}_{c}(\CF|_U,\underline{\mathcal{D}}|_U)&\rightarrow \mathrm{Hom}_\C(\CF(U),\C),\\
			D&\mapsto \left(u\mapsto  \int_{U} D(u)\right).
	\end{split}\end{equation} 
	% for every open subset $U$ of $M$, where $\mathrm{Diff}_{c}(\CF|_U,\underline{\mathcal{D}}|_U)$ is the space of all compactly supported differential operators (see \cite[(3.10)]{CSW}).%Section \ref{sec:diff} for more details).
	For every $D\in \mathrm{Diff}_c(\CF|_U,\underline{\mathcal{D}}|_U)$,
	it is clear that the linear functional $\rho_U(D)$ on $\CF(U)$ is continuous, where $\CF(U)$ is equipped
	with the smooth topology (see \cite[Definition 4.1]{CSW}).
	Write
	\be \label{eq:Dcindual} \mathrm{D}^\infty_c(U;\CF):=\rho_U(\mathrm{Diff}_{c}(\CF|_U,\underline{\mathcal{D}}|_U))\subset (\CF(U))'.\ee
	Here $(\CF(U))'$ is the space of all continuous linear maps from $\CF(U)$ to $\BC$.
	\vspace{3mm}

	As stated in the Introduction, we make the following definition.
	
	\begin{dfn}\label{de:formalden}
		An element in $\mathrm{D}^\infty_c(M;\CO)$ 
		is called a compactly supported formal density on $M$.
	\end{dfn}

	When $M$ is a smooth manifold, we will show in  Proposition \ref{prop:charocd} that $\mathrm{D}^\infty_c(M;\CO)$ 
	is nothing but the space of compactly supported smooth densities on $M$.

	For an open subset $U$ of M, given a formal function $f\in \CO(M)$  supported in $U$ and $u\in \CF(U)$, we define 
	\[f u\in \CF(M)
	\]
	by requiring that
	\[
	(fu)|_U=f|_U\, u \quad \textrm{and}\quad (fu)|_{M\setminus \mathrm{supp}\,f}=0.
	\]

	\begin{lemd}\label{lem:defextmap} Let $V\subset U$ be open subsets  of $M$. Then
		the transpose \[{}^t\mathrm{res}_{V,U}: (\CF(V))'\rightarrow (\CF(U))'\]  of the restriction map $\mathrm{res}_{V,U}:\CF(U)\rightarrow \CF(V)$
		yields an injective linear map
		\be\label{eq:extmap}\mathrm{ext}_{U,V}:={}^t\mathrm{res}_{V,U}|_{\mathrm{D}^\infty_c(V;\CF)}:\quad \mathrm{D}^\infty_c(V;\CF)\rightarrow \mathrm{D}^\infty_c(U;\CF).\ee
	\end{lemd}
	\begin{proof} Let $D\in \mathrm{Diff}_{c}(\CF|_V,\underline{\mathcal{D}}|_V)$. Write  \[\mathrm{ext}_{U,V}: \mathrm{Diff}_{c}(\CF|_V,\underline{\mathcal{D}}|_V) \rightarrow \mathrm{Diff}_{c}(\CF|_U,\underline{\mathcal{D}}|_U)\] 
		for the extension map in  $\mathcal{D}\mathrm{iff}_c(\CF,\underline{\mathcal{D}})$.
		Then 
		\begin{eqnarray*}\la\rho_U(\mathrm{ext}_{U,V}(D)),u\ra
			&=&\int_U \mathrm{ext}_{U,V}(D)(u)
			=\int_V D(u|_V)\\
			&=&\la \rho_V(D), u|_V\ra \\ &=& \la {}^t\mathrm{res}_{V,U}(\rho_V(D)),u \ra
		\end{eqnarray*} for every $u\in \CF(U)$. It follows that 
		the map \eqref{eq:extmap} is well-defined and  satisfies the relation
		\be\label{eq:cosheafhom} \rho_U(\mathrm{ext}_{U,V}(D))=\mathrm{ext}_{U,V}(\rho_V(D)).
		\ee
		
		To see the injectivity of \eqref{eq:extmap},
		assume  that $\mathrm{ext}_{U,V}(\rho_V(D))=0$.
		Take an open neighborhood $W$  of $\mathrm{supp}\,D$ in $V$ such that
		$\overline{W}\subset  V$, where $\overline{W}$ is the closure of $W$ in $U$.
		Let $v$ be a section in $\CF(V)$, and let $f$ be a  formal  function on $U$ such that \[f|_{\overline{W}}=1 \quad \text{and}\quad  \mathrm{supp}\,f\subset V.\]
		Then we have a section $fv\in \CF(U)$.
		By  \eqref{eq:cosheafhom}, we have that 
		\begin{eqnarray*}\la\rho_V(D),v\ra&=&\int_V D(v)=\int_W D(v)|_W \\&=&
			\int_U(\mathrm{ext}_{U,V}(D))(fv)
			=\la\rho_U(\mathrm{ext}_{U,V}(D)),fv\ra\\
			&=&\la \mathrm{ext}_{U,V}(\rho_V(D)),fv\ra=0.
		\end{eqnarray*}  This finishes the proof.
	\end{proof}

	With \eqref{eq:extmap} as the extension maps,
	the assignment
	\be\label{eq:cosheafd} U\mapsto \mathrm{D}^\infty_c(U; \CF)\quad(\text{$U$ is an open subset of $M$})\ee
	forms a precosheaf of complex vector spaces on $M$ (see \cite[Chapter V, Definition 1.1]{Br2}), which we denote as $\mathrm{D}^\infty_c(\CF)$. Recall that a precosheaf $\CG$ over $M$ is called flabby if the extension map \[\mathrm{ext}_{U,V}: \CG(V)\rightarrow \CG(U)\] is injective for all open sets $V\subset U$ of $M$ (see \cite[Chapter V, Definition 1.2]{Br2}).
	Then Lemma \ref{lem:defextmap} implies that the precosheaf  $\mathrm{D}^\infty_c(\CF)$ is flabby.

	From \eqref{eq:cosheafhom}, one  concludes that the family 
	\be\label{eq:rho}
	\rho=\{\rho_U\}_{\text{$U$ is an open subset of $M$}}:\quad\mathcal{D}\mathrm{iff}_c(\CF,\underline{\mathcal{D}})
	\rightarrow \mathrm{D}^\infty_c(\CF)\ee is a  homomorphism of precosheaves.
	The following result shows that $\rho$ is in fact a homomorphism of cosheaves.

	\begin{prpd}\label{prop:F'smisacosheaf} The precosheaf  $\mathrm{D}^\infty_c(\CF)$ is a  cosheaf.
	\end{prpd}
	
	The rest of this subsection is devoted to a proof of Proposition \ref{prop:F'smisacosheaf}.
	Let $U_1$ and $U_2$ be two open subsets of $M$, and set $U:=U_1\cup U_2$.
	We consider the  map
	\[\begin{array}{rcl}
		\phi:\,\mathrm{D}^\infty_c(U_1;\CF)\oplus \mathrm{D}^\infty_c(U_2;\CF) &\rightarrow& \mathrm{D}^\infty_c(U;\CF),\\(\eta_1,\eta_2)
		&\mapsto&\mathrm{ext}_{U,U_1}(\eta_1)+
		\mathrm{ext}_{U,U_2}(\eta_2).
	\end{array}
	\]
	Fix an element
	\be\label{eq:defw} \eta =(\eta_1,\eta_2)\in \ker\phi,\ee
	and  choose a differential operator $D_i\in \mathrm{Diff}_c(\CF|_{U_i},\underline{\mathcal{D}}|_{U_i})$ such
	that $\eta_i=\rho_{U_i}(D_i)$ for each $i=1,2$.
	Set $K:=\mathrm{supp}\, D_1\cap \mathrm{supp}\, D_2$.
	
	\begin{lemd}\label{lem:MV1}
		Let $i\in\{1,2\}$. If $u_i\in \CF(U_i)$  satisfies that $\mathrm{supp}\,u_i\cap K=\emptyset$, then $\la \eta_i,u_i\ra=0$.
	\end{lemd}
	\begin{proof}
		Without loss of generality, suppose that $i=1$. Set \[K_1:=\mathrm{supp}\,u_1\cap \mathrm{supp}\,D_1\quad\text{and}\quad  K_2:=\mathrm{supp}\,u_1\cap \mathrm{supp}\,D_2.\]
		Since $K_1$ is compact in $U_1$ and $K_1\cap K_2\subset \mathrm{supp}\,u_1\cap K=\emptyset$, there are
		open neighborhoods $U'$ and $U''$ of $K_1$ in $U_1$ such that \[\overline{U'}
		\subset U''\quad \text{and} \quad U''\cap K_2=\emptyset.\] Here $\overline{U'}$ is the closure of $U'$ in $U$.
		Let $f$ be a
		formal function  on $U$ such that \[f|_{\overline{U'}}=1 \quad \text{and} \quad f|_{U\setminus U''}=0.\]
		This yields an element $fu_1\in\CF(U)$.
		Then we have that 
		\begin{eqnarray*}
			&&\la \eta_1, u_1\ra=\la\rho_{U_1}(D_1),u_1\ra\\
			&=& \int_{U_1} D_1(u_1)
			=\int_{U'} (D_1(u_1))|_{U'}\,\,\,\quad \quad \quad  \quad \quad (\text{as $\mathrm{supp}\,D_1(u_1) \subset K_1\subset U'$})\\
			&=&
			\int_U (\mathrm{ext}_{U,U_1}(D_1))(fu_1)\, \,\,\, \quad\quad \quad \quad\quad\quad \quad\quad (\text{as  $(fu_1)|_{U'}=u_1|_{U'}$})\\
			&=& \la\rho_{U}(\mathrm{ext}_{U,U_1}(D_1)),fu_1\ra\quad \\
			&=&
			\la \rho_{U}(\mathrm{ext}_{U,U_1}(D_1)+\mathrm{ext}_{U,U_2}(D_2)), fu_1\ra \quad \,\quad(\text{as $\mathrm{supp}\,(fu_1)\cap \mathrm{supp}\,D_2=\emptyset$})\\
			&=& \la \phi((\rho_U(D_1),\rho_U(D_2))),fu_1\ra\quad\quad \quad \quad \quad \quad \,\,(\text{by \eqref{eq:cosheafhom}})\\
			&=&\la \phi(\eta),fu_1\ra=0,
		\end{eqnarray*}
		which proves the lemma. 
	\end{proof}
	
	Set $V:=U_1\cap U_2$. For each $i=1,2$, take an
	open subset $V_i$ of $U_i$ such that
	\[
	K\subset V_i\subset \overline{V_i}\subset V\subset U_i\quad (\overline{V_i}\ \text{is the closure of}\ V_i\ \text{in}\ U_i).\]
	Let $\{g_i,g_i'\}$ be a partition of unity on $U_i$ subordinate to the open cover $\{V, U_i\setminus \overline{V_i}\}$.
	It is clear that $D_i\circ g_i \in \mathrm{Diff}_c(\CF|_{U_i},\underline {\mathcal{D}}|_{U_i} )$, where $D_i\circ g_i$ is defined as in \eqref{eq:defofDf}. 
	Set \[D_i':=(D_i\circ g_i)|_V \in \mathrm{Diff}_c(\CF|_V,\underline{\mathcal{D}}|_V)\quad \text{and}\quad  \eta'_i:=\rho_{V}(D_i')\in \mathrm{D}^\infty_c(V;\CF).\] 
	
	\begin{lemd}\label{lem:MV2} We have   that $\eta'_1+\eta'_2=0$, and $\eta_i=\mathrm{ext}_{U_i,V}(\eta'_i)$ for $i=1,2$. 
	\end{lemd}
	\begin{proof} For every $u_i'\in \CF(U_i)$, %we have element $g_iu'_i$ as in \eqref{eq:deffu}, 
		since $$\mathrm{supp}\,(g_i'u_i')\cap K=\emptyset, $$ it
		follows from Lemma \ref{lem:MV1} that $\la \eta_i, g_i'u_i'\ra=0$.
		Together with the fact that $\mathrm{supp}\,g_i\subset V$, we have that
		\begin{eqnarray*}
			\la \eta_i, u_i'\ra &=&\la \eta_i, (g_i+g_i')u_i'\ra=\la \eta_i,g_iu_i'\ra\\&=& \int_{U_i} D_i(g_iu_i')
			= \int_V D_i'(u_i'|_V)
			\\ &=&\la \eta'_i,u_i'|_V\ra=\la \mathrm{ext}_{U_i,V}(\eta'_i), u_i'\ra.
		\end{eqnarray*}
		This implies  that $\eta_i=\mathrm{ext}_{U_i,V}(\eta'_i)$. Moreover, we have the following equalities:
		\begin{eqnarray*}
			\mathrm{ext}_{U,V}(\eta'_1+\eta'_2)
			&=&\mathrm{ext}_{U,U_1}\mathrm{ext}_{U_1,V}(\eta'_1)
			+\mathrm{ext}_{U,U_2}\mathrm{ext}_{U_2,V}(\eta'_2)\\
			&=&\mathrm{ext}_{U,U_1}(\eta_1)+\mathrm{ext}_{U,U_2}(\eta_2)
			\\ &=&\phi(\eta)=0.
		\end{eqnarray*}
		Then the flabby property of $\mathrm{D}^\infty_c(\CF)$ forces that $\eta'_1+\eta'_2=0$, as required.
	\end{proof}
	
	\begin{lemd}\label{lem:MV3}  The
		Mayer-Vietoris sequence
		\bee
		\mathrm{D}^\infty_c(V;\CF)\xrightarrow{\psi}\mathrm{D}^\infty_c(U_1;\CF)
		\oplus \mathrm{D}^\infty_c(U_2; \CF)\xrightarrow{\phi} \mathrm{D}^\infty_c(U;\CF)\rightarrow 0
		\eee
		in $\mathrm{D}^\infty_c(\CF)$ is exact, where $\psi:=(\mathrm{ext}_{U_1,V},-\mathrm{ext}_{U_2,V})$.
	\end{lemd}
	\begin{proof} 
		As $\mathcal{D}\mathrm{iff}_c(\CF,\underline{\mathcal{D}})$ is a cosheaf, the Mayer-Vietoris sequence
		\[
		\mathrm{Diff}_c(\CF|_V,\underline{\mathcal{D}}|_V)\rightarrow\mathrm{Diff}_c(\CF|_{U_1},\underline{\mathcal{D}}|_{U_1})
		\oplus \mathrm{Diff}_c(\CF|_{U_2},\underline{\mathcal{D}}|_{U_2})\rightarrow \mathrm{Diff}_c(\CF|_U,\underline{\mathcal{D}}|_U)\rightarrow 0
		\] in $\mathcal{D}\mathrm{iff}_c(\CF,\underline{\mathcal{D}})$ is exact (see \cite[Chapter VI, Proposition 1.4]{Br2}). Using 
		the homomorphism $\rho$, it follows 
		that $\phi$ is surjective and $\phi\circ \psi=0$. On the other hand, we have  that $\ker\phi\subset \mathrm{im}\, \psi$ by Lemma \ref{lem:MV2}. This finishes the proof. 
	\end{proof}
	
	\begin{lemd}\label{lem:MV4} Let $\{U_\alpha\}_{\alpha\in \Gamma}$ be a family   of open subsets of $M$ directed upwards by inclusions. Then
		the canonical map
		\be
		\label{eq:fsmisacosheaf1}
		\varinjlim_{\alpha\in \Gamma} \mathrm{D}^\infty_c(U_{\alpha}; \CF)\rightarrow\mathrm{D}^\infty_c(\cup_{\alpha\in \Gamma}U_{\alpha}; \CF)
		\ee
		is a linear isomorphism.
	\end{lemd}
	\begin{proof} 
		As $\mathcal{D}\mathrm{iff}_c(\CF,\underline{\mathcal{D}})$ is a cosheaf, the canonical map \[\varinjlim_{\alpha\in \Gamma}\mathrm{Diff}_c(\CF|_{U_{\alpha}},\underline{\mathcal{D}}|_{U_{\alpha}})\rightarrow\mathrm{Diff}_c(\CF|_{\cup_{\alpha\in \Gamma}U_{\alpha}},\underline{\mathcal{D}}|_{\cup_{\alpha\in \Gamma}U_{\alpha}})\] is a linear isomorphism (see \cite[Chapter VI, Proposition 1.4]{Br2}). 
		Using the homomorphism $\rho$, it then follows 
		that the map \eqref{eq:fsmisacosheaf1} is surjective. 
		On the other hand, the injectivity of \eqref{eq:fsmisacosheaf1} follows from the flabby property of $\mathrm{D}^\infty_c(\CF)$, as required.
	\end{proof}

	Finally,  by \cite[Chapter VI, Proposition 1.4]{Br2}, Proposition \ref{prop:F'smisacosheaf} is implied by Lemmas \ref{lem:MV3} and \ref{lem:MV4}.

	\subsection{The topology on $\mathrm{D}^\infty_c(M;\CF)$}
	In this subsection, we define an 
	inductive limit topology on   $\mathrm{D}^\infty_c(M; \CF)$.
	
	We first endow a locally convex topology  on the space $\mathrm{Diff}_{\mathrm{fin}}(\CF,\underline{\mathcal{D}})$ of finite order differential
	operators (see \cite[Definition 3.6]{CSW}).
	Let  $X\in \mathrm{Diff}_c(\underline{\mathcal{D}},\underline{\CO})$ and $u\in \CF(M)$.
	For every $D\in \mathrm{Diff}_{\mathrm{fin}}(\CF,\underline{\mathcal{D}})$, define
	\[
	\abs{D}_{X,u}:=\sup_{a\in M}|((X\circ D)(u))(a)|.
	\]
	Then $\abs{\,\cdot\,}_{X,u}$ is a seminorm on $\mathrm{Diff}_{\mathrm{fin}}(\CF,\underline{\mathcal{D}})$, and
	$\mathrm{Diff}_{\mathrm{fin}}(\CF,\underline{\mathcal{D}})$ becomes an LCS with the topology defined by all such seminorms.

	Equip $\underline{\mathcal{D}}(M)$ with the smooth topology.
	Then we have the following straightforward result.

	\begin{lemd}\label{lem:convergechar} A net $\{D_i\}_{i\in I}$  in $\mathrm{Diff}_{\mathrm{fin}}(\CF,\underline{\mathcal{D}})$
		converges to $0$ if and only if $\{D_i(u)\}_{i\in I}$ converges to $0$ in $\underline{\mathcal{D}}(M)$ for all $u\in \CF(M)$. 
		Moreover, the topology
		on $\mathrm{Diff}_{\mathrm{fin}}(\CF,\underline{\mathcal{D}})$ is Hausdorff.
	\end{lemd}
	
	There is a  directed set $(\mathcal{C}(M),\preceq)$, where
	\be\label{eq:C(M)} \mathcal{C}(M):=\{(K,r)\mid \text{$K$ is a compact subset of $M$ and $r\in \BN$}\},\ee
	and for two pairs $(K,r),(K',r')\in \mathcal{C}(M)$, 
	\[\text{$(K,r)\preceq(K',r')$ if and only if $K\subset K'$ and $r\le r'$.}\]
	For  $(K,r)\in \mathcal{C}(M)$, define
	\[\mathrm{Diff}_{K,r}(\CF,\underline{\mathcal{D}}):=\mathrm{Diff}_K(\CF,\underline{\mathcal{D}})\cap
	\mathrm{Diff}_r(\CF,\underline{\mathcal{D}}),\]
	where $\mathrm{Diff}_K(\CF,\underline{\mathcal{D}})$ denotes the subspace of $\mathrm{Diff}_c(\CF,\underline{\mathcal{D}})$  consisting of all the
	differential operators supported in $K$, and $\mathrm{Diff}_r(\CF,\underline{\mathcal{D}})$ denotes the space of differential operators with order $\le r$ (see \cite[Definition 3.5]{CSW}).
	Equip $\mathrm{Diff}_{K,r}(\CF,\underline{\mathcal{D}})$ with the subspace topology of $\mathrm{Diff}_{\mathrm{fin}}(\CF,\underline{\mathcal{D}})$.

	Write  
	\[\mathrm{D}^\infty_{K,r}(M;\CF):=\rho_M(\mathrm{Diff}_{K,r}(\CF,\underline{\mathcal{D}}))\subset \mathrm{D}^\infty_{c}(M;\CF),\]
	which is endowed with the quotient topology of $\mathrm{Diff}_{K,r}(\CF,\underline{\mathcal{D}})$.
	Then $\mathrm{D}^\infty_{K,r}(M;\CF)$ becomes an LCS.
	Equip \be \label{eq:D_ctop} \mathrm{D}^\infty_c(M;\CF)=\varinjlim_{(K,r)\in \mathcal{C}(M)}
	\mathrm{D}^\infty_{K,r}(M;\CF)\ee with the inductive limit topology.
	%We equip $\mathrm{D}^\infty_c(M;\CF)$ with the inductive limit topology  of  the  directed system \be \label{eq:D_c} \{\mathrm{D}^\infty_{K,r}(M;\CF)\}_{(K,r)\in \mathcal{C}(M)}\ee of LCS.

	%\begin{proof} For two pairs $(K,r)$, $(K',r')\in \mathcal{C}(M)$ with $(K,r)\preceq (K',r')$, there is a commutative diagram \[\begin{CD}\mathrm{Diff}_{K,r}(\CF,\underline{\mathcal{D}})@>>> \mathrm{Diff}_{K',r'}(\CF,\underline{\mathcal{D}})\\@VVV @VVV\\ \mathrm{D}^\infty_{K,r}(M;\CF) @> >>\mathrm{D}^\infty_{K',r'}(M;\CF),\end{CD} \] where vertical arrows are quotient maps, and the top horizontal arrow is a topological embedding. It is clear that the bottom horizontal arrow is a topological embedding. The lemma then follows.\end{proof}
	
	\begin{lemd}\label{lem:Hausdorffandregular} 
		For every $(K,r)\in \mathcal{C}(M)$, the LCS $\mathrm{D}^\infty_{K,r}(M;\CF)$ is Hausdorff, and the LCS $\mathrm{D}^\infty_{c}(M;\CF)$ is also Hausdorff provided that the inductive limit 
		\[
		\varinjlim_{(K,r)\in \mathcal{C}(M)}
		\mathrm{D}^\infty_{K,r}(M;\CF)
		\] is strict (see Definition \ref{de:regualrlim}).
	\end{lemd}
	\begin{proof}
		By Lemma \ref{lem:convergechar},  it is easy to see that the subspace $\ker \rho_M\cap \mathrm{Diff}_{K,r}(\CF,\underline{\mathcal{D}})$
		is closed in $\mathrm{Diff}_{K,r}(\CF,\underline{\mathcal{D}})$ for every $(K,r)\in \mathcal{C}(M)$.
		This implies that the LCS $\mathrm{D}^\infty_{K,r}(M;\CF)$ is Hausdorff.
		
		By definition, we have that 
		\be \label{eq:DC=directsumind}
		\mathrm{D}^\infty_c(M;\CF)
		=\varinjlim_{(K,r)\in \mathcal{C}(M)}
		\mathrm{D}^\infty_{K,r}(M;\CF)
		=\bigoplus_{Z\in \pi_0(M)}
		\varinjlim_{(K,r)\in \mathcal{C}(Z)}\mathrm{D}^\infty_{K,r}(Z;\CF)
		\ee as LCS.
		Then by Lemma \ref{lem:basicsonindlim} (b), the LCS $\mathrm{D}^\infty_c(M;\CF)$ is  Hausdorff.
	\end{proof}

	\begin{lemd}\label{lem:conextD}  Let $U$ be an open subset of $M$. Then the extension map
		\be \label{eq:conextF}\mathrm{ext}_{M,U}:\ \mathrm{D}^\infty_{c}(U;\CF)\rightarrow \mathrm{D}^\infty_{c}(M;\CF)
		\ee
		is continuous.
	\end{lemd}
	\begin{proof}
		Let $(K,r)\in \mathcal{C}(U)$. Note that the extension map $\mathrm{ext}_{M,U}$ on $\mathrm{Diff}_{c}(\CF|_U,\underline{\mathcal{D}}|_U)$ induces a linear map
		\be \label{eq:condiffkr2}
		\mathrm{Diff}_{K,r}(\CF|_U,\underline{\mathcal{D}}|_U)\rightarrow \mathrm{Diff}_{K,r}(\CF,\underline{\mathcal{D}}).
		\ee
		Then by taking the restriction, \eqref{eq:conextF} yields a linear map \be
		\label{eq:conextFKr}\mathrm{D}^\infty_{K,r}(U;\CF)\rightarrow \mathrm{D}^\infty_{K,r}(M;\CF).
		\ee

		Let $f$ be a formal function on $M$ such that $\mathrm{supp}\,f\subset U$ and $f|_K=1$.
		Note that
		\[
		\abs{\mathrm{ext}_{M,U}(D)}_{X,u}
		=\abs{D}_{(f\circ X)|_U,u|_U}\quad \text{for all $D\in  \mathrm{Diff}_{K,r}(\CF|_U,\underline{\mathcal{D}}|_U)$},
		\]
		where $X\in \mathrm{Diff}_c(\underline{\mathcal{D}},\underline{\mathcal{O}})$, $u\in \CF(M)$, and $f\circ X$ is as in \eqref{eq:defofDf}.
		Thus the map \eqref{eq:condiffkr2} is continuous, and so is the map \eqref{eq:conextFKr}.
		Therefore,  the composition of  
		\[\mathrm{D}^\infty_{K,r}(U;\CF)\rightarrow \mathrm{D}^\infty_{K,r}(M;\CF)\rightarrow \mathrm{D}^\infty_c(M;\CF)
		\]
		is continuous as well.
		The lemma then follows.
	\end{proof}
	
	Recall the definition of a cosheaf of $\mathcal{O}$-modules on $M$ from the Introduction. 
	In what follows, we give a canonical $\CO$-module structure on the cosheaf $\mathrm{D}^\infty_{c}(\CF)$. 
	For every open subset $U$ of $M$, we have a natural $\CO(U)$-structure on the space $(\CF(U))'$ with 
	\be\label{eq:actiononF'}
	\CO(U)\times (\CF(U))' \rightarrow(\CF(U))',\quad  (f,\eta)\mapsto\eta\circ f,   
	\ee
	where $\eta\circ f$ is the map 
	\be\label{eq:deffomega} \eta\circ f: \, \CF(U)\rightarrow \BC, \quad u\mapsto \eta(fu).  \ee
	\begin{lemd}\label{lem:OmoduleonFsm} Let $\eta\in \mathrm{D}^\infty_{c}(U;\CF)$ and $f\in \CO(U)$.
		Then the linear functional $\eta\circ f$
		lies in $\mathrm{D}^\infty_{c}(U;\CF)$.
	\end{lemd}
	\begin{proof} Choose a differential operator $D\in \mathrm{Diff}_c(\CF|_U,\underline{\mathcal{D}}|_U)$ such that
		$\rho_U(D)=\eta$.
		The lemma follows from the fact that $\eta\circ f=\rho_U(D\circ f)$. Here $D\circ f\in \mathrm{Diff}_c(\CF|_U,\underline{\mathcal{D}}|_U)$ is defined as in \eqref{eq:defofDf}.
	\end{proof}
	
	By  
	Lemma \ref{lem:OmoduleonFsm} and \eqref{eq:actiononF'}, there is an $\CO$-module structure on the cosheaf $\mathrm{D}^\infty_{c}(\CF)$
	transferred from that on $\CF$.
	On the other hand,  the  $\CO$-module structure  given by ($U$ is an open subset of $M$)
	\[
	\CO(U)\times \mathrm{Diff}(\CF|_{U},\underline{\mathcal{D}}|_{U})\rightarrow \mathrm{Diff}(\CF|_{U},\underline{\mathcal{D}}|_{U}),\quad 
	(f,D)\mapsto D\circ f\]
	on $\mathcal{D}\mathrm{iff}(\CF,\underline{\mathcal{D}})$ induces
	an  $\CO$-module structure on  $\mathcal{D}\mathrm{iff}_c(\CF,\underline{\mathcal{D}})$.
	It is obvious that $\rho$ (see \eqref{eq:rho})  is an $\CO$-module homomorphism of cosheaves.

	\begin{lemd}\label{lem:rescosheafonopensubset}
		Let $(K,r)\in \mathcal{C}(M)$, $U$ an open subset of $M$, and $f$ a formal function on $M$ supported in $U$.
		Then for every $\eta\in\mathrm{D}^\infty_{K,r}(M;\CF)$, there is a unique element \be\label{eq:etaf|U} (\eta\circ f)|_U\in \mathrm{D}^\infty_{\mathrm{supp}f\cap K,r}(U;\CF)\ee
		such that
		$\mathrm{ext}_{M,U}((\eta\circ f)|_U)=\eta\circ f$. Furthermore, the linear map
		\be \label{eq:conmfCF} m_f:\ \mathrm{D}^\infty_{c}(M;\CF)\rightarrow \mathrm{D}^\infty_{c}(U;\CF),\quad \eta\mapsto (\eta\circ f)|_U
		\ee
		is continuous.
	\end{lemd}
	\begin{proof} The uniqueness follows from the flabby property of $\mathrm{D}^{\infty}_{c}(\CF)$ (see Lemma \ref{lem:defextmap}).
		Choose  $D\in \mathrm{Diff}_{K,r}(\CF,\underline{\mathcal{D}})$ such that $\eta=\rho_M(D)$.
		Note that \[
		(D\circ  f)|_U\in  \mathrm{Diff}_{\mathrm{supp}f\cap K,r}(\CF|_U,\underline{\mathcal{D}}|_U),\]
		where $D\circ f$ is defined as in \eqref{eq:defofDf}.
		By using the fact that $\rho$ is an $\CO$-module homomorphism of cosheaves,
		the first assertion then follows by setting \be
		\label{eq:defetaf}(\eta\circ f)|_U:=\rho_U((D\circ  f)|_U).\ee

		For the second one, we first prove that the map
		\be \label{eq:mfconpre1}
		m_f:\ \mathrm{Diff}_{K,r}(\CF,\underline{\mathcal{D}})\rightarrow \mathrm{Diff}_{K\cap \mathrm{supp}\,f,r}(\CF|_U,\underline{\mathcal{D}}|_U),\quad D\mapsto (D \circ f)|_U\ee
		is continuous.
		Let $X\in \mathrm{Diff}_c(\underline{\mathcal{D}}|_U,
		\underline{\CO}|_U)$, $u\in \CF(U)$, and let $g$ be a formal function  on $M$ such that \[\mathrm{supp}\,g\subset U \quad \text{and} \quad g|_{\mathrm{supp}\,f}=1.\]
		Then we have an element 
		$u':=fgu\in \CF(M)$, as well as an element  \[X':= \mathrm{ext}_{M,U}(X)\in \mathrm{Diff}_c(\underline{\mathcal{D}},\underline{\CO}).\]
		For every $D\in \mathrm{Diff}_{K,r}(\CF,\underline{\mathcal{D}})$, we have that 
		\begin{eqnarray*}
			(X'\circ D)(u')&=&X'(D(fgu))=X'((D\circ f)(gu))\\
			&=&(\mathrm{ext}_{M,U}(X))(\mathrm{ext}_{M,U}((D\circ f)|_U(u))\\
			&=&\mathrm{ext}_{M,U} ((X\circ (D\circ f)|_U)(u)).
		\end{eqnarray*} 
		Consequently, we obtain that
		\begin{eqnarray*}
			|D|_{X',u'}&=&\sup_{a\in K}|((X'\circ D)(u'))(a)|\\
			&=&\sup_{a\in K\cap \mathrm{supp}f}|((X\circ (D\circ f)|_U)(u))(a)|\\
			&=&|m_f(D)|_{X,u}. 
		\end{eqnarray*}
		This implies that \eqref{eq:mfconpre1} is continuous. 
		
		From \eqref{eq:defetaf} and the continuity of \eqref{eq:mfconpre1}, one concludes that the map 
		\be\label{eq:defmfDKr} m_f:\ \mathrm{D}^\infty_{K,r}(M;\CF)\rightarrow \mathrm{D}^\infty_{\mathrm{supp}f\cap K,r}(U;\CF),\quad \eta\mapsto (\eta\circ f)|_U
		\ee
		is continuous, and so is the composition of 
		\[
		\mathrm{D}^\infty_{K,r}(M;\CF)\xrightarrow{m_f} \mathrm{D}^\infty_{\mathrm{supp}f\cap K,r}(U;\CF)\rightarrow 
		\mathrm{D}^\infty_{c}(U;\CF).
		\]
		This implies that
		\eqref{eq:conmfCF} is continuous, and we complete the proof.
	\end{proof}
	
	We refer to Appendixes \ref{appendixA4} and \ref{appendixB3} for the following usual notations: 
	\[
	\CL_b(E_1,E_2),\quad\CL_c(E_1,E_2),\quad (E_1)_b',\quad (E_1)_c'\quad \text{($E_1$ and $E_2$ are LCS)},
	\]
	which are used to denote various LCS of continuous linear maps. 
	Let $f: E_1\rightarrow E_2$
	be a continuous linear map. Note that the transposes
	\be\label{eq:trancont}{}^tf:  \CL_b(E_2,E_3)\rightarrow \CL_b(E_1,E_3)\ \text{and}\ {}^tf:  \CL_c(E_2,E_3)\rightarrow \CL_c(E_1,E_3) 
	\ \text{are continuous},\ee 
	where $E_3$ is another LCS (\cf \cite[Corollary of Proposition 19.5]{Tr}). Together with  Lemma \ref{lem:rescosheafonopensubset}, 
	the following result is clear.
	\begin{cord} Let $U$ be an open subset of $M$, and let $f$ be  a formal function on $M$ supported in $U$.
		Then the linear map
		\be\begin{split}\label{eq:defTfFUsmtoFMsm} T_f:\ \CL_c(\mathrm{D}^\infty_{c}(U;\CF),\mathrm{D}^\infty_{c}(U;\CF))
			&\rightarrow \CL_c(\mathrm{D}^\infty_{c}(M;\CF),\mathrm{D}^\infty_{c}(M;\CF)),\\
			\phi&\mapsto \mathrm{ext}_{M,U}\circ \phi\circ m_f
		\end{split}\ee
		is continuous. Here $m_f$ is defined as in \eqref{eq:conmfCF}.
	\end{cord}
	
	Let $(K,r)\in \mathcal{C}(M)$.  Take a finite family $\{U_i\}_{i=1}^s$ of open subsets   of $M$ that covers  $K$, and
	take a family $\{f_i\}_{i=1}^s$ of formal functions on $M$ such that 
	\[\text{$\mathrm{supp}\,f_i\subset U_i$ for all $i=1,2,\dots,s$,\quad and\quad
		$\sum_{i=1}^s f_i|_K=1$.}\]
	Recall from the proof of Lemma \ref{lem:conextD} that the linear map \eqref{eq:conextFKr} is continuous. Then there is a continuous linear map
	\be\begin{split}
		\label{eq:compacthomonto2}
		\bigoplus_{i=1}^s \mathrm{D}^\infty_{K_i,r}(U_i;\CF)&\rightarrow \mathrm{D}^\infty_{K,r}(M;\CF),\\
		(\eta_1,\eta_2,\dots,\eta_s)&\mapsto 
		\sum_{i=1}^s \mathrm{ext}_{M,U_i}(\eta_i),
	\end{split}\ee
	where $K_i:=\mathrm{supp}\,f_i\cap K$ for all $i=1,2,\dots,s$. On the other hand, recall also from the proof of Lemma \ref{lem:rescosheafonopensubset} that the linear map \eqref{eq:defmfDKr} is continuous. Then there is  another continuous linear map  
	\be\begin{split}
		\label{eq:rinvcompacthomonto2}
		\mathrm{D}^\infty_{K,r}(M;\CF)&\rightarrow 
		\bigoplus_{i=1}^s \mathrm{D}^\infty_{K_i,r}(U_i;\CF),\\
		\eta&\mapsto
		((\eta\circ f_1)|_{U_1},(\eta\circ f_2)|_{U_2},\dots,(\eta\circ f_s)|_{U_s}).
	\end{split}\ee
	
	The following result is straightforward. 
	
	\begin{lemd}\label{lem:homextDKr} The continuous linear map \eqref{eq:rinvcompacthomonto2} is a right inverse of
		\eqref{eq:compacthomonto2}. 
		In particular, the continuous linear map  \eqref{eq:compacthomonto2} is open and surjective, and \eqref{eq:rinvcompacthomonto2} is a topological embedding.
	\end{lemd}

	Furthermore, we have the following result. 
	
	\begin{prpd}\label{prop:extconD} For every  open cover  $\{U_\alpha\}_{\alpha\in \Gamma}$  of $M$,
		the canonical map
		\be\label{eq:openextDiff}
		\bigoplus_{\alpha\in \Gamma}
		\mathrm{D}^\infty_c(U_{\alpha};\CF)
		\rightarrow\mathrm{D}^\infty_c(M;\CF),\quad \{\eta_\alpha\}_{\alpha\in \Gamma}\mapsto \sum_{\alpha\in \Gamma}
		\mathrm{ext}_{M,U_\alpha}(\eta_\alpha)
		\ee
		has a continuous linear right inverse. In particular, the linear map 
		\eqref{eq:openextDiff}
		is continuous, open and surjective.
	\end{prpd}
	\begin{proof} Take a  partition of unity $\{f_\alpha\}_{\alpha\in \Gamma}$ on $M$ subordinate to  the cover $\{U_\alpha\}_{\alpha\in \Gamma}$. 
		It is easy to see that the linear map 
		\be\label{eq:DcMFtoDUF} \mathrm{D}^\infty_c(M;\CF) \rightarrow\bigoplus_{\alpha\in \Gamma}
		\mathrm{D}^\infty_c(U_{\alpha};\CF),\quad \eta\mapsto \{(\eta\circ f_\alpha)|_{U_\alpha}\}_{\alpha\in \Gamma}
		\ee
		is a  right inverse of \eqref{eq:openextDiff}.
		Thus, it suffices to prove that \eqref{eq:DcMFtoDUF} is continuous.

		For every $(K,r)\in \mathcal{C}(M)$ and $\alpha\in \Gamma$, write
		\[
		K_{\alpha}:=\mathrm{supp}\,f_\alpha\cap K\quad\text{and}\quad \Gamma_K:=\{\alpha\in \Gamma\mid K_{\alpha}\ne \emptyset\}.
		\] 
		As in \eqref{eq:rinvcompacthomonto2}, there is a continuous linear map 
		\[
		\mathrm{D}^\infty_{K,r}(M;\CF)\rightarrow \bigoplus_{\alpha\in \Gamma_K} \mathrm{D}^\infty_{K_{\alpha}, r}(U_{\alpha};\CF),\quad \eta\mapsto \{(\eta\circ f_\alpha)|_{U_\alpha}\}_{\alpha\in \Gamma_K}.\]
		This implies that the restriction of \eqref{eq:DcMFtoDUF} on $\mathrm{D}^\infty_{K,r}(M;\CF)$ is continuous, and so is the map \eqref{eq:DcMFtoDUF}, as required.

	\end{proof}

	The following result is a by-product of Proposition \ref{prop:extconD}.
	
	\begin{cord}\label{cor:Dc=directsum}
		There is an identification 
		\be\label{eq:Dc=directsum}
		\mathrm{D}^\infty_c(M;\CF)=\bigoplus_{Z\in \pi_0(M)}
		\mathrm{D}^\infty_c(Z;\CF)
		\ee  as LCS.
	\end{cord}

	\subsection{Compactly supported formal densities}
	Let $N$ be a smooth manifold and $k\in \BN$. 
	In this subsection, we suppose that  $(M,\CO)=(N,\CO_N^{(k)})$, and we will give  a description of the LCS $\mathrm{D}^\infty_{c}(M;\CO)$ (see Propositions \ref{prop:charlcsosm'}).
	Particularly, if $k=0$, it follows  that  $\mathrm{D}^\infty_{c}(M;\CO)$ is exactly the space of  compactly supported (complex-valued) smooth densities on $N$. 
	This justifies  Definition  \ref{de:formalden}.

	Write  $\mathrm{D}_c^\infty(N)$ for the space of all compactly supported (complex-valued) smooth densities on $N$.
	With the obvious  extension maps, the assignment
	\[\mathcal{D}^{(k)}_{N,c}:\ U\mapsto \mathrm{D}^\infty_c(U)[y_1^*,y_2^*,\dots,y_k^*]\quad (\text{$U$ is an open subset of $N$})\]
	forms a cosheaf of complex vector spaces over $N$. For every open subset $U$ of $N$, there is a   non-degenerate pairing
	\be\label{eq:pairing} \CO_N^{(k)}(U)\times \mathcal{D}_{N,c}^{(k)}(U)\rightarrow \C
	\ee given by \be \label{eq:pairing2}
	\la \sum_{J\in \BN^k} f_J y^J,
	\tau \cdot  (y^*)^L\ra= L! \cdot \int_U f_L\tau,\ee
	where $f_J\in \mathrm{C}^\infty(U)$, $\tau\in \mathrm{D}^\infty_c(U)$, and \[(y^*)^L=(y_1^*)^{l_1}(y_2^*)^{l_2}\cdots(y_k^*)^{l_k}\] for every $L=(l_1,l_2,\dots,l_k)\in\BN^k$.
	
	For every open subset $U$ of $N$, we view $\mathcal{D}_{N,c}^{(k)}(U)$ as a subspace of $(\CO_N^{(k)}(U))'$ through the pairing \eqref{eq:pairing}. 
	On the other hand, recall that 
	$\mathrm{D}_c^{\infty}(U;\CO)$ is also a subspace of $(\CO(U))'=(\CO_N^{(k)}(U))'$ (see \eqref{eq:Dcindual}).
	We have the following result.
	
	\begin{prpd}\label{prop:charocd} Suppose that $M=N^{(k)}$ as above. Then
		$\mathrm{D}^\infty_{c}(\CO)=\mathcal{D}_{N,c}^{(k)}$ as cosheaves over $M$.
	\end{prpd}
	\begin{proof} Note that  the extension maps   in the cosheaf $\mathcal{D}_{N,c}^{(k)}$
		are compatible with the transposes of the restriction maps of the sheaf $\CO$.
		Thus, it suffices to prove that for every open subset $U$ of $N$, 
		\be \label{eq:cosheaveseq}\mathrm{D}^\infty_{c}(U;\CO)=\mathcal{D}_{N,c}^{(k)}(U)\ee
		as linear subspaces of $(\CO(U))'$. 
		
		Due to the cosheaf property of $\mathrm{D}^\infty_{c}(\CO)$ and $\mathcal{D}_{N,c}^{(k)}$, it is enough to focus on the case that $N$ is an open submanifold of $\R^n$ for some  $n\in \BN$.
		In this case, it is easy to check that the sheaf $\underline{\mathcal{D}}$ of $\CO$-modules  satisfies the two conditions in \cite[Proposition 3.10]{CSW}. It follows that 
		\be \label{eq:diffcod}
		\mathrm{Diff}_{c}(\CO|_U,\underline{\mathcal{D}}|_U)
		=\bigoplus_{I\in \BN^n,J\in \BN^k} \mathrm{D}^\infty_c(U)\circ \partial_x^I\partial_y^J.
		\ee
		
		For every $\tau\in \mathrm{D}^\infty_c(U)$,  $I\in \BN^n$, $L\in \BN^k$, and $f=\sum_{J\in \BN^k}f_J y^J\in \CO_{N}^{(k)}(U)$, we have that 
		\be\begin{split}\label{eq:oemgaILf}
			\la\rho_U(\tau\circ \partial_x^I\partial_y^L), f\ra &=
			\int_U (\tau\circ \partial_x^I\partial_y^L)(f)\\
			&=L!\cdot \int_U \tau(\partial_x^I(f_L))\\
			&=L!\cdot \la \tau \partial_x^I, f_L\ra
			\\&=\la (\tau \partial_x^I)(y^*)^L,f\ra.
		\end{split}\ee
		Here $\tau\partial_x^I$ denotes the element in $\RD^\infty_c(U)$ such that
		\[\la \tau \partial_x^I, g\ra= \la \tau,  \partial_x^I g\ra\quad\] for every $g\in\mathrm C^{\infty}(U)$.
		Using  \eqref{eq:oemgaILf}, it follows  that
		\be\label{eq:onkusm=}\rho_U\left(\sum_{I\in \BN^n, L\in \BN^k} \tau_{I,L}\circ \partial_x^I\partial_y^L\right)=
		\sum_{I\in \BN^n, L\in \BN^k} (\tau_{I,L}\partial_x^I) (y^*)^L,\ee
		where $\tau_{I,L}\in \mathrm{D}_c^{\infty}(U)$ for $I\in \BN^n, L\in \BN^k$.
		This, together with \eqref{eq:diffcod}, shows that
		\begin{eqnarray*}
			\mathrm{D}^\infty_{c}(U;\CO)=\rho_U(\mathrm{Diff}_c(\CO|_U,\underline{\mathcal{D}}|_U))
			=\mathrm{D}^\infty_c(U)[y^*_1,y^*_2,\dots,y^*_k]= \mathcal{D}_{N,c}^{(k)}(U)
		\end{eqnarray*}
		as linear subspaces of $(\CO(U))'$, as desired.
	\end{proof}

	Similar to the space $\mathrm C^\infty(N)$,  the space $\mathrm{D}^\infty(N)$ of all smooth densities on $N$ is naturally an LCS under the smooth topology. 
	For every compact subset $K$ of $N$, write $\mathrm{D}_K^\infty(N)$ for the space of all
	smooth densities supported in $K$, which is equipped with the subspace topology of $\mathrm{D}^\infty(N)$. It is obvious that $\mathrm D_{K}^{\infty}(N)$ is a nuclear Fr\'{e}chet space.
	
	For every $(K,r)\in \mathcal{C}(N)$, equip $\mathrm{D}^\infty_K(N)[y^*_1,y^*_2,\dots,y^*_k]_{\le r}$  with the term-wise convergence topology (see Example
	\ref{ex:toponpolynomial}). Then $\mathrm{D}^\infty_K(N)[y^*_1,y^*_2,\dots,y^*_k]_{\le r}$  becomes a nuclear Fr\'echet space as well.
	
	Endow  
	\be\label{inddd} \mathcal{D}_{N,c}^{(k)}(N)=\varinjlim_{(K,r)\in \mathcal C(N)}
	\mathrm{D}^\infty_K(N)[y^*_1,y^*_2,\dots,y^*_k]_{\le r}
	\ee with the inductive limit topology. By Lemma \ref{lem:indtensorLF}, we have that 
	\be\label{eq:mathcalD=mathrmD}
	\mathcal{D}_{N,c}^{(k)}(N)=\mathrm{D}^\infty_c(N)\widetilde\otimes_{\mathrm{i}}\C[y^*_1,y^*_2,\dots,y^*_k]=\mathrm{D}^\infty_c(N)\widehat\otimes_{\mathrm{i}} \C[y^*_1,y^*_2,\dots,y^*_k]\ee as LCS.

	\begin{remarkd}The inductive limit of the directed system
		\[\left\{\mathrm{D}^\infty_c(N)[y^*_1,y^*_2,\dots,y^*_k]_{\le r}\right\}_{r\in \BN}\] defines another topology on $\mathcal{D}_{N,c}^{(k)}(N)$. By  Example  \ref{ex:toponpolynomial} and \eqref{eq:mathcalD=mathrmD}, 
		this topology coincides with the topology given by the inductive limit  \eqref{inddd}, and also agrees with  the inductive tensor product topology on
		$\mathcal{D}_{N,c}^{(k)}(N)=\mathrm{D}^\infty_c(N)\otimes \C[y^*_1,y^*_2,\dots,y^*_k]$.
	\end{remarkd}

	\begin{prpd}\label{prop:charlcsosm'} Assume that $M=N^{(k)}$ with $N$  a smooth manifold and $k\in \BN$.
		Then for every $(K,r)\in \mathcal{C}(M)$, we have that
		\begin{eqnarray}\label{eq:CONkUKr=}
			\mathrm{D}^\infty_{K,r}(M;\CO)=\mathrm{D}^\infty_K(N)[y^*_1,y^*_2,\dots,y^*_k]_{\le r}
		\end{eqnarray}
		as LCS. Furthermore, 
		\[\mathrm{D}^\infty_{c}(M;\CO)= \mathcal{D}_{N,c}^{(k)}(N)=\mathrm{D}^\infty_c(N)\widetilde\otimes_{\mathrm{i}}\C[y^*_1,y^*_2,\dots,y^*_k]=\mathrm{D}^\infty_c(N)\widehat\otimes_{\mathrm{i}} \C[y^*_1,y^*_2,\dots,y^*_k]\] as LCS.
	\end{prpd}
	\begin{proof} 
		In view of \eqref{eq:mathcalD=mathrmD}, we only need to prove the LCS identification \eqref{eq:CONkUKr=}. 
		Applying Lemma \ref{lem:homextDKr},  we may (and do) assume that  $N$ is an open submanifold of $\mathbb{R}^n$ ($n\in\BN$). 
		In such case,
		the LCS identification \eqref{eq:CONkUKr=} is implied by \eqref{eq:onkusm=} and Lemma \ref{lem:convergechar}. 
	\end{proof}

	\begin{prpd}\label{cor:NLFFMsm'} 
		Suppose that  $\CF$ is locally free of finite rank. Then 
		\begin{itemize}
			\item for every $(K,r)\in \mathcal{C}(M)$,
			$\mathrm{D}^\infty_{K,r}(M;\CF)$ is a nuclear Fr\'echet space;
			\item when $M$ is secondly countable, $\mathrm{D}^\infty_{c}(M;\CF)$ is a nuclear LF space (see Definition \ref{de:regualrlim}); and 
			\item in general, $\mathrm{D}^\infty_{c}(M;\CF)$ is a direct sum of nuclear LF spaces.
		\end{itemize}
	\end{prpd}
	\begin{proof} 
		Take an atlas $\{U_\gamma\}_{\gamma\in \Gamma}$ of $M$ such that for every $\gamma\in \Gamma$,  $\CF|_{U_\gamma}$ is a free $\CO|_{U_\gamma}$-module of
		finite rank $r_\gamma$. 
		For every $\gamma\in \Gamma$ and $(K',r')\in \mathcal{C}(U_{\gamma})$, we have that 
		\be\begin{split}\label{eq:locfinFsm}
			&\mathrm{D}^\infty_{K',r'}(U_\gamma;\CF)\\
			=\ &\rho_{U_\gamma}\left(\mathrm{Diff}_{K',r'}(\CF|_{U_\gamma},\underline{\mathcal{D}}|_{U_\gamma})\right)\cong\rho_{U_\gamma}\left(\mathrm{Diff}_{K',r'}\left((\CO|_{U_\gamma})^{r_\gamma},\underline{\mathcal{D}}|_{U_\gamma}\right)\right)
			\\
			=\ &\rho_{U_\gamma}\left(\left(\mathrm{Diff}_{K',r'}(\CO|_{U_\gamma},\underline{\mathcal{D}}|_{U_\gamma})\right)^{ r_\gamma}\right)
			=\left(\mathrm{D}^\infty_{K',r'}(U_\gamma;\CO)\right)^{ r_\gamma}
		\end{split}\ee 
		as LCS. This together with \eqref{eq:CONkUKr=} implies that \be\label{eq:DK'r'NF}
		\mathrm{D}^\infty_{K',r'}(U_\gamma;\CF)\ \text{is a nuclear Fr\'echet space}.\ee 
		
		Take a partition of unity $\{f_\gamma\}_{\gamma\in \Gamma}$ on $M$ subordinate to  $\{U_\gamma\}_{\gamma\in \Gamma}$.
		For every $(K,r)\in \mathcal{C}(M)$, it follows from Lemma \ref{lem:homextDKr} that there is a surjective, open and continuous linear map 
		\be \label{eq:DKUtoDMsur}
		\bigoplus_{\gamma\in \Gamma_K} \mathrm{D}^\infty_{K_\gamma,r}(U_\gamma;\CF)
		\rightarrow \mathrm{D}^\infty_{K,r}(M;\CF),
		\ee
		as well as a linear topological embedding 
		\[
		\iota_{K,r}:\ \mathrm{D}^\infty_{K,r}(M;\CF)\rightarrow 
		\bigoplus_{\gamma\in \Gamma_K} \mathrm{D}^\infty_{K_\gamma,r}(U_\gamma;\CF),\quad 
		\eta\mapsto  \{ (\eta\circ f_\gamma)|_{U_\gamma}\}_{\gamma\in \Gamma_K},
		\]
		where 
		$K_\gamma:=K\cap \mathrm{supp}\,f_\gamma$ and $ \Gamma_K:=\{\gamma\in \Gamma\mid K_\gamma\ne \emptyset\}$.
		
		We claim that 
		\be \label{eq:claimstrict}
		\text{the inductive limit}\ \varinjlim_{(K,r)\in \mathcal{C}(M)} \mathrm{D}^\infty_{K,r}(M;\CF)\ \text{is strict}.
		\ee
		Indeed, 
		let $(K_1,r_1)$ and $(K_2,r_2)$ be two pairs in $\mathcal{C}(M)$ such that $(K_1,r_1)\preceq (K_2,r_2)$. 
		Then we have the following commutative diagram:
		\[
		\begin{CD}
			\bigoplus\limits_{\gamma\in \Gamma_{K_1}} \mathrm{D}^\infty_{K_{1,\gamma},r_1}(U_\gamma;\CF)@>  >>  \bigoplus\limits_{\gamma\in \Gamma_{K_2}} \mathrm{D}^\infty_{K_{2,\gamma},r_2}(U_\gamma;\CF)\\
			@A \iota_{K_1,r_1}    AA
			@AA  \iota_{K_2,r_2} A\\
			\mathrm D^{\infty}_{K_1,r_1}(M;\CF) @ >> >
			\mathrm D^{\infty}_{K_2,r_2}(M;\CF).
		\end{CD}
		\]
		It follows from \eqref{eq:CONkUKr=} and \eqref{eq:locfinFsm} that the top horizontal arrow is a  linear topological embedding.
		This, together with the fact that both $\iota_{K_1,r_1}$ and $\iota_{K_2,r_2}$  are linear topological embeddings, implies that the bottom horizontal arrow is  also a linear topological embedding, as required.

		By using \eqref{eq:DK'r'NF} and \eqref{eq:DKUtoDMsur},  we obtain  that $\mathrm{D}^\infty_{K,r}(M;\CF)$ is a nuclear Fr\'echet space for every $(K,r)\in \mathcal{C}(M)$. 
		Furthermore, in view of the claim \eqref{eq:claimstrict}, 
		$\mathrm{D}^\infty_{c}(M;\CF)$ is  a nuclear LF space when $M$ is secondly countable, and in general,
		$\mathrm{D}^\infty_{c}(M;\CF)$ is a direct sum of nuclear LF spaces by \eqref{eq:Dc=directsum}.
	\end{proof}
	
	\begin{cord}\label{cor:DcFregular} 
		Suppose that  $\CF$ is locally free of finite rank. Then the inductive limit \[\varinjlim_{(K,r)\in \mathcal{C}(M)}\mathrm{D}^\infty_{K,r}(M;\CF)\] is regular (see Definition \ref{de:regualrlim}).
	\end{cord}
	\begin{proof}
		The assertion follows from Proposition  \ref{cor:NLFFMsm'}  and Lemma \ref{lem:basicsonindlim} (c).
	\end{proof}

	\begin{cord}\label{lem:strictapprFMsm'} Suppose  that  $\CF$  is  locally free of finite rank. 
		Then the LCS $\mathrm{D}^\infty_{c}(M;\CF)$ has the strict approximation property (see Definition \ref{de:strictappro}).
	\end{cord}
	\begin{proof}
		Let $\{U_\gamma\}_{\gamma\in \Gamma}$ be an open cover of $M$ such that for every $\gamma\in \Gamma$,
		\begin{itemize}
			\item $(U_\gamma,\CO|_{U_\gamma})\cong (\R^{n_\gamma})^{(k_\gamma)}$ for some $n_\gamma,k_\gamma\in \BN$; and
			\item $\CF|_{U_\gamma}$ is a free $\CO|_{U_\gamma}$-module.
		\end{itemize}
		Then $\mathrm{D}^\infty_{c}(U_{\gamma};\CF)$ is a finite direct sum of some copies of $\RD^\infty_c(\R^{n_\gamma})[y_1^*,y_2^*,\dots,y_{k_\gamma}^*]$
		(see Proposition \ref{prop:charlcsosm'} and \eqref{eq:locfinFsm}).
		Thus, by Lemma \ref{lem:strictapppoly} and \eqref{eq:mathcalD=mathrmD},  $\mathrm{D}^\infty_{c}(U_{\gamma};\CF)$ has the strict approximation property for every $\gamma\in \Gamma$.
		
		Take a partition of unity $\{f_\gamma\}_{\gamma\in \Gamma}$ on $M$ subordinate to the cover $\{U_\gamma\}_{\gamma\in \Gamma}$.
		Let $A$ be the quasi-closure  (see \cite[Definition 6.1]{CSW} or \cite[Page 91]{Sc}) of the subset $(\mathrm{D}^\infty_{c}(M;\CF))'\otimes \mathrm{D}^\infty_{c}(M;\CF)$ in $\CL_c(\mathrm{D}^\infty_c(M;\CF),\mathrm{D}^\infty_{c}(M;\CF))$.
		Recall from \eqref{eq:defTfFUsmtoFMsm} the continuous linear map \[T_{f_{\gamma}}:\ \CL_c(\mathrm{D}^\infty_{c}(U_{\gamma};\CF),\mathrm{D}^\infty_{c}(U_{\gamma};\CF))
		\rightarrow \CL_c(\mathrm{D}^\infty_{c}(M;\CF),\mathrm{D}^\infty_{c}(M;\CF)). \]
		Note that $T_{f_{\gamma}}$ maps $(\mathrm{D}^\infty_c(U_{\gamma};\CF))'\otimes \mathrm{D}^\infty_{c}(U_{\gamma};\CF)$
		into $(\mathrm{D}^\infty_{c}(M;\CF))'\otimes \mathrm{D}^\infty_{c}(M;\CF)$.
		%if  \[\phi:\quad  \mathrm{D}^\infty_{c}(U_{\gamma};\CF)\rightarrow \mathrm{D}^\infty_{c}(U_{\gamma};\CF)\] is a continuous linear map of finite rank, then so is
		%\[T_{f_\gamma}(\phi):\quad \mathrm{D}^\infty_{c}(M;\CF)\rightarrow \mathrm{D}^\infty_{c}(M;\CF).\]Here,$T_{f_\gamma}$ is as   in .
		As the inverse image of a quasi-closed set under a continuous linear map is still quasi-closed (see \cite[Page 92, $2^\circ$)]{Sc}), it follows that   $T_{f_\gamma}^{-1}(A)$ is a quasi-closed subset of $\CL_c(\mathrm{D}^\infty_{c}(U_{\gamma};\CF),\mathrm{D}^\infty_{c}(U_{\gamma};\CF))$ containing
		$(\mathrm{D}^\infty_{c}(U_{\gamma};\CF))'\otimes \mathrm{D}^\infty_c(U_{\gamma};\CF)$, and hence equals to $\CL_c(\mathrm{D}^\infty_{c}(U_{\gamma};\CF),\mathrm{D}^\infty_{c}(U_{\gamma};\CF))$.  Note that \[B:=\left\{\sum_{\gamma\in \Gamma_0}T_{f_\gamma}(\mathrm{id}_{\mathrm{D}^\infty_{c}(U_{\gamma};\CF)})\mid \Gamma_0\ \text{is a finite subset of}\ \Gamma\right\}\] is a bounded subset of $A$.
		The lemma then follows from the fact that
		\[\mathrm{id}_{\mathrm{D}^\infty_{c}(M;\CF)}=\sum_{\gamma\in \Gamma} T_{f_\gamma}(\mathrm{id}_{\mathrm{D}^\infty_{c}(U_{\gamma};\CF)})\in\text{ the closure of }B\subset A.\]
	\end{proof}

	\section{Formal generalized functions}\label{sec:formalgenfun}

	In this section,  we introduce a sheaf $\mathrm{C}^{-\infty}(\CF;E)$ of $\CO$-modules over $M$, which generalizes the concept of sheaves of $E$-valued generalized functions on smooth manifolds.
	
	\subsection{The sheaf $\mathrm{C}^{-\infty}(\CF;E)$}

	For an open subset $U$ of $M$, write
	\be\label{eq:defoffgf} \mathrm{C}^{-\infty}(U;\CF;E):=\CL_b(\mathrm{D}^\infty_{c}(U;\CF), E).\ee 
	By Proposition \ref{prop:charlcsosm'}, when $M$ is a smooth manifold and $\CF=\CO$, then $\mathrm{C}^{-\infty}(M;\CF;E)$ is the space of $E$-valued generalized functions on $M$.
	In view of this, we introduce the following definition. 
	
	\begin{dfn}\label{de:formalgen}
		An element in $\mathrm{C}^{-\infty}(M;\CO;E)$ 
		is called an $E$-valued formal generalized function on $M$. 
	\end{dfn}

	With the  transposes of the extension maps in $\mathrm{D}^\infty_{c}(\CF)$ as restriction maps, the assignment
	\[\mathrm{C}^{-\infty}(\CF;E):\ U\mapsto \mathrm{C}^{-\infty}(U;\CF;E) \quad\text{($U$ is an open subset of $M$)}\]
	forms a presheaf of complex vector spaces over $M$.
	When $E=\C$, set
	\[\mathrm{C}^{-\infty}(\CF):=\mathrm{C}^{-\infty}(\CF;E)\qquad\text{and}\qquad \mathrm{C}^{-\infty}(U;\CF):=\mathrm{C}^{-\infty}(U;\CF;E).\]
	% It is obvious that $\CF(M)$ is naturally identified with an $\CO(M)$-submodule of $\mathrm{C}^{-\infty}(M;\CF)$. 

	\begin{lemd}
		\label{cor:C-infty=product}
		We have the following LCS identification 
		\be\label{eq:C-infty=product}
		\mathrm{C}^{-\infty}(M;\CF;E)=\prod_{Z\in \pi_0(M)}
		\mathrm{C}^{-\infty}(Z;\CF;E).
		\ee
	\end{lemd}
	\begin{proof}
		For a family  $\{E_i\}_{i\in I}$ of LCS and a Hausdorff LCS $F$, there is an identification \be \label{eq:strong}\CL_b
		\left(\bigoplus_{i\in I} E_i,F\right)=\prod_{i\in I}\CL_b(E_i,F)\ee of  LCS (\cf \cite[\S\,39.8 (2)]{Ko2}). This, together with \eqref{eq:Dc=directsum}, implies that 
		\begin{eqnarray*}
			\mathrm{C}^{-\infty}(M;\CF;E)&=&
			\CL_b(\mathrm{D}^\infty_{c}(M;\CF), E)\\
			&=&\CL_b\left(\bigoplus_{Z\in \pi_0(M)}\mathrm{D}^\infty_{c}(Z;\CF), E\right)\\
			&=&\prod_{Z\in \pi_0(M)}\CL_b(\mathrm{D}^\infty_{c}(Z;\CF), E)\\
			&=&\prod_{Z\in \pi_0(M)}\mathrm{C}^{-\infty}(Z;\CF;E).
		\end{eqnarray*}
	\end{proof}

	As usual, we have the following result  (\cf  \cite[Theorem 24.1]{Tr}). 
	
	\begin{lemd}\label{lem:GFsheaf} The presheaf $\mathrm{C}^{-\infty}(\CF;E)$ over $M$ is a sheaf.
	\end{lemd}
	\begin{proof} Let $\{U_\alpha\}_{\alpha\in \Gamma}$ be an open cover of $M$, and let $\{f_\alpha\}_{\alpha\in \Gamma}$ be a partition of unity on $M$ subordinate to $\{U_\alpha\}_{\alpha\in \Gamma}$. By Lemma \ref{lem:rescosheafonopensubset}, we have that 
		\begin{eqnarray*}
			\la\eta,u\ra &=& \la \sum_{\alpha\in \Gamma}\eta\circ f_{\alpha},u \ra
			% \\&=& \sum_{\alpha\in\Gamma }\la \eta\circ f_{\alpha},u \ra 
			\\ &=&\sum_{\alpha\in\Gamma}\la \mathrm{ext}_{M,U_{\alpha}}((\eta\circ f_{\alpha})|_{U_{\alpha}}),u \ra \quad\quad \quad \text{(see \eqref{eq:etaf|U})}\\
			&=&\sum_{\alpha\in\Gamma} \la(\eta\circ f_{\alpha})|_{U_{\alpha}},u|_{U_{\alpha}}\ra  
		\end{eqnarray*}  for  every $u\in \mathrm C^{-\infty}(M;\CF;E)$ and $\eta\in \mathrm{D}^{\infty}_c(M;\CF)$. %Here $(\eta\circ f_{\alpha})|_{U_{\alpha}}$ is as in .
		This implies that the map \[\mathrm C^{-\infty}(M;\CF;E)\rightarrow \prod_{\alpha\in \Gamma}\mathrm C^{-\infty}(U_{\alpha};\CF;E),\quad u\mapsto\{u|_{U_\alpha}\}_{\alpha\in \Gamma}\] is injective.
		
		Now, let
		$\{u'_\alpha\in \mathrm{C}^{-\infty}(U_{\alpha};\CF;E)\}_{\alpha\in \Gamma} $ be a family of sections such that \[\text{$u'_\alpha|_{U_\alpha\cap U_\beta}
			=u'_\beta|_{U_\alpha\cap U_\beta}$\quad for all $\alpha,\beta\in \Gamma$.}\]
		We define a linear map $u':\ \mathrm{D}^{\infty}_c(M;\CF)\rightarrow E$ by setting
		\[ \la\eta,u'\ra:=\sum_{\alpha\in \Gamma} \la  (  \eta\circ f_\alpha)|_{U_\alpha}, u'_\alpha\ra,\]
		where $\eta\in \mathrm D^{\infty}_c(M;\CF)$. 
		Lemma \ref{lem:rescosheafonopensubset} implies that $u'$ is continuous and so lies in $\mathrm{C}^{-\infty}(M;\CF;E)$. Moreover, it is easy to check that $u'|_{U_\alpha}=u'_\alpha$ for all $\alpha\in \Gamma$.
		This finishes the proof.
	\end{proof}

	Let $U$ be an open subset of $M$. By Lemma \ref{lem:conextD} and \eqref{eq:trancont}, the 
	restriction map 
	\[
	\mathrm{C}^{-\infty}(M;\CF;E)\rightarrow \mathrm{C}^{-\infty}(U;\CF;E)
	\]
	is continuous. 
	Furthermore, we have the following generalization 
	of \eqref{eq:C-infty=product}.
	
	\begin{prpd}\label{prop:GFclosedembedding}
		Let $\{U_\gamma\}_{\gamma\in \Gamma}$ be an open cover of $M$. If   the inductive limit \[\varinjlim_{(K,r)\in \mathcal{C}(M)}\mathrm{D}^\infty_{K,r}(M;\CF)\] is regular, 
		then  the linear map
		\be\label{eq:gffembedding}
		\mathrm{C}^{-\infty}(M;\CF;E)\rightarrow \prod_{\gamma\in \Gamma}\mathrm{C}^{-\infty}(U_{\gamma};\CF;E), \quad u\mapsto \{u|_{U_\gamma}\}_{\gamma\in \Gamma}
		\ee
		is a closed  topological embedding.
	\end{prpd}
	\begin{proof} The continuity  of  \eqref{eq:gffembedding} is obvious, and one concludes from the sheaf property of $\mathrm{C}^{-\infty}(\CF;E)$
		that  \eqref{eq:gffembedding} is injective and has closed image.

		Let $B$ be a bounded subset of $\mathrm{D}^{\infty}_{c}(M; \CF)$, and let $\abs{\,\cdot\,}_\nu$ be a continuous seminorm on $E$.
		Pick a partition of unity $\{f_\gamma\}_{\gamma\in \Gamma}$ on $M$ subordinate to
		$\{U_\gamma\}_{\gamma\in \Gamma}$.
		By Lemma \ref{lem:rescosheafonopensubset}, for every $\gamma\in \Gamma$, the set
		\[
		B_\gamma:=\{(\eta\circ f_\gamma)|_{U_\gamma}\mid \eta\in B\}
		\]
		is  bounded in $\mathrm{D}^{\infty}_{c}(U_{\gamma}; \CF)$.
		
		Since the inductive limit \[\varinjlim_{(K,r)\in \mathcal{C}(M)}\mathrm{D}^\infty_{K,r}(M;\CF)\] is regular,
		there is a pair $(K,r)\in \mathcal{C}(M)$ such that
		$B$ is a bounded subset of $\mathrm{D}^{\infty}_{K,r}(M; \CF)$.
		The lemma then follows from the following inequality:
		\[|u|_{B,\abs{\,\cdot\,}_\nu}\le \sum_{\gamma\in \Gamma;\, \mathrm{supp}\,f_\gamma\cap K\ne \emptyset} |(u|_{U_\gamma})|_{B_\gamma,\abs{\,\cdot\,}_\nu}
		\qquad\text{for all}\ u\in \mathrm{C}^{-\infty}(M;\CF;E).\]
		Here the seminorms $|\cdot|_{B,\abs{\,\cdot\,}_\nu}$ and $ |\cdot|_{B_\gamma,\abs{\,\cdot\,}_\nu}$ are as in 
		\eqref{eq:defabsBnu}.
	\end{proof}
	
	For every open subset $U$ of $M$, $\mathrm C^{-\infty}(U;\CF;E)$ admits an  $\CO(U)$-module structure as follows:
	\[\CO(U)\times \mathrm C^{-\infty}(U;\CF;E) \rightarrow \mathrm C^{-\infty}(U;\CF;E), \quad (f,u)\mapsto fu,\]
	where $fu\in \mathrm C^{-\infty}(U;\CF;E)$ is defined by 
	\[\la\eta,fu\ra=\la \eta\circ f,u\ra \]
	for each 
	$\eta\in 
	\mathrm{D}^\infty_{c}(U;\CF)$. Here 
	$\eta\circ f$ is as in \eqref{eq:deffomega}. 
	This makes $\mathrm C^{-\infty}(\CF;E)$ 
	a sheaf of $\CO$-modules.
	
	\subsection{Formal generalized functions}
	Let $N$ be a smooth manifold. 
	As usual, we write 
	\[\RC^{-\infty}(N;E):=\CL_b(\RD^\infty_c(N);E)\] for the space of all $E$-valued generalized functions on $N$, and set $\RC^{-\infty}(N):=\RC^{-\infty}(N;\C)$. Let $n\in \BN$.
	Recall from \cite[Proposition 12]{Sc3} that  there is an LCS identification 
	\be \label{eq:charC-inftyNE}\RC^{-\infty}(\R^n;E)=\RC^{-\infty}(\R^n)\widetilde\otimes E,\ee
	and that $\RC^{-\infty}(\R^n;E)$ is complete provided that $E$ is complete.
	We have the following generalization in the setting of formal manifolds. 
	
	\begin{prpd}\label{prop:GFMEten}
		Suppose that  $\CF$ is locally free of finite rank.
		Then 
		\begin{itemize}
			\item $\mathrm{C}^{-\infty}(M;\CF)$ is a complete nuclear reflexive  LCS, and $\mathrm{C}^{-\infty}(M;\CF;E)$ is a  quasi-complete  LCS;
			\item the canonical linear map 
			$\mathrm{C}^{-\infty}(M;\CF)\otimes E\rightarrow \mathrm{C}^{-\infty}(M;\CF;E)$ induces an  identification
			\be \label{eq:CinftyE} \mathrm{C}^{-\infty}(M;\CF)\wt\otimes E=\mathrm{C}^{-\infty}(M;\CF;E)
			\ee
			as LCS; and
			\item 
			if $E$ is complete, then 
			\[ \mathrm{C}^{-\infty}(M;\CF)\wt\otimes E=\mathrm{C}^{-\infty}(M;\CF;E)=\mathrm{C}^{-\infty}(M;\CF)\wh\otimes E
			\] 
			as LCS.
		\end{itemize}
	\end{prpd}
	\begin{proof} When $M$ is secondly countable, the proposition is implied by  Corollaries \ref{cor:conmaps=tensorproduct},   \ref{lem:strictapprFMsm'}, Proposition \ref{cor:NLFFMsm'},  and Lemma \ref{lem:dualofnF}. 
		In general, the first assertion follows from \eqref{eq:C-infty=product} and the fact that  a product of nuclear (resp.\,reflexive; complete; quasi-complete) LCS is still a nuclear (resp.\,reflexive; complete; quasi-complete) LCS.
		For the second one,  by \eqref{eq:C-infty=product} and 
		\cite[Lemma 6.6]{CSW}, we have that 
		\begin{eqnarray*}
			\mathrm{C}^{-\infty}(M;\CF)\wt\otimes E&=&
			\left(\prod_{Z\in \pi_0(M)}\mathrm{C}^{-\infty}(Z;\CF)\right)\wt\otimes E\\
			&=&\prod_{Z\in \pi_0(M)} \mathrm{C}^{-\infty}(Z;\CF)\wt\otimes E\\
			&=&\prod_{Z\in \pi_0(M)}\mathrm{C}^{-\infty}(Z;\CF;E)\\
			&=&\mathrm{C}^{-\infty}(M;\CF;E).
		\end{eqnarray*}
		Similarly, when $E$ is complete, we  also obtain from \eqref{eq:C-infty=product} and 
		\cite[Lemma 6.6]{CSW} that 
		\[\mathrm{C}^{-\infty}(M;\CF;E)=\mathrm{C}^{-\infty}(M;\CF)\wh\otimes E.\]
	\end{proof}

	When $M=N^{(k)}$, just as \eqref{eq:pairing}, there is a canonical separately continuous bilinear map
	\[\RD^\infty_c(N)[y_1^*,y_2^*,\dots,y_k^*]\times \RC^{-\infty}(N;E)[[y_1,y_2,\dots,y_k]]\rightarrow E.\] Under this pairing, 
	$\RC^{-\infty}(N;E)[[y_1,y_2,\dots,y_k]]$ is a linear subspace of  $\mathrm{C}^{-\infty}(M;\CO;E)$.  
	Then  we have the following result.
	
	\begin{prpd}\label{prop:charGO} Assume that $M=N^{(k)}$ with $N$  a  smooth manifold and $k\in \BN$. Then we have that 
		\[
		\mathrm{C}^{-\infty}(M;\CO)
		=\RC^{-\infty}(N)\widehat\otimes\,\C[[y_1,y_2,\dots,y_k]]
		=\RC^{-\infty}(N)\widetilde\otimes\,\C[[y_1,y_2,\dots,y_k]]
		\]
		as LCS. For the general case, we also have that 
		\[\mathrm{C}^{-\infty}(M;\CO;E)=\RC^{-\infty}(N;E)[[y_1,y_2,\dots,y_k]]\] as LCS.
	\end{prpd}
	\begin{proof} From \eqref{eq:mathcalD=mathrmD}, Proposition \ref{prop:charlcsosm'} and Lemma \ref{lem:indtensorLF},  we have that
		\begin{eqnarray*}\mathrm{C}^{-\infty}(M;\CO)&=&(\RD^\infty_c(N)\widehat\otimes_{\mathrm{i}} \C[y_1^*,y_2^*,\dots,y_k^*])_b'\\
			&=&(\RD^\infty_c(N))_b'\widehat\otimes (\C[y_1^*,y_2^*,\dots,y_k^*])_b'\\
			&=&\RC^{-\infty}(N)\widehat\otimes\,\C[[y_1,y_2,\dots,y_k]]\\
			&=&\RC^{-\infty}(N)\widetilde\otimes\,\C[[y_1,y_2,\dots,y_k]].
		\end{eqnarray*}
		Then by  \eqref{eq:CinftyE} and Example \ref{ex:E-valuedpowerseries}, the following  equalities hold:
		\begin{eqnarray*}
			\mathrm{C}^{-\infty}(M;\CO;E)&=&\mathrm{C}^{-\infty}(M;\CO)\widetilde\otimes\, E\\
			&=&\RC^{-\infty}(N)\widetilde\otimes\,\C[[y_1,y_2,\dots,y_k]]\widetilde\otimes\, E\\
			&=& \RC^{-\infty}(N;E)\widetilde\otimes\,\C[[y_1,y_2,\dots,y_k]]\\
			&=&\RC^{-\infty}(N;E)[[y_1,y_2,\dots,y_k]].
		\end{eqnarray*}
	\end{proof}

	\section{Formal distributions}\label{sec:formaldis}
	In this section,  we introduce  a sheaf $\mathrm{D}^{-\infty}(\CF;E)$ of
	$\CO$-modules over $M$, which generalizes the concept of sheaves of $E$-valued distributions on smooth manifolds.
	
	\subsection{The cosheaf $\CF_c$}
	Recall from  \cite[Section 3.2]{CSW} that, with the extension by zero maps, the assignment
	\be\label{eq:Fcdef} \CF_c:\ U\mapsto \CF_c(U)\qquad (\text{$U$ is an open subset of $M$})\ee
	forms a flabby 
	cosheaf of $\CO$-modules over $M$.
	
	For every compact subset $K$  of $M$,
	write $\CF_K(M)$ for the space of all global sections of $\CF$ that are supported in $K$.
	Equip $\CF(M)$ with the smooth topology, and equip $\CF_K(M)$ with the subspace topology.

	\begin{lemd}\label{lem:HausFimplyclosed} If $\CF(M)$ is Hausdorff, then $\CF_K(M)$ is closed in $\CF(M)$ for every compact subset $K$  of $M$.
	\end{lemd}
	\begin{proof} Let  $\{u_i\}_{i\in I}$ be a net in $\CF_K(M)$ that converges to an element $u\in \CF(M)$.
		It suffices to show that $u$ is also supported in $K$.  Otherwise,
		take a point $a\in\mathrm{supp}\,u$ and an open
		neighborhood $U$ of $a$ in $M$ such that $U\cap K=\emptyset$.
		Let $g$ be a formal function on $M$ such that $\mathrm{supp}\,g\subset U$ and $g_a=1$, where $g_a$ is the germ of $g$ in $\CO_a$.
		Then the zero net $\{g u_i\}_{i\in I}$ converges to a nonzero element $g u$ in $\CF(M)$ by \cite[Lemma 4.3]{CSW}, which leads to a contradiction.
	\end{proof}

	Equip \[\CF_c(M)=\varinjlim_{K \,\text{is a compact subset of}\, M}\CF_K(M)\] with
	the inductive limit topology, where the compact subsets of $M$ are directed by inclusions. 
	
	We refer to Appendix \ref{appendixB1} for the usual notations $\RC^\infty_K(N)$, $\RC^\infty_c(N;E)$, and etc. Here $N$ is a smooth manifold and $K$ is a compact subset of it.
	\begin{exampled}\label{ex:DON} If $M=N^{(k)}$ with $N$ a smooth manifold and $k\in \BN$,
		then
		\be \label{eq:desOKN} \CO_K(M)= \RC^\infty_K(N)[[y_1,y_2,\dots,y_k]]\quad (\text{$K$ is a compact subset of $M$})\ee is a nuclear Fr\'echet space. 
		Furthermore, we have that 
		\[\CO_c(M)=\RC^\infty_c(N;\C[[y_1,y_1,\dots,y_k]])
		=\RC^\infty_c(N)\widehat\otimes_{\mathrm{i}} \C[[y_1,y_2,\dots,y_k]]\]
		as LCS by Example \ref{ex:E-valuedcptsm}. 
	\end{exampled}

	\begin{lemd}\label{lem:FcMhau+regular} If $\CF(M)$ is Hausdorff, then the space $\CF_c(M)$ is Hausdorff, and the inductive limit ${\varinjlim}_{K}\CF_K(M)$ is regular (see Definition \ref{de:regualrlim}).
	\end{lemd}
	\begin{proof}
		The assertion follows from Lemma  \ref{lem:HausFimplyclosed} and Lemma \ref{lem:basicsonindlim} (b), (c).
	\end{proof}

	\begin{lemd}\label{lem:mfofD} Let $U$ be an open subset of $M$, $K$ a compact subset of $M$, and  $f$  a formal function on $M$ supported in $U$.\\ \noindent (a) The extension  map
		\be
		\mathrm{ext}_{M,U}:\quad \CF_c(U)\rightarrow \CF_c(M)
		\ee
		is continuous.\\
		\noindent (b)  The linear map
		\be \label{eq:FKMtoFKU}
		m_f:\quad \CF_K(M) \rightarrow \CF_{K\cap \mathrm{supp}\,f}(U), \quad
		u\mapsto (fu)|_U 
		\ee
		is well-defined and continuous.\\
		\noindent (c)  The linear map
		\be \label{eq:FcMtoFcU}
		\begin{array}{rcl}
			m_f:\quad \CF_c(M) \rightarrow \CF_c(U), \quad
			u\mapsto (fu)|_U 
		\end{array}\ee
		is well-defined and continuous.\\
		\noindent (d) The linear map 
		\be
		\label{eq:LcUtoLcM}
		\begin{array}{rcl}
			T_f:\quad \CL_c(\CF_c(U),\CF_c(U))&\rightarrow &\CL_c(\CF_c(M),\CF_c(M)),\\
			\phi&\mapsto & \mathrm{ext}_{M,U}\circ\phi \circ m_f\end{array}
		\ee
		is continuous.
	\end{lemd}
	\begin{proof} 
		Let $K'$ be a compact subset of $U$, and let $g$ be a formal function on $M$ such that 
		$\mathrm{supp}\,g\subset U$ and $g|_{K'}=1$.
		Then we have that 
		\[
		|\mathrm{ext}_{M,U}(u)|_{D}=|u|_{(g\circ D)|_U}\quad \text{for all $u\in \CF_{K'}(U)$},
		\]
		where $D\in \mathrm{Diff}_c(\CF,\underline{\CO})$ and 
		$g\circ D$ is as in \eqref{eq:defofDf}.
		This implies that the map 
		\[
		\mathrm{ext}_{M,U}:\quad \CF_{K'}(U)\rightarrow \CF_{K'}(M)
		\]
		is continuous. Then the assertion (a) follows. 
		
		Note that the assertion (d) follows from \eqref{eq:trancont}, the assertions (a) and (c). The assertion (c) is implied by the assertion (b). 
		Thus, it remains to prove the assertion (b). 
		
		It is clear that the map \eqref{eq:FKMtoFKU} is well-defined.
		Also, we have that
		\[
		|((fu)|_U)|_{D}=|u|_{(\mathrm{ext}_{M,U}D)\circ f}
		\quad \text{for all $u\in \CF_{K}(M)$},
		\]
		where $D\in \mathrm{Diff}_c(\CF|_U,\underline{\CO}|_U)$ and 
		$(\mathrm{ext}_{M,U}D)\circ f$ is as in \eqref{eq:defofDf}.
		These imply the assertion (b). 
	\end{proof}

	Similar to Corollary \ref{lem:strictapprFMsm'}, we have the following result.
	\begin{lemd}\label{lem:CFcMappro}
		If $\CF$ is locally free of finite rank,
		then the space $\CF_c(M)$ has the strict approximation property (see Definition \ref{de:strictappro}).
	\end{lemd}
	\begin{proof}
		By Lemma \ref{lem:strictapppoly} and Example \ref{ex:DON}, the assertion holds when $M=(\R^n)^{(k)}$ ($n,k\in \BN$) and $\CF$ is free.  In the general case, by considering  Lemma \ref{lem:mfofD} (d), the assertion is proved in a way similar to that in  Corollary \ref{lem:strictapprFMsm'}.
		%In view of  Lemma \ref{lem:mfofD} (d), by  Corollary \ref{lem:strictapprFMsm'}a similar argument to that in the proof of Corollary \ref{lem:strictapprFMsm'}, the proof is reduced to the special case that $M=(\R^n)^{(k)}$ $(n,k\in \BN)$ and $\CF$ is free. However, in this special case, the assertion is implied by  Lemma  \ref{lem:strictapppoly} and Example \ref{ex:DON}.
	\end{proof}
	
	\begin{lemd}\label{lem:extFKopen} Let $K$ be a compact subset of $M$, $\{U_i\}_{i=1}^s$   a finite cover of $K$ by open subsets of $M$, and
		$\{f_i\}_{i=1}^s$  a family of formal functions on $M$ such that \[\mathrm{supp}\,f_i\subset U_i \quad   \text{for all $1\leq i\leq s$, and} \quad \sum_{i=1}^s f_i|_K=1.\]
		Write $K_i:=\mathrm{supp}\,f_i\cap K$.
		Then the linear map
		\be \label{eq:FKitoFK}
		\bigoplus_{i=1}^s \CF_{K_i}(U_i)  \rightarrow    \CF_K(M), \quad
		(u_1,u_2,\dots,u_s) \mapsto   \sum\limits_{i=1}^{s}\mathrm{ext}_{M,U_i}(u_i)
		\ee
		is continuous, open and surjective. 
	\end{lemd}
	\begin{proof}
		The assertion follows from the fact that the continuous linear map  (see Lemma \ref{lem:mfofD} (b))
		\[
		\CF_K(M) \rightarrow \bigoplus_{i=1}^s \CF_{K_i}(U_i)    , \quad
		u\mapsto ((f_1u)|_{U_1},(f_2u)|_{U_2},\dots,(f_su)|_{U_s}).
		\]
		is a right inverse of 
		\eqref{eq:FKitoFK}. 
	\end{proof}
	
	\begin{prpd}\label{prop:FcMopen} Let $\{U_\gamma\}_{\gamma\in \Gamma}$ be an open cover of $M$.
		Then the linear map
		\be \label{eq:FCUgammatoFCM}
		\bigoplus_{\gamma\in \Gamma}\CF_c(U_\gamma)\rightarrow \CF_c(M),\quad 
		\{u_{\gamma}\}_{\gamma\in \Gamma}\mapsto \sum_{\gamma\in \Gamma} \mathrm{ext}_{M,U_{\gamma}} (u_{\gamma}) \ee
		is continuous, open and surjective.
	\end{prpd}
	\begin{proof}
		As in the proof of  Proposition \ref{prop:extconD}, the assertion follows as the map 
		\[
		\CF_c(M)\rightarrow\bigoplus_{\gamma\in \Gamma}\CF_c(U_\gamma),\quad 
		u\mapsto \{(f_\gamma u)|_{U_\gamma}\}_{\gamma\in \Gamma}
		\]
		is a continuous right inverse of 
		\eqref{eq:FCUgammatoFCM}, where 
		$\{f_\gamma\in \CO(M)\}_{\gamma\in \Gamma}$ is a partition of unity subordinate to $\{U_\gamma\}_{\gamma\in \Gamma}$.
	\end{proof}

	As a special case of Proposition \ref{prop:FcMopen}, we have the following result.
	\begin{cord} 
		There is an LCS identification
		\be\label{eq:FcMdec}
		\CF_c(M)=\bigoplus_{Z\in \pi_0(M)} \CF_c(Z).
		\ee
	\end{cord}

	\begin{cord}\label{cor:CFcMisNLF}  
		Assume that $\CF$ is locally free of finite rank.
		Then
		\begin{itemize}
			\item for every compact subset $K$ of $M$, $\CF_K(M)$ is a nuclear Fr\'echet space;
			\item when $M$ is secondly countable, $\CF_c(M)$ is a nuclear LF space (see Definition \ref{de:LFspace}); and
			\item in general, $\CF_c(M)$ is a direct sum of nuclear LF spaces. 
		\end{itemize}
	\end{cord}
	\begin{proof} 
		Using Example \ref{ex:DON} and  Lemma \ref{lem:extFKopen}, one concludes that  each $\CF_K(M)$ 
		is a nuclear Fr\'echet space. Then 
		$\CF_c(M)$ is a nuclear LF space when $M$ is secondly countable. 
		In general,  $\CF_c(M)$ is a direct sum of nuclear LF spaces by \eqref{eq:FcMdec}. 
	\end{proof}

	\subsection{The sheaf $\mathrm{D}^{-\infty}(\CF;E)$} In this subsection, we introduce a sheaf 
	$\mathrm{D}^{-\infty}(\CF;E)$ of $\CO$-modules over $M$ associated to $\CF$ and $E$.
	
	For every open subset $U$ of $M$, set 
	\[
	\mathrm{D}^{-\infty}(U;\CF;E):= \CL_b(\CF_c(U),E).
	\]
	If $M$ is a smooth manifold and  $\CF=\CO$, then $\mathrm{D}^{-\infty} (M;\CF;E)$ is the space of  $E$-valued distributions on $M$ (see \cite{Sc3}).
	
	\begin{dfn}\label{def:distribution}
		An element in $\mathrm{D}^{-\infty}(M;\CO;E)$ is called an $E$-valued formal distribution on $M$.   
	\end{dfn}

	With the transposes of the extension maps of $\CF_c$ as restriction maps, the assignment
	\[\mathrm{D}^{-\infty}(\CF;E):\ U\mapsto \mathrm{D}^{-\infty}(U;\CF;E)\quad (U\ \text{is an open subset of}\ M)\]
	forms a presheaf of complex vector spaces over $M$. We also write
	\[\mathrm{D}^{-\infty}(\CF):=\mathrm{D}^{-\infty}(\CF;\BC)\qquad\text{and}\qquad \mathrm{D}^{-\infty}(U;\CF):=\mathrm{D}^{-\infty}(U;\CF;\BC).\]

	\begin{lemd} The presheaf $\mathrm{D}^{-\infty}(\CF;E)$ is a sheaf on $M$.
	\end{lemd}
	\begin{proof} Similar to the proof of Lemma \ref{lem:GFsheaf}, the assertion follows easily from the continuity of
		the map \eqref{eq:FcMtoFcU}.
	\end{proof}
	
	\begin{cord}
		There is an LCS identification 
		\be \label{eq:D-inftydec}
		\mathrm{D}^{-\infty}(M;\CF;E)=\prod_{Z\in \pi_0(M)} \mathrm{D}^{-\infty}(Z;\CF;E).
		\ee
	\end{cord}
	\begin{proof}
		By using \eqref{eq:strong} and \eqref{eq:FcMdec},
		the assertion follows in a way similar to the proof of  Lemma \ref{cor:C-infty=product}.    
	\end{proof}

	For an open subset $U$ of $M$, the restriction map 
	\be\mathrm{D}^{-\infty}(M;\CF;E)\rightarrow \mathrm{D}^{-\infty}(U;\CF;E)\ee is continuous by Lemma \ref{lem:mfofD} (a) and \eqref{eq:trancont}.  
	The following result is similar to Proposition \ref{prop:GFclosedembedding}.
	
	\begin{prpd}\label{prop:D-inftyclosdtopim}
		Let $\{U_\gamma\}_{\gamma\in \Gamma}$ be an open cover of $M$.
		If $\CF(M)$ is
		Hausdorff, then the canonical linear map
		\be\label{eq:dffembedding}
		\mathrm{D}^{-\infty}(M;\CF;E)\rightarrow \prod_{\gamma\in \Gamma} \mathrm{D}^{-\infty}(U_{\gamma};\CF;E), \quad \eta\mapsto \{\eta|_{U_\gamma}\}_{\gamma\in \Gamma}
		\ee
		is a closed topological embedding.
	\end{prpd}
	\begin{proof}
		By using Lemma \ref{lem:FcMhau+regular} and the continuity of \eqref{eq:FcMtoFcU},
		the assertion follows in a way similar to the proof of  Proposition \ref{prop:GFclosedembedding}.
	\end{proof}
	Note that $\mathrm{D}^{-\infty}(\CF;E)$ is naturally  a sheaf of $\CO$-modules with the module structure   given by \be\label{eq:omoduleDinfty} \CO(U)\times \mathrm{D}^{-\infty}(U;\CF;E)\rightarrow \mathrm{D}^{-\infty}(U;\CF;E), \quad (f,\eta)\mapsto \eta\circ f \ee for every open subset $U$ of $M$.
	Here, for $f\in \CO(U)$ and $\eta\in \mathrm{D}^{-\infty}(U;\CF;E)$, the element $\eta\circ f\in \mathrm{D}^{-\infty}(U;\CF;E)$  is defined  by 
	\be \label{eq:defofetafD}\la \eta\circ f, u\ra=\la \eta, fu\ra\ee
	for every $u\in \CF_c(U)$.
	
	%Let $U$ be an open subset of $M$, and let $f$ be a formal function on $M$ supported in $U$.As the map \eqref{eq:FcMtoFcU} is continuous, the linear map\be \label{eq:mfonD-infty} m_f:\ \mathrm{D}^{-\infty}(U;\CF;E)\rightarrow \mathrm{D}^{-\infty}(M;\CF;E),\quad \eta\mapsto \eta\circ f \eeis also continuous, where $\eta\circ f$ is defined by \[\la \eta\circ f,u\ra= \la\eta, (fu)|_U\ra\] for every $u\in \CF_c(M)$.

	As usual, for a smooth manifold $N$,
	write 
	\[\RD^{-\infty}(N;E):=\CL_b(\RC^\infty_c(N);E)\] for the space of  $E$-valued distributions on $N$, and also write 
	$\RD^{-\infty}(N):=\RD^{-\infty}(N;\C)$. 
	Recall from \cite[Proposition 12]{Sc3}   that for every $n\in \BN$, there is an LCS identification 
	\[\RD^{-\infty}(\R^n;E)=\RD^{-\infty}(\R^n)\widetilde\otimes E ,\]
	and $\RD^{-\infty}(\R^n;E)$ is complete provided that $E$ is complete. 
	We have the following generalizations in the setting of formal manifolds. 
	
	\begin{prpd}\label{prop:charDFME}
		Assume that $\CF$ is locally free of finite rank.
		Then 
		\begin{itemize}
			\item  $\mathrm{D}^{-\infty}(M;\CF)$ is a complete nuclear reflexive  LCS, and  $\mathrm{D}^{-\infty}(M;\CF;E)$ is a   quasi-complete  LCS;
			\item the canonical linear map 
			$\mathrm{D}^{-\infty}(M;\CF)\otimes E\rightarrow \mathrm{D}^{-\infty}(M;\CF;E)$ induces an  identification
			\[\mathrm{D}^{-\infty}(M;\CF)\wt\otimes E=\mathrm{D}^{-\infty}(M;\CF;E)
			\]
			of LCS; and
			\item  if $E$ is complete, then 
			\[\mathrm{D}^{-\infty}(M;\CF)\wt\otimes E=\mathrm{D}^{-\infty}(M;\CF;E)=\mathrm{D}^{-\infty}(M;\CF)\wh\otimes E
			\]
			as LCS.
		\end{itemize}
	\end{prpd}
	\begin{proof} When $M$ is secondly countable, the proposition is implied by  Corollaries \ref{cor:conmaps=tensorproduct}, \ref{cor:CFcMisNLF} and Lemmas \ref{lem:CFcMappro}, \ref{lem:dualofnF}. In  the general case, the assertion follows from  \cite[Lemma 6.6]{CSW} and \eqref{eq:D-inftydec} and the fact that  a product of  nuclear (resp.\,reflexive; complete; quasi-complete) LCS is still a nuclear (resp.\,reflexive; complete; quasi-complete) LCS.
	\end{proof}
	
	In particular, we have the following
	description of $\mathrm{D}^{-\infty}(M;\CO;E)$
	by using Proposition \ref{prop:charDFME}, Example \ref{ex:DON} and Lemma \ref{lem:indtensorLF}.
	
	\begin{prpd}\label{prop:localD-infty} Suppose that $M=N^{(k)}$, where $N$ is  a  smooth manifold and $k\in \BN$. Then 
		\begin{eqnarray*}
			\mathrm{D}^{-\infty}(M;\CO;E)&=&\RD^{-\infty}(N)\widetilde\otimes \C[y_1^*,y_2^*,\dots,y_k^*]\widetilde\otimes E\\
			&=&(\RD^{-\infty}(N)\widehat\otimes \C[y_1^*,y_2^*,\dots,y_k^*])\widetilde\otimes E\\
			&=&\RD^{-\infty}(N;\C[y_1^*,y_2^*,\dots,y_k^*]\widetilde\otimes E)
		\end{eqnarray*} as LCS. 
	\end{prpd}

	\section{Compactly supported formal distributions}\label{sec:comformaldis}
	%Throughout this section, let $\CF$ be a sheaf of  $\CO$-modules, and let $E$ be a quasi-complete  LCS. 
	In this section, we introduce the cosheaf $\mathrm{D}^{-\infty}_c(\CF;E)$ of $\CO$-modules over $M$, which generalizes the concept of cosheaves of compactly supported vector-valued distributions on smooth manifolds.. %defined by  taking the compactly supported sections in $\mathrm{D}^{-\infty}(\CF;E)$.
	\subsection{The cosheaf $\mathrm{D}^{-\infty}_c(\CF;E)$}
	For every open subset $U$ of $M$,
	write $\mathrm{D}^{-\infty}_c(U;\CF;E)$ for the space of all compactly supported sections in $\mathrm{D}^{-\infty}(U;\CF;E)$.
	Then the assignment 
	\be \label{eq:defofDc-infty}
	\mathrm{D}^{-\infty}_c(\CF;E):\ U\mapsto \mathrm{D}^{-\infty}_c(U;\CF;E)\quad \text{($U$ is an open subset of $M$)}
	\ee
	forms a flabby cosheaf of $\CO$-modules on $M$.
	Here the module structure is induced by \eqref{eq:omoduleDinfty}.

	For every compact subset 
	$K$ of $M$, write $\mathrm{D}^{-\infty}_K(M;\CF;E)$ for the subspace of $\mathrm{D}^{-\infty}(M;\CF;E)$ consisting  of  global sections
	supported in $K$. Equip it with the subspace topology of $\mathrm{D}^{-\infty}(M;\CF;E)$. 
	We also write 
	\[
	\mathrm{D}^{-\infty}_K(M;\CF):=\mathrm{D}^{-\infty}_K(M;\CF;\BC).\]
	
	The following result is obvious.
	
	\begin{lemd}\label{lem:closedDisK}
		For every compact subset 
		$K$ of $M$,  the space $\mathrm{D}^{-\infty}_K(M;\CF;E)$ 
		is closed in $\mathrm{D}^{-\infty}(M;\CF;E)$.
	\end{lemd} 
	
	\begin{lemd}\label{lem:DKMFEtoDKUFE}
		Let $U$ be an open subset of $M$, $K$ a compact subset of $M$, and  $f$ a formal function on $M$ supported in $U$. Then the linear map 
		\[
		\mathrm{D}^{-\infty}_K(M;\CF;E)
		\rightarrow 
		\mathrm{D}^{-\infty}_{K\cap \mathrm{supp}\,f}(U;\CF;E),\quad 
		\eta\mapsto (\eta\circ f)|_U
		\]
		is well-defined and continuous. Here $\eta\circ f\in \mathrm D^{-\infty}(M;\CF; E)$ is defined as in \eqref{eq:defofetafD}.
	\end{lemd}
	
	\begin{proof} Let $B$ be a bounded subset of $\CF_c(U)$, and let $\abs{\,\cdot\,}_{\nu}$ be a continuous seminorm on $E$. 
		Set \[B':=\{f\,\mathrm{ext}_{M,U}(u)\mid u\in B\},\]
		which is a bounded subset in  $\CF_c(M)$.
		Then we have that (see \eqref{eq:defabsBnu})
		\[
		|((\eta\circ f)|_U)|_{B,\abs{\,\cdot\,}_{\nu}}
		=|\eta|_{B',\abs{\,\cdot\,}_{\nu}},
		\]
		which implies the lemma. 
	\end{proof}
	
	Endow \[\mathrm{D}^{-\infty}_c(M;\CF;E)=\varinjlim_{K \text{ is a compact subset of } M}\mathrm{D}^{-\infty}_K(M;\CF;E)\] with the inductive limit topology. Here compact subsets of $M$ 
	are directed by inclusions.
	We also write
	\[\mathrm{D}^{-\infty}_c(\CF):=\mathrm{D}^{-\infty}_c(\CF;\BC)\quad\text{and}\quad \mathrm{D}^{-\infty}_c(M;\CF):=\mathrm{D}^{-\infty}_c(M;\CF;\BC).\]

	The following result is straightforward (\cf Lemma \ref{lem:mfofD} (a)). 
	\begin{lemd}
		For every open subset $U$  of $M$, the extension map
		\[\mathrm{ext}_{M,U}:\ \mathrm{D}^{-\infty}_c(U;\CF;E)\rightarrow \mathrm{D}^{-\infty}_c(M;\CF;E)\]
		is continuous.
	\end{lemd}

	\begin{lemd}\label{eq:KofD-infty} Let $K$ be a compact subset of $M$, $\{U_i\}_{i=1}^s$ a finite cover of $K$ by open subsets of $M$, and
		$\{f_i\}_{i=1}^s$ a family of formal functions on $M$ such that \[\mathrm{supp}\,f_i\subset U_i \, \, \text{ for $1\leq i\leq s$}\quad \text{and} \quad \sum_{i=1}^s f_i|_K=1.\]
		Write $K_i:=\mathrm{supp}\,f_i\cap K$.
		Then the linear  map
		\be \label{eq:DFKitoDFK}\begin{array}{rcl}
			\bigoplus_{i=1}^s\mathrm{D}^{-\infty}_{K_i}(U_i;\CF;E)&\rightarrow& \mathrm{D}^{-\infty}_K(M;\CF;E),\\
			(\eta_1,\eta_2,\dots,\eta_s)&\mapsto& \sum_{i=1}^s\mathrm{ext}_{M,U_i}(\eta_i)
		\end{array}
		\ee
		is continuous, open and surjective.
	\end{lemd}
	\begin{proof}
		Note that  the continuous linear map (see Lemma \ref{lem:DKMFEtoDKUFE}) 
		\bee \begin{array}{rcl}
			\mathrm{D}^{-\infty}_K(M;\CF;E)&\rightarrow& \bigoplus_{i=1}^s\mathrm{D}^{-\infty}_{K_i}(U_i;\CF;E),\\
			\eta&\mapsto& ((\eta\circ f_1)|_{U_1},(\eta\circ f_2)|_{U_2},\dots,(\eta\circ f_s)|_{U_s})
		\end{array}
		\eee
		is a right inverse of \eqref{eq:DFKitoDFK}. This implies the lemma.
	\end{proof}
	
	Furthermore, we have the following result.
	
	\begin{prpd}\label{prop:DFcEopenonto} Let $\{U_\gamma\}_{\gamma\in \Gamma}$ be an open cover of $M$.
		Then the linear  map
		\be\label{eq:extsurofDc} 
		\bigoplus_{\gamma\in \Gamma}\mathrm{D}^{-\infty}_c(U_{\gamma};\CF;E)\rightarrow \mathrm{D}^{-\infty}_c(M;\CF;E),\quad 
		\{\eta_\gamma\}_{\gamma\in \Gamma}\mapsto \sum_{\gamma\in \Gamma}\mathrm{ext}_{M,U_\gamma}(\eta_\gamma)
		\ee
		is continuous, open and surjective.
	\end{prpd}
	\begin{proof} The proof is similar to that in Propositions \ref{prop:extconD} and \ref{prop:FcMopen},
		and it is omitted here.
	\end{proof}

	As  a special case of Proposition \ref{prop:DFcEopenonto}, we have the following result. 
	\begin{cord}\label{cor:Dc-infty=sum}
		There is an  LCS identification
		\be
		\mathrm{D}_c^{-\infty}(M;\CF;E)
		=\bigoplus_{Z\in \pi_0(M)} \mathrm{D}_c^{-\infty}(Z;\CF;E).
		\ee
	\end{cord}
	
	% \begin{lemd}
		%     If  $\CF$ is locally free of finite rank, then $\mathrm{D}^{-\infty}_c(M;\CF;E)$  is quasi-complete and Hausdorff. 
		% \end{lemd}
	% \begin{proof}
		%     For every compact subset $K$ of $M$, it is easy to see that  $\mathrm{D}^{-\infty}_K(M;\CF;E)$ is a closed subspace of $\mathrm{D}^{-\infty}(M;\CF;E)$.
		% This together with Proposition \ref{prop:charDFME} implies that $\mathrm{D}^{-\infty}_K(M;\CF;E)$ is a quasi-complete Hausdorff LCS. 
		% The assertion then follows from Lemma \ref{lem:basicsonindlim} (e) and Corollary \ref{cor:Dc-infty=sum}.
		% \end{proof}
	
	\subsection{Structure of  $\RD_c^{-\infty}(M;\CF;E)$}\label{subsec:struofcompactdis}
	For a smooth manifold $N$, write
	$\mathrm D_c^{-\infty}(N;E)$ for the LCS of  compactly supported $E$-valued distributions
	on $N$, and write    \[\mathrm D_c^{-\infty}(N):=\mathrm D_c^{-\infty}(N;\C).\]
	It is known that if $E$ is a DF space (see \cite[\S\,29.3]{Ko1}), then for every $n\in \BN$, 
	\be \label{eq:charcsEdN} 
	\mathrm D_c^{-\infty}(\R^n)\widetilde\otimes E=
	\mathrm D_c^{-\infty}(\R^n)\widehat\otimes E=
	\mathrm D_c^{-\infty}(\R^n;E)=\CL_b(\RC^\infty(\R^n),E)
	\ee
	as LCS (see \cite[Pages 62-63]{Sc3}). 
	The main goal of this subsection is to generalize 
	these LCS identifications to the setting of formal manifolds.
	
	We start with the following lemma.

	\begin{lemd}\label{lem:FcdenseinF} The space $\CF_c(M)$ is dense in $\CF(M)$.
	\end{lemd}
	\begin{proof} For every $Z\in \pi_0(M)$, choose a  sequence $\{U_{r,Z}\}_{r\in \BN}$ of relatively compact open subsets in $Z$ such that
		\[U_{0,Z}\subset \overline{U_{0,Z}}\subset  U_{1,Z}\subset \overline{U_{1,Z}}\subset U_{2,Z}\subset \overline{U_{2,Z}}\subset \cdots \]
		and
		$\cup_{r\in \BN} U_{r,Z}=Z$.
		Let $\Gamma$ be the directed set  of all  open subsets in $M$ that have the form
		\be\label{eq:defGamma} U_{j_1,Z_1}\cup U_{j_2,Z_2}\cup \cdots \cup U_{j_r,Z_r}\quad (j_1,\dots,j_r>0,\,
		Z_1,\dots,Z_r\in \pi_0(M)),
		\ee which are directed by inclusions.
		For every $U\in \Gamma$  as in \eqref{eq:defGamma},  take a formal function $f_{U}$ on $M$ such that
		\[\mathrm{supp}\,f_{U}\subset U\quad \text{and}\quad  f_{U}|_
		{\overline{U_{j_1-1,Z_1}}\cup \overline{U_{j_2-1,Z_2}}\cup \cdots \cup \overline{U_{j_r-1,Z_r}}}=1.\]
		The assertion then follows by the fact that for every $u\in \CF(M)$,
		the net $\{f_U u\}_{U\in \Gamma}$ in $\CF_c(M)$ converges to $u$.
	\end{proof}

	By Lemma \ref{lem:FcdenseinF}, the  transpose 
	\be \label{eq:FM'toDF}
	\CL_b(\CF(M),E)\rightarrow \mathrm{D}^{-\infty}(M;\CF;E)=\CL_b(\CF_c(M),E)
	\ee
	of the inclusion $\CF_c(M)\rightarrow \CF(M)$ is injective.
	
	\begin{lemd}\label{lem:DFcMEinLFME}The space  $\mathrm{D}^{-\infty}_c(M;\CF;E)$ is contained in the image of \eqref{eq:FM'toDF}.
	\end{lemd}
	\begin{proof} Let $\eta$ be an element in $\mathrm{D}^{-\infty}_K(M;\CF;E)$, where   $K$ is a compact subset of $M$.
		Extend $\eta$ to a linear map  $\eta':\CF(M)\rightarrow E$ by setting
		\[\la \eta', u\ra:=\la \eta, f u\ra\quad \text{for all}\ u\in \CF(M),\]
		where $f$ is  a compactly supported formal function  on $M$ such that $f|_K=1$. We only need to show that $\eta'$ is continuous. 
		
		Let  $\abs{\,\cdot\,}_\nu$  be a continuous seminorm on $E$.
		Since \[\eta|_{\CF_{\mathrm{supp}\,f}(M)}:\CF_{\mathrm{supp}\,f}(M)\rightarrow E\] is  continuous,
		there is a $D\in \mathrm{Diff}_c(\CF,\underline{\CO})$ such that
		\[|\la \eta,v\ra|_\nu\le |v|_D\quad\text{for all}\ v\in \CF_{\mathrm{supp}\,f}(M).\]
		Using this, we find that
		\[|\la \eta' , u\ra|_\nu=|\la \eta, f u\ra|_\nu\le |fu|_D=|u|_{D\circ f}\quad \text{for all}\ u\in \CF(M).\] Here $D\circ f\in \mathrm{Diff}_c(\CF,\underline{\CO}) $ is as in \eqref{eq:defofDf}.
		Thus $\eta'$ is contained in  $\CL(\CF(M),E)$, as required.
	\end{proof}
	
	In  general, the space  $\mathrm{D}^{-\infty}_c(M;\CF;E)$ may not be  the image of \eqref{eq:FM'toDF} (even when $M$ is reduced).
	However,  we still have the following partial result.
	Recall from \cite[\S\,40.2]{Ko2} that a linear map $\phi:E_1\rightarrow E_2$ of LCS is called bounded if
	there is a  neighborhood $U$ of $0$ in $E_1$ such that $\phi(U)$ is bounded in $E_2$.
	
	\begin{lemd}\label{lem:charDFcE} If every element in $\CL(\CF(M),E)$ is bounded, then
		the image of \eqref{eq:FM'toDF} is exactly $\mathrm{D}^{-\infty}_c(M;\CF;E)$.
	\end{lemd}
	\begin{proof}
		Let $\phi$ be an element in $\CL(\CF(M),E)$. Since $\phi$ is bounded,
		there exist $\varepsilon>0$ and $\{D_i\in \mathrm{Diff}_c(\CF,\underline{\CO})\}_{1\leq i\leq n}$ for some $n\in \BN$  such that the set
		\[B:= \left\{\la \phi,v\ra\mid v\in \CF(M)\ \text{with}\ \sum_{1\leq i\leq n}|v|_{D_i}<\varepsilon\right\}\]
		is bounded in $E$.
		Let $u$ be an element in $\CF_c(M)$ supported in $M\setminus(\cup_{1\leq i\leq n}\mathrm{supp}\,D_i)$.
		Then $\la \phi,\lambda u\ra$ lies in $B$ for all $\lambda\in\BC$.
		This forces that  $\la \phi,u\ra=0$, and so $\phi$ is supported in $\cup_{1\leq i\leq n}\mathrm{supp}\,D_i$, as required.
	\end{proof}
	
	\begin{cord}\label{cor:charDFc} When $E=\C$, the image of \eqref{eq:FM'toDF} is exactly $\mathrm{D}^{-\infty}_c(M;\CF)$.
	\end{cord}
	
	It is known that every continuous linear map from a Fr\'echet (resp.\,DF)  space to a DF (resp.\,Fr\'echet) space  is bounded (see \cite[Page 62, footnote (2)]{Sc3} and \cite[\S\,40.2 (9)]{Ko2}).

	\begin{cord}
		Assume that $\CF(M)$ is a Fr\'echet (resp.\,DF) space and $E$ is a DF (resp.\,Fr\'echet) space. Then the image of \eqref{eq:FM'toDF} is exactly $\mathrm{D}^{-\infty}_c(M;\CF;E)$.
	\end{cord}
	
	Lemma \ref{lem:DFcMEinLFME} implies that there is a natural injective linear map 
	\be\label{eq:DFcMEtoLFME}\iota_E:\quad \mathrm{D}^{-\infty}_c(M;\CF;E)\rightarrow \CL_b(\CF(M),E).\ee
	The following result implies that \eqref{eq:DFcMEtoLFME} is  continuous, at least when $\CF(M)$ is
	Hausdorff.
	
	\begin{lemd}\label{CMK'=DFKM} Let $K$ be a compact subset of $M$.
		If $\CF(M)$ is
		Hausdorff,
		then the injective linear map
		\[
		\iota_{E}|_{ \mathrm{D}_K^{-\infty}(M;\CF;E)}:\ \mathrm{D}^{-\infty}_K(M;\CF;E)\rightarrow  \CL_b(\CF(M),E)
		\]
		induced by \eqref{eq:DFcMEtoLFME}
		is a topological embedding.
	\end{lemd}
	\begin{proof} Let $B$ be a bounded subset of $\CF_c(M)$ and let $\abs{\,\cdot\,}_\nu$ be a continuous seminorm on $E$.
		Since $\CF(M)$ is Hausdorff, it follows from  Lemma \ref{lem:FcMhau+regular} that
		$B$ is also bounded in  $\CF(M)$. Equip \[\iota_E(\mathrm{D}^{-\infty}_K(M;\CF;E))\subset \CL_b(\CF(M),E)\]  with the subspace topology.
		Then viewing as a seminorm on $\iota_E(\mathrm{D}^{-\infty}_K(M;\CF;E))$,  $\abs{\,\cdot\,}_{B,\abs{\,\cdot\,}_\nu}$ (see \eqref{eq:defabsBnu})
		is still continuous.
		This shows that the linear map \[\mathrm{D}^{-\infty}_K(M;\CF;E)\rightarrow  \iota_E(\mathrm{D}^{-\infty}_K(M;\CF;E))\] is open.
		
		On the other hand, let $B'$ be a bounded subset of $\CF(M)$  and let $\abs{\,\cdot\,}_\nu$ be a continuous seminorm on $E$.
		Take a compactly supported formal function $f$ on $M$ such that $f|_K=1$.
		Then we have a  bounded subset 
		\[fB':=\{fu\mid u\in B'\}\] of $\CF_c(M)$. Moreover, the following equality holds: 
		\[|\iota_{E}(\eta)|_{B',\abs{\,\cdot\,}_\nu}=|\eta|_{fB',\abs{\,\cdot\,}_\nu}\quad \text{for all}\ \eta\in\mathrm{D}^{-\infty}_K(M;\CF;E).\]
		This shows that $\iota_{E}$ is continuous, and we complete the proof.
	\end{proof}

	When $E=\C$, Corollary \ref{cor:charDFc} says that the map 
	\eqref{eq:DFcMEtoLFME} is a linear isomorphism.  
	Now we are going to give a sufficient condition for this map to be a topological linear isomorphism (\cf\,\cite[Pages 89-90]{Sc2}).

	\begin{lemd}\label{lem:iotaCtopiso} Assume that $\CF(M)$ is barreled 
		and $\CF(M)_b'$ is bornological.
		Then  the continuous map
		\[\iota_\C:\quad  \mathrm{D}^{-\infty}_c(M;\CF)\rightarrow\CF(M)_b' \]
		is a topological linear isomorphism.
	\end{lemd}
	\begin{proof}
		By Corollary \ref{cor:charDFc} and Lemma \ref{CMK'=DFKM}, it suffices to prove that 
		\[
		\CF(M)'_b=\varinjlim_{K}\iota_\C(\mathrm{D}^{-\infty}_K(M;\CF))
		\]
		as LCS, where $\iota_\C(\mathrm{D}^{-\infty}_K(M;\CF))$ is equipped with the subspace topology. 
		This is equivalent to prove that for every  seminorm $\abs{\,\cdot\,}$  on $\CF(M)'_b$, if its 
		restriction  on $\iota_\C(\mathrm{D}^{-\infty}_K(M;\CF))$ is continuous for all compact subsets $K$ of $M$, then $\abs{\,\cdot\,}$ is continuous on $\CF(M)'_b$.
		
		Let $B$ be a bounded subset in $\CF(M)'_b$.
		Since $\CF(M)$ is barreled, $B$ is an equicontinuous set in $\CF(M)'_b$ (see \cite[Theorem 33.2]{Tr}).
		Thus there exist $D_1,D_2,\dots ,D_n\in \mathrm{Diff}_c(\CF,\underline{\CO})$ and $\varepsilon>0$ such that for all $ v\in \CF(M)$, if $|v|_{D_i}<\varepsilon$ for all  $1\leq i\leq n$, then
		\[ |\la \phi, v\ra|\le 1\quad \text{for all}\ \phi\in B.\]
		Let $\phi\in B$, and let $u$ be an element in $\CF(M)$ supported in $M\setminus \cup_{1\leq i\leq n}\mathrm{supp}\,D_i$.
		Then we have that  $|\la \phi,\lambda u\ra|\le 1$ for all $\lambda\in \C$. 
		This implies that $\la\phi,u\ra=0$, and hence \[\phi\in  \iota_\C(\mathrm{D}^{-\infty}_{\cup_{1\leq i\leq n}\mathrm{supp}\,D_i}(M;\CF)).\]
		Due to the assumption that $\abs{\,\cdot\,}$ is continuous on $\iota_\C(\mathrm{D}^{-\infty}_{\cup_{1\leq i\leq n}\mathrm{supp}\,D_i}(M;\CF))$,
		we see that $\abs{\,\cdot\,}$ sends $B$ into a bounded set.
		This, together with the fact that $\CF(M)'_b$ is bornological, implies that  $\abs{\,\cdot\,}$ is continuous on $\CF(M)'_b$, as required.
	\end{proof}

	\begin{prpd}\label{prop:DFcM=FM'} Assume that $\CF$ is locally free of finite rank.
		Then  \be\label{eq:DFcM=FM'}
		\mathrm{D}^{-\infty}_c(M;\CF)=\CF(M)_b'\ee
		as LCS.
	\end{prpd}
	\begin{proof}
		Recall that, in this case, $\CF(M)$ is a product of nuclear Fr\'echet spaces (see \cite[Corollary 4.7]{CSW}). 
		Then Lemma \ref{lem:bornological} implies that  $\CF(M)$ is barreled and 
		$\CF(M)'_b$ is bornological.
		The proposition then follows from  Lemma \ref{lem:iotaCtopiso}.
	\end{proof}

	\begin{cord}\label{cord:Dc-infynDF}
		Assume that $\CF$ is locally free of finite rank. Then 
		$\RD_c^{-\infty}(M;\CF)$ is a direct sum of complete nuclear DF spaces. In particular, if $M$ is secondly countable, then $\RD_c^{-\infty}(M;\CF)$ is a complete nuclear DF space. 
	\end{cord}
	\begin{proof}
		The assertion follows from 
		\eqref{eq:DFcM=FM'}, \eqref{eq:sdualofprodandsum}, \cite[Corollary 4.7]{CSW} and Lemma \ref{lem:dualofDF}.
	\end{proof}

	With the  transposes of the restriction maps in $\CF$ as extension maps, the assignment 
	\be\label{eq:CF'sheaf} U\mapsto \CF(U)_b'\qquad (U\ \text{is an open subset of}\ M) \ee 
	forms a precosheaf of $\CO$-modules over $M$  (see \eqref{eq:actiononF'}).

	\begin{cord}\label{cor:CF'cosheaf} Assume that $\CF$ is  locally free  of finite rank. 
		Then  the precosheaf  \eqref{eq:CF'sheaf} coincides with 
		$\mathrm{D}_c^{-\infty}(\CF)$ and hence is a cosheaf.
		Furthermore, $\RD_c^{\infty}(\CF)$ is naturally identified with a subcosheaf of $\RD_c^{-\infty}(\CF)$. 
		
	\end{cord}
	\begin{proof} 
		The assertion follows from the fact that, via the identification \eqref{eq:DFcM=FM'}, 
		the extension maps of  the precosheaf \eqref{eq:CF'sheaf} coincide with that of $\mathrm{D}^{-\infty}_c(\CF;E)$.
	\end{proof}

	\begin{prpd}\label{prop:desDcMO}
		Assume that $M=N^{(k)}$ with $N$ a smooth manifold and $k\in \BN$. Then 
		\[ 
		\RD_c^{-\infty}(M;\CO)=\RD_c^{-\infty}(N)\widetilde\otimes_{\mathrm{i}}
		[y_1^*,y_2^*,\dots,y_k^*]=\RD_c^{-\infty}(N)
		\widehat\otimes_{\mathrm{i}} \C[y_1^*,y_2^*,\dots,y_k^*]
		\] as LCS. 
	\end{prpd}
	\begin{proof}
		By \eqref{eq:DFcM=FM'} and Lemma \ref{lem:dualofprodFot}, we have that 
		\[\RD_c^{-\infty}(M;\CO)=(\RC^\infty(N)\widehat\otimes \C[[y_1,y_2,\dots,y_k]])_b'=\RD_c^{-\infty}(N)
		\widehat\otimes_{\mathrm{i}} \C[y_1^*,y_2^*,\dots,y_k^*].\] Then the proposition follows from Example \ref{ex:toponpolynomial}.
	\end{proof}

	Let $n\in \BN$. It is known that  $\CL_b(\RC^\infty(\R^n),E)$ is topologically isomorphic to 
	$
	\RC^{\infty}(\R^n)_b'\widetilde\otimes E
	$, and it is also topologically isomorphic to $\RC^{\infty}(\R^n)_b'\widehat\otimes E$ if $E$ is complete (see \cite[Page 55, Examples]{Sc3}).
	We have the following generalizations.
	
	\begin{prpd}\label{prop:charLFME}
		Assume that $M$ is secondly countable, and  $\CF$ is locally free of finite rank. Then
		\[\CF(M)_b'\widetilde\otimes E=\CL_b(\CF(M),E)\] as
		LCS. And, if $E$ is complete, then
		\[\CF(M)_b'\widetilde\otimes E=\CL_b(\CF(M),E)=\CF(M)_b'\widehat\otimes E\] as LCS. 
	\end{prpd}
	\begin{proof} Note that in this case, $\CF(M)$ is a nuclear Fr\'echet space (see \cite[Corollary 4.7]{CSW}).
		By %In view of 
		Corollary \ref{cor:conmaps=tensorproduct}, it suffices to prove that $\CF(M)$ has the strict approximation property.
		
		Let $\{U_i\}_{i\in \BN}$ be a cover of $M$ by relatively compact open subsets of $M$ such that for every $i\in \BN$,
		$\overline{U_i}\subset U_{i+1}$.
		Let $f_i$ be a formal function supported in $U_{i+1}$ with  $f_i|_{\overline{U_i}}=1$.
		Set
		\[
		m_{f_i}:\ \CF(M)\rightarrow \CF_c(M),\quad u\mapsto f_iu,
		\]
		which is  a continuous linear map from $\CF(M)$ to $\CF_c(M)$.
		Identify  $\CL(\CF(M),\CF_c(M))$ as a subspace of $\CL(\CF(M),\CF(M))$ through the inclusion $\CF_c(M)\subset \CF(M)$.
		Then it  is obvious that the sequence 
		\[\{m_{f_i}\}_{i\in \BN}\] converges to the identity operator in $\CL_b(\CF(M),\CF(M))$.
		This, together with the fact that $\CF_c(M)$ has the strict approximation property (see Lemma \ref{lem:CFcMappro}), implies that $\CF(M)$ also has
		the strict approximation property (see \cite[Pr\'eliminaires, Proposition 2]{Sc3}).
	\end{proof}

	Also, we have the following result, which generalizes the identification \eqref{eq:charcsEdN}.
	
	\begin{prpd}\label{prop:DFE=LME} Assume that $M$ is secondly countable, $\CF$ is locally free of finite rank, and
		$E$ is a DF space. 
		Then  $\mathrm{D}^{-\infty}_c(M;\CF;E)$ is a complete  DF space, and the following LCS identifications hold:
		\be\mathrm{D}^{-\infty}_c(M;\CF)\wt\otimes E=\mathrm{D}^{-\infty}_c(M;\CF)\wh\otimes E=\mathrm{D}^{-\infty}_c(M;\CF;E)=\CL_b(\CF(M),E).\ee
	\end{prpd}
	
	\begin{proof} Recall  that for every compact subset $K$  of $M$,  $\mathrm{D}^{-\infty}_K(M;\CF;E)$ is closed in $\mathrm{D}^{-\infty}(M;\CF;E)$ (see Lemma \ref{lem:closedDisK}).
		It then follows from Proposition \ref{prop:charDFME} that $\mathrm{D}^{-\infty}_K(M;\CF;E)$ is a   complete   LCS. 
		%As a countably strict inductive limit of complete   LCS, 
		By  Lemma \ref{lem:basicsonindlim} (d), 
		$\mathrm{D}^{-\infty}_c(M;\CF;E)$ is also a  complete LCS.
		
		Note that
		\[\mathrm{D}^{-\infty}_K(M;\CF)\wt\otimes E=\mathrm{D}^{-\infty}_K(M;\CF;E)=\mathrm{D}^{-\infty}_K(M;\CF)\wh\otimes E
		\] by  Proposition \ref{prop:charDFME}.
		This implies that 
		\[
		\varinjlim_K (\mathrm{D}^{-\infty}_K(M;\CF)\widetilde\otimes E)=\varinjlim_K (\mathrm{D}^{-\infty}_K(M;\CF)\widehat\otimes E)
		=\mathrm{D}^{-\infty}_c(M;\CF;E)
		\]
		as LCS. %In this proof, $K$ runs over all compact subsets of $M$. 
		Consequently,  the inductive limit 
		$\varinjlim_K (\mathrm{D}_K^{-\infty}(M;\CF)\widetilde\otimes E)$ is quasi-complete, and the inductive limit 
		$\varinjlim_K (\mathrm{D}_K^{-\infty}(M;\CF)\widehat\otimes E)$
		is complete. 
		Then it follows from Lemma \ref{lem:indlimcommprotensorpre}
		that the canonical maps  
		\[\varinjlim_K (\mathrm{D}^{-\infty}_K(M;\CF)\widetilde\otimes E)
		\rightarrow (\varinjlim_K \mathrm{D}^{-\infty}_K(M;\CF))\widetilde\otimes E
		=\mathrm{D}^{-\infty}_c(M;\CF)\widetilde\otimes E\]
		and 
		\[\varinjlim_K (\mathrm{D}^{-\infty}_K(M;\CF)\widehat\otimes E)
		\rightarrow (\varinjlim_K \mathrm{D}^{-\infty}_K(M;\CF))\widehat\otimes E=
		\mathrm{D}^{-\infty}_c(M;\CF)\widehat\otimes E\]
		are both topological linear isomorphisms. Thus
		\[\mathrm{D}^{-\infty}_c(M;\CF)\wt\otimes E=\mathrm{D}^{-\infty}_c(M;\CF)\wh\otimes E=\mathrm{D}^{-\infty}_c(M;\CF;E)\]
		as LCS.  
		Since the completed projective tensor product of two DF spaces is still a DF space (see \cite[\S\,41.4 (7)]{Ko2}), 
		$\mathrm{D}^{-\infty}_c(M;\CF;E)$ is a complete  DF space by Corollary \ref{cord:Dc-infynDF}.
		
		By using 
		Propositions \ref{prop:DFcM=FM'} and \ref{prop:charLFME}, we conclude that 
		\[\mathrm{D}^{-\infty}_c(M;\CF;E)=\mathrm{D}^{-\infty}_c(M;\CF)\wh\otimes E=\CF(M)_b'\widehat\otimes E
		=\CL_b(\CF(M),E)\]
		as LCS. This completes the proof.
	\end{proof}

	Without assuming second countability in Proposition \ref{prop:DFE=LME},
	we still have  the following result by replacing the 
	projective tensor products by inductive tensor products.
	
	\begin{prpd} \label{prop:charDFcME}
		Assume that  $\CF$ is locally free of finite rank, and $E$ is a barreled DF space.
		Then $\mathrm{D}^{-\infty}_c(M;\CF;E)$ is a direct sum of complete barreled DF spaces, and 
		\[\mathrm{D}^{-\infty}_c(M;\CF)\wt\otimes_{\mathrm{i}} E=\mathrm{D}^{-\infty}_c(M;\CF;E)=\mathrm{D}^{-\infty}_c(M;\CF)\wh\otimes_{\mathrm{i}} E\]
		as LCS.
	\end{prpd}
	\begin{proof}
		Recall that for two complete barreled DF spaces, their completed projective tensor product is a complete barreled DF space, and coincides with their completed inductive tensor product  (see \cite[\S\,40.2 (11) and \S\,41.4 (7)]{Ko2}). Then 
		for every $Z\in \pi_0(M)$, it follows from Proposition 
		\ref{prop:DFE=LME} and Lemma \ref{lem:bornological} that 
		\begin{eqnarray*}
			&&\mathrm{D}^{-\infty}_c(Z;\CF)\wt\otimes_{\mathrm{i}} E=
			\mathrm{D}^{-\infty}_c(Z;\CF)\wt\otimes E=
			\mathrm{D}^{-\infty}_c(Z;\CF;E)\\
			&=&\mathrm{D}^{-\infty}_c(Z;\CF)\wh\otimes E=
			\mathrm{D}^{-\infty}_c(Z;\CF)\wh\otimes_{\mathrm{i}} E,
		\end{eqnarray*}
		which are all complete barreled DF spaces. 
		This, together with  Corollary \ref{cor:Dc-infty=sum} 
		and \eqref{eq:sumcomiten}, shows that 
		\begin{eqnarray*}
			&&\mathrm{D}^{-\infty}_c(M;\CF)\wt\otimes_{\mathrm{i}} E
			=\left(\bigoplus_{Z\in \pi_0(M)}\mathrm{D}^{-\infty}_c(Z;\CF)\right)\wt\otimes_{\mathrm{i}} E
			=\bigoplus_{Z\in \pi_0(M)}\left(\mathrm{D}^{-\infty}_c(Z;\CF)\wt\otimes_{\mathrm{i}} E\right)\\&=&\bigoplus_{Z\in \pi_0(M)} \mathrm{D}^{-\infty}_c(Z;\CF;E)
			=\mathrm{D}^{-\infty}_c(M;\CF;E)=
			\bigoplus_{Z\in \pi_0(M)} \mathrm{D}^{-\infty}_c(Z;\CF;E)\\&=&\bigoplus_{Z\in \pi_0(M)}\left(\mathrm{D}^{-\infty}_c(Z;\CF)\wh\otimes_{\mathrm{i}} E\right)
			=\left(\bigoplus_{Z\in \pi_0(M)}\mathrm{D}^{-\infty}_c(Z;\CF)\right)\wh\otimes_{\mathrm{i}} E=\mathrm{D}^{-\infty}_c(M;\CF)\wh\otimes_{\mathrm{i}} E,
		\end{eqnarray*} 
		as required. 
	\end{proof}

	\subsection{Formal distributions supported at a point} 
	Let $a\in M$. Write  \[
	\mathrm{Dist}_a(M):= \mathrm{D}^{-\infty}_{\{a\}}(M;\CO),\]
	%for the LCS of all  formal distributions on $M$ supported in $\{a\}$, 
	which  is equipped with 
	the subspace topology of $ \mathrm{D}^{-\infty}(M;\CO)$.
	On the other hand, we equip the space of polynomials with the usual 
	inductive limit topology  (see Example \ref{ex:E-valuedpowerseries}).
	
	The main goal of this subsection is to prove the following description of the LCS $\mathrm{Dist}_a(M)$ (\cf \cite[Theorem 26.4]{Tr}).
	\begin{prpd}\label{prop:distiso} For every $a\in M$, we have that 
		\[\mathrm{Dist}_a(M)\cong \C[x_1^*,x_2^*,\dots,x_n^*,y_1^*,y_2^*,\dots,y_k^*]\]
		as LCS, where $n:=\dim_a(M)$ and $k:=\deg_a(M)$.
	\end{prpd}
	In view of Proposition \ref{prop:DFcM=FM'},
	we identify   $\mathrm{Dist}_{a}(M)$ 
	as a subspace of $(\CO(M))_b'$.
	Recall the continuous character $\mathrm{Ev}_a\in (\CO(M))_b'$ defined in \cite[(5.1)]{CSW}. 
	%Then we have the following consequence of Proposition \ref{prop:distiso}.
	
	\begin{exampled}\label{dist} Let $n,k\in \BN$, and
		suppose that  $M=N^{(k)}$ for some  open submanifold $N$ of $\R^n$. Then for every $a\in M$, the family
		\[
		\{\mathrm{Ev}_a \circ(\partial_x^I \partial_y^J)\}_{ I\in \BN^n, J\in \BN^k}
		\]
		forms a basis of the complex-vector space $\mathrm{Dist}_{a}(M)$.
	\end{exampled}
	
	The rest of this subsection is devoted to a proof of 
	Proposition \ref{prop:distiso}.
	We endow the stalk
	\be\label{eq:CO_a}
	\CO_a:=\varinjlim_U \CO(U)
	\ee
	with the inductive limit topology, where $U$ runs over all open neighborhoods of $a$ in $M$.
	Then $\CO_a$ becomes an LCS that may not be Hausdorff.

	Note that every $D\in \mathrm{Diff}_c(\CO,\underline{\CO})$ gives rise to   a seminorm 
	\begin{eqnarray*}
		|\cdot|_{D,a}:\ \CO_a\rightarrow \BR,\quad g\mapsto |D_a(g)(a)|
	\end{eqnarray*}
	on $\CO_a$, where
	$D_a: \CO_a\rightarrow\underline{\CO}_a$ is the  differential operator defined as in \cite[(3.18)]{CSW}.
	
	\begin{lemd}\label{lem:seminormbasis}
		Let $a\in M$. The topology on $\CO_a$ is defined by the family 
		\be\label{seminormoa}
		\{|\cdot|_{D,a}\}_{D\in \mathrm{Diff}_c(\CO,\underline{\CO})}
		\ee
		of seminorms. Moreover, the linear map \be
		\label{eq:OMtoOa}
		\CO(M)\rightarrow \CO_a\ee is  continuous, open and surjective.
	\end{lemd}
	\begin{proof}
		The family \eqref{seminormoa} of seminorms defines another locally convex topology on the underlying vector space of $\CO_a$.  Denote by $\CO_a^\circ$ the resulting LCS. 
		Recall that the canonical map \be\label{eq:Mtoformalstalk}\CO(M)\rightarrow \CO_a^\circ\ee is surjective (see \cite[Corollary 2.12]{CSW}). Note that for every $D\in \mathrm{Diff}_c(\CO, \underline{\CO})$,
		\be\label{eq:inf} |g|_{D,a}=\inf_{f\in \CO(M);f_a=g}|f|_D\qquad  \textrm{for all $g\in \CO_a^\circ$}.\ee
		Thus the topology on $\CO_a^\circ$ agrees with the quotient topology with respect to the map \eqref{eq:Mtoformalstalk}. This implies that \eqref{eq:Mtoformalstalk} is continuous and open. Likewise, for each open neighborhood $U$ of $a$ in $M$, the  linear map $\CO(U)\rightarrow \CO_a^\circ$ is also   continuous, open and surjective. Then it is clear that $\CO_a^\circ=\varinjlim_U \CO(U)$ as LCS. This finishes the proof of the lemma.  
	\end{proof}

	The transpose of 
	\eqref{eq:OMtoOa} yields an  injective continuous linear map
	\be \label{eq:disttocompact}
	(\CO_a)'_b\rightarrow \CO(M)_b'=\mathrm{D}_c^{-\infty}(M;\CO)\quad\text{(see Proposition \ref{prop:DFcM=FM'})}.
	\ee  
	The following result is straightforward. 
	\begin{lemd}\label{lem:imoftocomapct}
		The image of \eqref{eq:disttocompact} is exactly the space $\mathrm{Dist}_a(M)$.
	\end{lemd}
	% \begin{proof}
		%     Let $\eta\in \mathrm{Dist}_a(M)$. Note that $\la \eta, f\ra=0$ for every $f\in \CO(M)$ with $f_a=0$, where $f_a$ is the germ of $f$ at $a$. Then it induces a continuous map $\eta':\ (\CO_a)'_b\rightarrow \C$ such that the diagram 
		%     \[\xymatrix{\CO(M)\ar[d]\ar[r]^{\eta}&\BC\\\CO_a\ar[ur]_{\eta'}}\]
		%     commutes. 
		%    Consequently, the map \eqref{eq:disttocompact} maps $\eta'$ to $\eta$. This finishes the proof.  
		% \end{proof}

	Recall from \cite[Section 2.3]{CSW} that $\wh{\CO}_a$ is a Hausdorff LCS under the  projective limit topology.
	The following result arrests that $\wh{\CO}_a$ is precisely the Hausdorff LCS associated to $\CO_a$, namely \[
	\CO_a/\overline{\{0\}}=\wh\CO_a\] as LCS.
	
	\begin{lemd}\label{lem:Hausstalk} The intersection $\cap_{r\in \BN} \m_a^r$ is the closure of $\{0\}$ in $\CO_a$. Furthermore, 
		the natural homomorphism (see \cite[(2.6)]{CSW})
		\[\CO_a/\cap_{r\in \BN}\m_a^r\rightarrow \wh{\CO}_a\]  is a topological algebra isomorphism.
	\end{lemd}
	\begin{proof} Without loss of generality, we  assume that $M=(\R^n)^{(k)}$ for some $n,k\in \BN$.
		For each $r\in \BN$,  \cite[Lemma 2.5]{CSW} implies  that
		\[\m_a^r=\{g\in\CO_a\mid ((\partial_x^I\partial_y^J)(g))(a)=0\, \,\text{for all } I\in \BN^n, J\in \BN^k \,\, \text{with} \,\, |I|+|J|< r\}.\] According to Lemma \ref{lem:seminormbasis}, the intersection\[\cap_{r\in \BN} \m_a^r=\{g\in\CO_a\mid ((\partial_x^I\partial_y^J)(g))(a)=0\, \,\text{for all } I\in \BN^n, J\in \BN^k\}\] is a closed subspace.
		This, together with Lemma \ref{lem:seminormbasis}, implies that
		$\cap_{r\in \BN} \m_a^r$ is exactly the  the closure of $\{0\}$ in $\CO_a$.
		The second assertion follows from \cite[Proposition 2.6]{CSW} and Lemma \ref{lem:seminormbasis}, by comparing the seminorms that define the two topologies.
	\end{proof}
	
	In view of Lemma \ref{lem:Hausstalk}, there is an LCS identification 
	\[
	\left(\widehat\CO_a\right)'_b=(\CO_a)'_b.
	\] 
	Together with 
	Lemma \ref{lem:imoftocomapct}, we get a continuous linear isomorphism \be\label{eq:dualtodist} \left(\widehat\CO_a\right)'_b\rightarrow \mathrm{Dist}_a(M) . \ee
	\begin{lemd}\label{lem:dual=dist} The map \eqref{eq:dualtodist}
		is an isomorphism between LCS.
	\end{lemd}
	\begin{proof}
		Assume first that  $M$ is connected. Then both $\CO(M)$ and $\wh{\CO}_a$ are nuclear 
		Fr\'echet spaces (see \cite[Propositions 4.8 and 2.6]{CSW}), and hence are Montel spaces  (see \cite[Corollary 3 of Proposition 50.2]{Tr}). On the other hand, Lemmas \ref{lem:seminormbasis} and \ref{lem:Hausstalk} imply that the composition of 
		\be \label{eq:OMtoOatohOa3} 
		\CO(M)\rightarrow \CO_a\rightarrow \wh{\CO}_a
		\ee
		is open and continuous. As a result, it follows from  \cite[\S\,33.6(1)]{Ko2} that the composition  of \[(\wh \CO_a)'_b\rightarrow (\CO_a)'_b\rightarrow (\CO(M))'_b=\mathrm{D}^{-\infty}_c(M;\CO)\] is a linear topological embedding. Consequently, the map \eqref{eq:dualtodist} is an isomorphism of LCS.
		
		In the general case, let $Z$ be the connected component of $M$ that contains $a$. It is clear that the extension map \[\mathrm{ext}_{M,Z}:\ \mathrm{Dist}_a(Z)\rightarrow \mathrm{Dist}_a(M)\] is an isomorphism of LCS by Lemma \ref{eq:KofD-infty}. Note that the map \eqref{eq:dualtodist} coincides with the composition  of \[(\wh\CO_a)'_b\rightarrow \mathrm{Dist}_a(Z)\xrightarrow{\mathrm{ext}_{M,Z}}\mathrm{Dist}_a(M).\] Then it is an isomorphism of LCS, as desired. 
	\end{proof}

	Proposition \ref{prop:distiso}  follows from Lemma \ref{lem:dual=dist} and \cite[Proposition 2.6]{CSW}.

	\appendix 
	\section{Preliminaries on topological tensor products}\label{appendixA}
	This appendix provides an overview of topological tensor products of LCS, along with related results necessary for the paper.

	\subsection{Topological tensor products of LCS}\label{appendixA2}
	%Here we  introduce %fixour notations on topological tensor products of LCS.
	
	Let $E$ and $F$ be two  LCS. %{\color{cyan}As indicated in \cite{Gr}, there are at least three useful  locally convex topologies one can put on the algebraic tensor product $E\otimes F$. Namely, the inductive tensor product $E\otimes_{\mathrm{i}} F$, the projective tensor product $E\otimes_{\pi} F$, and the epsilon (injective) tensor product $E\otimes_{\varepsilon} F$. One may see \cite{Gr} for the explicit definitions.} {\footnotesize
		%As noted in \cite{Gr}, there are at least three valuable locally convex topologies on the algebraic tensor product $E\otimes F$: the inductive tensor product $E\otimes_{\mathrm{i}} F$, the projective tensor product $E\otimes_{\pi} F$, and the epsilon (injective) tensor product $E\otimes_{\varepsilon} F$. Explicit definitions can be found in \cite{Gr}.
		As discussed in \cite{Gr}, the algebraic tensor product $E\otimes F$ admits at least three valuable locally convex topologies:  the inductive tensor product $E\otimes_{\mathrm{i}} F$, the projective tensor product $E\otimes_{\pi} F$, and the epsilon (injective) tensor product $E\otimes_{\varepsilon} F$. Explicit definitions can be found in \cite{Gr}.
		In this context, as mentioned  in the Introduction,  we denote  the quasi-completions and completions of these topological tensor product spaces as follows:
		\[\text{$E\widetilde\otimes_{\mathrm{i}} F$,
			$E\widetilde\otimes_{\pi} F$, $E\widetilde\otimes_{\varepsilon} F$\quad and\quad $E\widehat\otimes_{\mathrm{i}} F$,
			$E\widehat\otimes_{\pi} F$, $E\widehat\otimes_{\varepsilon} F$}.\]
		When $E$ or $F$ is nuclear, it is a classical result of Grothendieck that their projective tensor product coincides with the epsilon tensor product (see \cite[Theorem 50.1]{Tr}).
		In this case, we
		will simply write
		\bee
		E\widetilde\otimes F:=E\widetilde\otimes_{\pi} F=E\widetilde\otimes_{\varepsilon} F
		\quad
		\text{and}
		\quad
		E\widehat\otimes F:=E\widehat\otimes_{\pi} F=E\widehat\otimes_{\varepsilon} F.
		\eee

		Given two continuous linear maps $\phi_1:E_1\rightarrow F_1$ and $\phi_2:E_2\rightarrow F_2$  between  LCS, we use $\phi_1\otimes \phi_2$ to denote the continuous linear map on various topological tensor products (or their quasi-completions or completions) obtained by the tensor product of $\phi_1$ and $\phi_2$. For example, we have the maps \[\phi_1\otimes \phi_2:\ E_1\widehat\otimes_{\pi} E_2\rightarrow F_1\widehat\otimes_{\pi} F_2\quad \text{and}\quad 
		\phi_1\otimes \phi_2:\ E_1\widetilde\otimes_{\mathrm{i}} E_2\rightarrow F_1\widetilde\otimes_{\mathrm{i}} F_2.\]

		%Let  $\phi_1:E_1\rightarrow F_1$ and $\phi_2:E_2\rightarrow F_2$ be  two continuous linear maps  of  LCS.
		%If there is no confusion,we will still denote by $\phi_1\otimes \phi_2$ the various quasi-completed  or completed topological tensor products of $\phi_1$ and $\phi_2$.

		\subsection{Inductive limits of LCS}\label{appendixA3}
		% \begin{dfn}
			%     Let $(I,\le)$ be a preordered set. A directed system in the category of LCS is a family $\{E_i\}_{i\in I}$ of LCS together with a family $\{f_{ij}:E_i\rightarrow E_j\}_{i,j\in I; i\leq j}$  of continuous linear maps such that %with the following property: 
			% \begin{itemize}
				%         \item $f_{ii}$ is the identity map of $E_i$ for all $i\in I$; and 
				%         \item $f_{ik}=f_{jk}\circ f_{ij}$ for all $i,j,k\in I$ with $i\leq j\leq k$.
				%     \end{itemize}
			
			% \end{dfn} 
		
		% The above directed system is denoted by $(E_i,f_{ij})^{(I,\leq)}$, and when  no confusion is possible, we will not distinguish this   directed system with the family $\{E_i\}_{i\in I}$.
		%\begin{dfn}In the category of LCS, an inductive limit of a directed system $(E_i,f_{ij})^{(I)}$  is an LCS $E$, together with  a family of continuous linear maps $\{\varphi_i:E_i\rightarrow E\}_{i\in I}$ such that $\varphi_{i}=\varphi_j\circ f_{ij}$ whenever $i\leq j$, with the following universal property: for every LCS $F$ and every family of continuous linear maps $\{_i:E_i\rightarrow F\}_{i\in I}$ such that $\varphi_{i}=\varphi_j\circ f_{ij}$ whenever $i\leq j$\end{dfn}
		For a directed set $(I,\le)$ and 
		a directed system 
		$\{E_i\}_{i\in I}$ of LCS, % we write $\varinjlim_{i\in I}E_i$ for the corresponding unique inductive limit 
		the inductive limit $\varinjlim_{i\in I}E_i$  exists and is unique in the category of LCS. It is topologically linearly  isomorphic to the LCS
		\[\left(\bigoplus_{i\in I} E_i\right)/\left(\sum_{i,j\in I;i \leq j}\{v_i-f_{ij}(v_i):v_i\in E_i\}\right),\]
		which is equipped with the quotient topology. %Here, for an LCS $E$ and a subset $S\subset E$, the subspace $\mathrm{span}(S)$ denotes the linear span of the set $S$ in $E$. 
		See \cite[Page 110]{Ja} for details. %Here, $\bigoplus_{i\in I} E_i$ is endowed with 
		%,where $N\subset\bigoplus_{i\in I} E_i$ is the linear span of the set \[\bigcup_{i,j\in I;i \leq j}\{u_i-f_{ij}(u_i):u_i\in E_i\}.\]
		%Note that $(\bigoplus_{i\in I} E_i)$ is the inductive limit of 
		
		%Here, $\bigoplus_{i\in I}E_i$ is endowed with the topology determined  by 
		% In the rest part of this section, all preordered sets appearing in inductive limits  are assumed to be directed sets. 
		\begin{dfn}\label{de:regualrlim}
			An inductive limit 
			$\varinjlim_{i\in I}E_i$ is said to be strict if 
			for every $i\le j$, the canonical map $E_i\rightarrow E_j$ is a linear topological embedding.
			Furthermore, an inductive limit 
			$\varinjlim_{i\in I}E_i$  is said to be regular if it is strict, and 
			every bounded subset of it is contained and bounded in
			some $E_i$.
		\end{dfn}
		
		We remark that when an inductive limit $\varinjlim_{i\in \BN}E_i$ is strict, the map \[E_j\rightarrow \varinjlim_{i\in \BN}E_i \] is a linear topological embedding for all $j\in \BN$ (see \cite[\S\,4.6 Proposition 9 (i)]{Bo}).
		
		%Throughout this paper, we refer to a strict inductive limit of the form $\varinjlim_{r\in \BN} E_{r}$ as a countably strict inductive limit of LCS. 
		
		\begin{dfn} \label{de:LFspace}
			An LCS is said to be  an LF space  if it is  isomorphic to a %countably
			strict inductive limit of the form 
			$\varinjlim_{r\in \BN} E_{r}$,
			where $E_r$ is a Fr\'echet space for every $r\in \BN$. 
		\end{dfn}

		The inductive limits of LCS  considered  in this paper often have the following form:
		\be \label{eq:gencouindlim}
		E=\bigoplus_{i\in I}E_i\quad (E_i=\varinjlim_{r\in \BN} E_{i,r}\ \text{is a %countably
			strict inductive limit of LCS}).
		\ee
		Here $E$ is regarded as the inductive limit of the family $\{\bigoplus_{j\in J} E_{j,r}\}_{J\in \mathfrak{F}, r\in \BN}$ consisting of finite partial sums, where $\mathfrak{F}$ is the set of all finite subsets of $I$. 
		In the subsequent lemma, we collect some basic properties about these inductive limits. 
		
		\begin{lemd} \label{lem:basicsonindlim}
			Let $E$ be an LCS as in \eqref{eq:gencouindlim}.

			\noindent (a) Let $J$ be a subset of $I$, and $r_j\in \BN$ for all $j\in J$. Then the subspace topology on the partial sum $\bigoplus_{j\in J} E_{j,r_j}$ induced by $E$ coincides with its direct sum topology.

			\noindent (b) If every $E_{i,r}$ is Hausdorff, then $E$ is Hausdorff.

			\noindent (c) If every $E_{i,r}$ is Hausdorff and  closed in $E_{i,r+1}$, then the inductive
			limit \[E=\varinjlim_{J\in \mathfrak{F}, r\in \BN}\left(\bigoplus_{j\in J} E_{j,r}\right)\] is regular, and
			$E_{i,r}$ is closed in $E$.
			
			\noindent (d) If every $E_{i,r}$ is %Hausdorff and 
			complete, then $E$ is complete. 
			
			\noindent (e) If every $E_{i,r}$ is %Hausdorff, 
			quasi-complete and closed in $E_{i,r+1}$, then 
			$E$ is quasi-complete.
		\end{lemd}
		\begin{proof} 
			The  assertion (a) follows from \cite[Chapitre II, \S\,4.5 Proposition 8 (i) and \S\,4.6 Proposition 9 (i)]{Bo}.
			The  assertion (b) is a direct consequence of \cite[Chapitre II, \S\,4.5 Corolliare 2 and \S\,4.6 Proposition 9 (i)]{Bo}.
			The assertion (c)  is a consequence of \cite[Chapitre II, \S\,4.5 Corollaire 2 and \S\,4.6 Proposition 9 (ii)]{Bo} and \cite[Chapitre III, \S\,1.4 Propositions 5 and 6]{Bo}.
			The assertion (d) follows from (b), \cite[Chapitre III, \S\,3.6 Corolliare 2 and Chapitre II, \S\,4.6 Proposition 9 (iii)]{Bo}.
			The assertion (e) is a  consequence of (b), \cite[Chapitre III, \S\,1.6 Proposition 9 (iii) and (iv)]{Bo}.
		\end{proof}
		
		%If  two Hausdorff LCS \[E=\varinjlim_{i\in I} E_i \quad \text{  and  } \quad F=\varinjlim_{j\in J}F_j\] are inductive limits of Hausdorff LCS,  then we have that \be \label{eq:ilimcomtenpre} E \otimes_{\mathrm{i}} F=\varinjlim_{(i,j)\in I\times J} (E_i\otimes_{\mathrm{i}} F_j) \eeas LCS (see \cite[Chapitre I, Proposition 14]{Gr}). Consequently, we have the following result. 
		%It is known that inductive limits commute with inductive tensor products in the category of  Hausdorff LCS(see \cite[Chapitre I, Proposition 14]{Gr}). Namely,  l
		For two  LCS $E=\varinjlim_{i\in I} E_i$ and $F=\varinjlim_{j\in J}F_j$,  it is easy to check that
		\be \label{eq:ilimcomtenpre} E \otimes_{\mathrm{i}} F=\varinjlim_{(i,j)\in I\times J}  (E_i\otimes_{\mathrm{i}} F_j) \ee
		as 
		LCS by using the universal property. Consequently, we have the following result (\cf \cite[Chapitre I, Proposition 14]{Gr}). 
		
		\begin{lemd}\label{lem:limofcomplete} Let $E=\varinjlim_{i\in I} E_i$ and $F=\varinjlim_{j\in J}F_j$ be inductive limits of LCS.   If the
			inductive limit $\varinjlim_{(i,j)\in I\times J}  (E_i\widetilde\otimes_{\mathrm{i}} F_j)$
			is 
			quasi-complete, then  
			\be  \label{eq:ilimcomiten}E \widetilde\otimes_{\mathrm{i}} F=\varinjlim_{(i,j)\in I\times J}  (E_i\widetilde\otimes_{\mathrm{i}} F_j)
			\ee
			as LCS. 
			Similarly, if the inductive limit $\varinjlim_{(i,j)\in I\times J}  (E_i\widehat\otimes_{\mathrm{i}} F_j)$ is complete, then
			\be \label{eq:ilimcomiten1}
			E \widehat\otimes_{\mathrm{i}} F=\varinjlim_{(i,j)\in I\times J}  (E_i\widehat\otimes_{\mathrm{i}} F_j) \ee as LCS.
		\end{lemd}
		The following lemma is a special case of Lemma \ref{lem:limofcomplete}.
		\begin{lemd}
			Let  \[
			E=\bigoplus_{i\in I}E_i\quad \text{and}\quad F=\bigoplus_{j\in J}F_j\] be both direct sums of Hausdorff LCS. Then 
			% we conclude from Lemma \ref{lem:basicsonindlim} that
			\be  \label{eq:sumcomiten}E\widetilde\otimes_{\mathrm{i}}F=\bigoplus_{(i,j)\in I\times J}  (E_i\widetilde\otimes_{\mathrm{i}} F_j)\quad\text{and}\quad 
			E \widehat\otimes_{\mathrm{i}} F=\bigoplus_{(i,j)\in I\times J}  (E_i\widehat\otimes_{\mathrm{i}} F_j)\ee
			as LCS.
		\end{lemd}

		% \begin{dfn}\label{de:dfspace} A Hausdorff LCS is said to be a DF space if 
			%  it has  a fundamental sequence of bounded sets (namely, there exists a countable sequence \[ B_1\subset B_2\subset\cdots \] of bounded subsets 
			%  such that every bounded subset $B$ 
			% is contained in some $B_i$), and
			% every union of countably many equicontinuous sets in its strong dual that is bounded is also equicontinuous.
			% \end{dfn}
		Recall the notion of  DF spaces introduced by  Grothendieck (see \cite[\S\,29.3]{Ko1}). 
		Note that a DF space is quasi-complete if and only if it is complete (see \cite[\S29.5 (3) a]{Ko1}).
		Let $F$ be a DF space, and let $E=\varinjlim_{r\in \BN} E_r$ be an inductive limit  of Hausdorff LCS,  
		it is proved in \cite[Chapitre \uppercase\expandafter{\romannumeral1}, Proposition 6 (2)]{Gr} that if $E$ and $\varinjlim_{r\in \BN}(E_r\otimes_\pi F)$ are Hausdorff, then
		\be 
		\label{eq:indlimcommprotensorpre}E\otimes_\pi F=\varinjlim_{r\in \BN}(E_r\otimes_\pi F)\ee
		as  LCS.  Consequently, we have the following result. 
		
		\begin{lemd}\label{lem:indlimcommprotensorpre}Let the notation be as above. If the inductive limit 
			$\varinjlim_{r\in \BN}(E_r\widetilde\otimes_\pi F)$ is quasi-complete, then
			\be\label{eq:indlimcommprotensor}
			E\widetilde\otimes_\pi F=\varinjlim_{r\in \BN}(E_r\widetilde\otimes_\pi F)
			\ee
			as LCS. Similarly, if the inductive limit $\varinjlim_{r\in \BN}(E_r\widehat\otimes_\pi F)$ is complete, then 
			\be \label{eq:indlimcommprotensor1}
			E\widehat\otimes_\pi F=\varinjlim_{r\in \BN}(E_r\widehat\otimes_\pi F)\ee as LCS.
		\end{lemd}

		\subsection{Strong dual of LCS}\label{appendixA4}
		Let $E$ and $F$ be two LCS. As in the Introduction, we use the notation  $\CL(E,F)$ to denote the space of all continuous linear maps from $E$ to $F$.
		For a continuous linear map $f:E\rightarrow G$,
		let
		\be \label{eq:contrans} {}^tf:\ \CL(G,F)\rightarrow \CL(E,F),\quad u'\mapsto (v\mapsto \la u', f(v)\ra)\ee
		denote the transpose of $f$,  where $G$ is another LCS.
		
		Unless otherwise mentioned, $\CL(E,F)$ is endowed with the strong topology, and the resulting LCS is denoted
		by $\CL_b(E,F)$.
		Let $B$ be a bounded set of $E$, and let $\abs{\,\cdot\,}_\nu$ be a continuous seminorm on $F$. Then
		\be \label{eq:defabsBnu}\abs{\,\cdot\,}_{B,\abs{\,\cdot\,}_\nu}:\ u'\mapsto \sup_{v\in B}|\la u',v\ra|_\nu\qquad (u'\in \CL(E,F))\ee
		is a continuous seminorm on $\CL_b(E,F)$, and the topology on $\CL_b(E,F)$ is defined by all such seminorms.
		When $F=\C$, we also write
		\[E':=\CL(E,\C),\quad E'_b:=\CL_b(E,\C)\quad\text{and}\quad \abs{\,\cdot\,}_{B}:=\abs{\,\cdot\,}_{B,\abs{\,\cdot\,}}.\]
		
		\begin{lemd}\label{lem:dualofDF}
			The strong dual of a metrizable LCS is a complete DF space (see \cite[\S\,29.3]{Ko1}), and the strong dual of a DF space is a
			Fr\'echet space (see \cite[\S\,29.3 (1)]{Ko1}).
		\end{lemd}

		The following result is about the strong dual of a product and  a direct sum of LCS, 
		whose proof can be found in \cite[Chapitre \uppercase\expandafter{\romannumeral4}, \S\,1.5]{Bo} for example. 
		
		\begin{lemd}
			Let $\{E_i\}_{i\in I}$ be a family of Hausdorff LCS. Then 
			\be \label{eq:sdualofprodandsum}\left(\prod_{i\in I}E_i\right)_b'=\bigoplus_{i\in I}(E_i)_b'\quad \text{and}\quad \left(\bigoplus_{i\in I}E_i\right)_b'=\prod_{i\in I}(E_i)_b'
			\ee
			as LCS.
		\end{lemd}
		As an abstract vector space, the continuous dual of an inductive limit $F=\varinjlim_{j\in J} F_j$ of Hausdorff LCS is canonically 
		isomorphic to the projective limit $\varprojlim_{j\in J} F_j'$ (see \cite[\S\,8.8, Corollary 11]{Ja}). This isomorphism may not be a topological linear isomorphism in general 
		(under the strong topologies). However,
		when the inductive limit $F=\varinjlim_{j\in J} F_j$ is regular, one immediately has the following result. 
		\begin{lemd}
			If $F=\varinjlim_{j\in J} F_j$ is a  regular inductive limit of  Hausdorff LCS, then
			\be \label{eq:sdualofindlim}
			\left(\varinjlim_{j\in J} F_j\right)_b'= 
			\varprojlim_{j\in J} (F_j)_b'
			\ee
			as LCS.
		\end{lemd}
		\subsection{Complete reflexive LCS}\label{appendixA5}

		Recall that an LCS $E$ is called reflexive  if the canonical map 
		$E\rightarrow(E_b')_b'$ is an isomorphism of  LCS. 
		The strong dual of a reflexive LCS is reflexive, a product of reflexive LCS is reflexive, and  a direct sum of reflexive LCS is reflexive (see \cite[\S\,23.5 (9)]{Ko1}). 
		It is known that every nuclear Fr\'echet space, or more generally, every nuclear LF space, is reflexive
		(see \cite[Corollary of Proposition 36.9, Corollary 3 of Proposition 50.2 and 
		Corollaries 1, 3 of Proposition 33.2]{Tr}).
		Thus,  we have the following  examples of complete reflexive  LCS:
		\begin{itemize}
			\item the products  of nuclear Fr\'echet spaces as well as their strong dual; and
			\item the direct sums of nuclear LF spaces as well as their strong dual.
		\end{itemize}
		We would like to point out that various function spaces on formal manifolds
		studied in this paper all have the above forms.
		% \begin{dfn}\label{de:barreled}
			%     An LCS is called barreled if it is Hausdorff and 
			%  every barrel in it is a neighborhood of $0$ (see \cite[\S\,27.1]{Ko1}). 
			% \end{dfn}
		% \begin{dfn}\label{de:bornological}
			% An LCS  is called bornological if it is Hausdorff and each seminorm on it that is bounded on bounded subsets is continuous.
			% \end{dfn}
		
		\begin{lemd}\label{lem:bornological}
			Let $E$ be a product of nuclear Fr\'echet spaces, and let $F$ be a direct sum of nuclear LF spaces. Then $E, E_b', F, F_b'$ are all barreled, and $E_b', F$ are bornological. 
		\end{lemd}
		\begin{proof}
			The first assertion follows from the fact that every reflexive LCS is barreled (see \cite[\S\,23.3 (4)]{Ko1}), and the second one follows from the following facts:
			\begin{itemize}
				\item all Fr\'echet spaces are bornological (\cite[\S\,28.1 (4)]{Ko1});
				\item the strong dual of every reflexive Fr\'echet space is bornological (see \cite[\S\,29.4 (4)]{Ko1}); and
				\item if an inductive limit of bornological LCS is Hausdorff, then it is still bornological (see \cite[\S\,28.4 (1)]{Ko1}).
			\end{itemize}
		\end{proof}
		
		%It is known that every complete nuclear DF space is barreled (see \cite[\S 21.5.4]{Ja}), and hence it is reflexive  (see \cite[Corollary 3 of Proposition 50.2]{Tr}). Since a Fr\'echet space is nuclear if and only its strong dual is nuclear (\cite[Proposition 50.6]{Tr}), we see that the strong dual of a nuclear Fr\'echet space is a complete nuclear DF space, and vice versa. On the other hand,   the strong dual of a reflexive Fr\'echet space is a reflexive DF space, and vice versa. Additionally, for two Fr\'echet  (resp.\, barreled DF) spaces, it is known that their projective tensor product coincides with the inductive tensor product (see \cite[\S\,40.2 (1),(11)]{Ko2}).  By combining the above facts together, one concludes the following result. 
		
		Together with Lemma \ref{lem:bornological}, the following lemma \ref{lem:nFrF} is a direct consequence of the following facts.
		\begin{itemize}
			\item For arbitrary two Fr\'echet  (resp. barreled DF) spaces,  their projective tensor product coincides with their inductive tensor product (see \cite[\S\,40.2 (1),(11)]{Ko2}).
			\item The strong dual of all reflexive Fr\'echet spaces are  reflexive DF spaces, and vice versa (by Lemma \ref{lem:dualofDF}).
			\item  Every complete nuclear DF space is barreled (see \cite[\S 21.5.4]{Ja} and \cite[\S\,27.1 (2)]{Ko1}) and reflexive (see \cite[Corollary of Proposition 36.9 and Corollary 3 of Proposition 50.2]{Tr}). 
			\item The strong dual of all nuclear Fr\'echet spaces are  complete nuclear DF spaces, and vice versa (see (\cite[Proposition 50.6]{Tr})).
			\item   Every reflexive LCS is barreled (see \cite[\S\,23.3 (4)]{Ko1}).

		\end{itemize}
		
		\begin{lemd}\label{lem:nFrF}
			Let $E$ be a nuclear Fr\'echet 
			(resp.\,complete nuclear DF) space, and let
			$F$ be a reflexive Fr\'echet (resp.\,reflexive DF) space. Then we have  the following LCS identifications:
			\be \label{eq:threetopontenprod=}
			E\otimes_{\mathrm{i}} F
			=E\otimes_\pi F=E\otimes_\varepsilon F
			\quad\text{and}\quad 
			E_b'\otimes_{\mathrm{i}} F_b'=E_b'\otimes_{\pi} F_b'
			=E_b'\otimes_{\varepsilon} F_b'.
			\ee
		\end{lemd}
		
		\subsection{Strong dual of topological tensor products}\label{appendixA6}
		Let $E$ and $F$ be two Fr\'echet spaces. It is known that, if  $E$ or $F$ is nuclear, then  
		\be \label{eq:dualprincipal}
		(E\widehat\otimes F)_b'=
		E_b'\widehat\otimes F_b'
		\ee
		as LCS (see \cite[Proposition 50.7]{Tr}).
		We finish Appendix \ref{appendixA} by presenting three simple generalizations of \eqref{eq:dualprincipal} as follows.

		\begin{lemd}\label{lem:dualofprodFot} Let $E=\prod_{i\in I} E_i$ be a product of nuclear Fr\'echet spaces, and  let
			$F=\prod_{j\in J}F_j$ be a product of reflexive Fr\'echet spaces. Then $E\widehat\otimes F=\prod_{(i,j)\in I\times J}
			E_i\widehat\otimes F_j
			$ is a product of  Fr\'echet spaces, and
			\be
			(E\widehat\otimes F)_b'=E_b'\widehat\otimes_{\mathrm{i}} F_b'
			\ee
			as LCS.
		\end{lemd}
		\begin{proof} 
			Since the completed  projective tensor product of arbitrary two Fr\'echet spaces is also a Fr\'echet space, the first assertion follows by \cite[Lemma 6.6]{CSW}. 
			
			For the second one, we have that 
			\begin{eqnarray*}
				&& (E\widehat\otimes F)_b'\\
				&=&
				\left(\left(\prod_{i\in I} E_i\right)\widehat\otimes \left(\prod_{j\in J} F_j\right) \right)_b'
				= \left(\prod_{(i,j)\in I\times J} E_i\widehat\otimes F_j \right)_b' \quad(\text{by \cite[Lemma 6.6]{CSW}})\\
				& =&\bigoplus_{(i,j)\in I\times J}
				(E_i\widehat\otimes F_j)_b'
				=\bigoplus_{(i,j)\in I\times J}
				(E_i)_b'\widehat\otimes (F_j)_b'\qquad \qquad\quad\, \,\, (\text{by}\ \eqref{eq:sdualofprodandsum}\ \text{and}\  \eqref{eq:dualprincipal})\\
				&=&\bigoplus_{(i,j)\in I\times J}(E_i)_b'\widehat\otimes_{\mathrm{i}} (F_j)_b'
				=\left(\bigoplus_{i\in I} (E_i)_b'\right)
				\widehat\otimes_{\mathrm{i}} \left(\bigoplus_{j\in J}(F_j)_b'\right)\ 
				(\text{by\ \eqref{eq:threetopontenprod=}\ and\ \eqref{eq:sumcomiten}})\\
				&=&
				\left(\prod_{i\in I} E_i\right)_b'\widehat\otimes_{\mathrm{i}} \left(\prod_{j\in J} F_j\right)_b'
				=E_b'\widehat\otimes_{\mathrm{i}} F_b'\qquad\qquad\quad\qquad\, \, \, (\text{by}\ 
				\eqref{eq:sdualofprodandsum}).
			\end{eqnarray*}
		\end{proof}

		\begin{lemd} Let $E=\bigoplus_{i\in I} E_i$ and 
			$F=\bigoplus_{j\in J} F_j$ be direct sums of complete nuclear DF spaces. Then 
			$E\widehat\otimes_{\mathrm{i}} F=\bigoplus_{(i,j)\in I\times J} E_i\wh\otimes_{\mathrm{i}} F_j$ is a direct sum of complete nuclear
			DF spaces, and
			\be
			(E\widehat\otimes_{\mathrm{i}} F)_b'=E_b'\widehat\otimes F_b'
			\ee
			as LCS.
		\end{lemd}
		\begin{proof} 
			Since the  projective tensor product of arbitrary two nuclear DF spaces is a nuclear DF space (see \cite[Proposition 50.1 (9)]{Tr} and \cite[\S\,41.4\,(7)]{Ko2}),
			it follows that  $E_i\wh\otimes F_j=E_i\wh\otimes_{\mathrm{i}}F_j$ is a complete nuclear  DF space.
			Then by \eqref{eq:sumcomiten}, \[E\widehat\otimes_{\mathrm{i}} F=\bigoplus_{(i,j)\in I\times J} E_i\wh\otimes_{\mathrm{i}} F_j\] is a direct sum of complete nuclear DF spaces. And, by applying Lemma \ref{lem:dualofprodFot}, we have that 
			\[(E\widehat\otimes_{\mathrm{i}} F)_b'=((E_b')_b'\widehat\otimes_{\mathrm{i}} (F_b')_b')_b'=((E_b'\widehat\otimes F_b')_b')_b'=
			E_b'\widehat\otimes F_b'\]
			as LCS. 
		\end{proof}
		
		\begin{lemd}\label{lem:indtensorLF}
			Let $E=\bigoplus_{i\in I}E_i$  and $F= \bigoplus_{j\in J}F_j$ be two direct sums of LCS, where 
			$E_i=\varinjlim_{r\in \BN}E_{i,r}$ is a nuclear LF space for every $i\in I$, and $F_j=\varinjlim_{r\in \BN}F_{j,r}$ is an LF space for every $j\in J$. Then  
			\be \label{eq:indtensorLF}
			E\widehat\otimes_{\mathrm{i}} F=\bigoplus_{(i,j)\in I\times J}
			\left(\varinjlim_{r\in \BN} (E_{i,r}\widehat\otimes F_{j,r})\right)
			\ee
			is a direct sum of %nuclear 
			LF spaces, and 
			\be (E\widehat\otimes_{\mathrm{i}} F)_b'=E'_b\widehat\otimes F'_b\ee
			as LCS.
		\end{lemd}
		\begin{proof} For every $r\in \BN$ and $(i,j)\in I\times J$,   
			the canonical linear map 
			\[E_{i,r}\widehat\otimes_{\mathrm{i}}F_{j,r}
			=E_{i,r}\widehat\otimes F_{j,r}
			\rightarrow E_{i,r+1}\widehat\otimes_{\mathrm{i}}F_{j,r+1}
			=E_{i,r+1}\widehat\otimes F_{j,r+1}\]
			is a topological embedding by \cite[Proposition 43.7]{Tr}. 
			Then by  Lemma \ref{lem:basicsonindlim} (d),  
			$\bigoplus_{(i,j)\in I\times J}
			\varinjlim_{r\in \BN} (E_{i,r}\widehat\otimes F_{j,r})$ is complete. 
			Then \eqref{eq:indtensorLF} follows from \eqref{eq:ilimcomiten1}.

			For the second assertion, 
			we have that 
			\begin{eqnarray*}
				&&(E\widehat\otimes_{\mathrm{i}} F)_b'\\
				&=&\left(\bigoplus_{(i,j)\in I\times J}\left(\varinjlim_{r\in \BN} E_{i,r}\widehat\otimes F_{j,r}\right)\right)_b'\qquad\quad(\text{by \eqref{eq:indtensorLF}})\\
				&=&\prod_{(i,j)\in I\times J}\left(\varprojlim_{r\in \BN} (E_{i,r}\widehat\otimes F_{j,r})_b'\right)\qquad\qquad\,(\text{by\ \eqref{eq:sdualofprodandsum}, \eqref{eq:sdualofindlim}\ and\ Lemma \ref{lem:basicsonindlim}\,(c)})\\
				&=&\prod_{(i,j)\in I\times J}\left(\varprojlim_{r\in \BN}((E_{i,r})_b'\widehat\otimes (F_{j,r})_b')\right)\quad\quad\, \, (\text{by \ \eqref{eq:dualprincipal}})\\
				&=&\left(\prod_{i\in I}\varprojlim_{r\in \BN}(E_{i,r})_b'\right)\widehat\otimes \left(\prod_{j\in J}\varprojlim_{r\in \BN} (F_{j,r})_b'\right)\ (\text{by \cite[\S\,15.4, Theorem 2]{Ja}, \cite[(6.2)]{CSW}})\\
				&=&E_b'\widehat\otimes F_b'\quad\qquad\qquad\qquad\qquad\qquad\qquad\,(\text{by\ \eqref{eq:sdualofindlim}\ and\ Lemma \ref{lem:basicsonindlim}\,(c)}).
			\end{eqnarray*}
		\end{proof}

		\section{Vector-valued smooth functions }\label{appendixB}
		
		In this appendix, we review some basics of vector-valued smooth functions, and present some examples that are necessary in Sections \ref{sec:formalden}-\ref{sec:comformaldis}. 
		We also review the notion of 
		strict approximation property on LCS introduced in \cite{Sc3}.
		
		\subsection{Vector-valued smooth functions}\label{appendixB1}
		In Appendixes \ref{appendixB1} and \ref{appendixB2}, let 
		$N$ be a smooth manifold, and let  $E$ be a quasi-complete  LCS. 
		
		Write $\RC^\infty(N;E)$ for the space of all $E$-valued  smooth functions on $N$.
		Endow $\RC^\infty(N;E)$ with
		the topology defined by the seminorms 
		\[
		\abs{f}_{D, \abs{\,\cdot\,}_\mu}:= \sup_{a\in N}\abs{(Df)(a)}_\mu\qquad (f\in \RC^\infty(N;E)),
		\]
		where $D$ is a compactly supported differential operator on $N$, and $\abs{\,\cdot \,}_\mu$ is a continuous seminorm on $E$.
		Then $\RC^\infty(N;E)$ is a quasi-complete LCS.
		
		Let $\mathrm{C}_c^\infty(N;E)$ denote the space of all compactly supported $E$-valued smooth functions on $N$,
		and let $\RC^\infty_K(N;E)$ denote the space of those $E$-valued smooth functions  supported in $K$, where
		$K$ is a compact subset of $N$.
		Equip  $\RC^\infty_K(N;E)$ with the subspace topology of $\RC^\infty(N;E)$, and equip
		$\mathrm{C}_c^\infty(N;E)$ with the inductive limit topology. We also write 
		\[\RC^\infty_K(N):=\RC^\infty_K(N;\C)\quad \text{and}\quad \RC^\infty_c(N):=\RC^\infty_c(N;\C).\]
		
		Identify $\RC^\infty(N)\otimes E$ as a linear subspace of $\RC^\infty(N;E)$, which consists of the smooth functions
		whose image is contained in a finite-dimensional subspace of $E$.
		By %In view of
		\cite[Th\'{e}mr\`{e}me 1]{Sc}, we have that
		\be\label{eq:vvsfiso}
		\RC^\infty(N)\widetilde \otimes E=\RC^\infty(N;E)
		\ee as LCS, 
		and if $E$ is complete, then
		\be
		\label{eq:vvsfiso2}
		\RC^\infty(N)\wt\otimes E=\RC^\infty(N;E)=\RC^\infty(N)\wh\otimes E
		\ee
		as LCS.
		
		The identification \eqref{eq:vvsfiso} induces an LCS identification
		\be\label{eq:vvsfinKiso}
		\RC^\infty_K(N)\widetilde \otimes E= \RC^\infty_K(N;E)
		\ee
		for every compact subset $K$ of $N$.
		By taking the inductive limit, the identifications \eqref{eq:vvsfinKiso} for various $K$ yield a  continuous linear map 
		\begin{eqnarray}\label{eq:vvsfinciso}
			\mathrm{C}_c^\infty(N;E)\rightarrow \mathrm{C}_c^\infty(N)\widetilde\otimes_{\pi} E.
		\end{eqnarray}
		Note that the map \eqref{eq:vvsfinciso} is injective, but it may not be a linear topological embedding, and may not be a linear isomorphism (as abstract vector spaces). See \cite[Pages 107-108]{Sc} for details.
		
		\begin{remarkd}\label{rem:tildeCcNE} Assume that $N$ is secondly countable. In \cite{Sc}, Schwartz introduced the following variant of $\RC_c^\infty(N;E)$:
			\[ \widetilde{\mathrm{C}}_c^\infty(N;E):=\{f\in \RC^\infty(N;E)\mid
			(a\mapsto \la f(a), u'\ra)\in \RC_c^\infty(N)\ \text{for all}\ u'\in E'\}.\]
			The topology on $\widetilde{\mathrm{C}}_c^\infty(N;E)$ is given as follows: a net $\{f_i\}_{i\in I}$ in it converges to zero if and only if for
			every equicontinuous subset $H'$ of $E'$, the nets $\{\la f_i, u'\ra\}_{i\in I}$ for all $u'\in H'$ converge uniformly to zero in
			$\RC_c^\infty(N)$.
			There is an LCS identification  (see \cite[Th\'{e}mr\`{e}me 1]{Sc})
			\[\RC_c^\infty(N)\widetilde \otimes E=\widetilde{\RC}_c^\infty(N;E),\] and if $E$ is complete,  we also have that \quad \[
			\RC_c^\infty(N)\widehat \otimes E=\widetilde{\RC}_c^\infty(N;E)\] as LCS.
			Note that,
			if there exists a continuous norm on $E$, then $\widetilde{\RC}_c^\infty(N;E)=\RC_c^\infty(N;E)$ as abstract vector spaces, and their bounded subsets coincide.
			Furthermore, if $E$ is a DF space, then $\widetilde{\RC}_c^\infty(N;E)=\RC_c^\infty(N;E)$ as LCS. See
			\cite[Chapitre \uppercase\expandafter{\romannumeral2}, Pages 83-84]{Gr} for details.
		\end{remarkd}
		
		\subsection{Examples} \label{appendixB2}
		Here we present some examples of vector-valued smooth functions, which arise naturally in the study of vector-valued formal functions.
		
		\begin{exampled}\label{ex:E-valuedpowerseries}Let $k\in \BN$ and write
			\[
			E[[y_1, y_2, \dots, y_k]]:=\left\{\sum_{J\in \BN^k} v_J y^J\,:\, v_J\in E \textrm{ for all }J\in \BN^k\right\}
			\]
			for the space of the formal power series.
			Equip  $E[[y_1, y_2, \dots, y_k]]$ with the term-wise convergence topology, which is defined  by the seminorms
			\be \label{eq:defabsnuJ}\abs{\,\cdot\,}_{\nu,J}:\quad \sum_{I\in \BN^k} f_I y^I \mapsto \abs{f_J}_{\nu},\ee
			where $\abs{\,\cdot\,}_\nu$ is a continuous seminorm on $E$ and $J\in \BN^k$.
			When  $N=\BN^k$, we have  LCS identifications
			\begin{eqnarray*}
				E[[y_1, y_2, \dots, y_k]]
				=\RC^\infty(N; E)
				=\RC^\infty(N)\widetilde \otimes E
				=E \widetilde \otimes \BC[[y_1, y_2, \dots, y_k]],
			\end{eqnarray*}
			and $ E[[y_1, y_2, \dots, y_k]]=E \widehat \otimes \BC[[y_1, y_2, \dots, y_k]]$ is complete provided that $E$ is complete.
			Furthermore, if $E$ is a Fr\'echet space, then
			\[E[[y_1, y_2, \dots, y_k]]
			=E \widetilde \otimes_{\mathrm{i}} \BC[[y_1, y_2, \dots, y_k]]= E \widehat \otimes_{\mathrm{i}} \BC[[y_1, y_2, \dots, y_k]].\]
		\end{exampled}
		\begin{exampled}\label{ex:toponpolynomial}
			Let
			$k,r\in \BN$. Write $E[z_1,z_2,\dots,z_k]$ for the space of the  polynomials with coefficients in $E$, and write
			\[E[z_1,z_2,\dots,z_k]_{\le r}:=\left\{\sum_{L=(l_1,l_2,\dots,l_k)\in \BN^k;|L|\le r} v_L z^L\mid v_L\in E\ \text{for all}\ L\in \BN^k\right\}\]
			for the space of the polynomials with degree $\le r$, where  $z^L:=z_1^{l_1}z_2^{l_2}\cdots z_k^{l_k}$.
			Equip $E[z_1,z_2,\dots,z_k]_{\le r}$  with the term-wise convergence topology,
			and equip \[E[z_1,z_2,\dots,z_k]=\varinjlim_{r\in \BN}E[z_1,z_2,\dots,z_k]_{\le r}\] with the inductive limit topology.
			When $N=\BN^k$, we have 
			\be\label{eq:E[z]} E[z_1,z_2,\dots,z_k]=\RC_c^\infty(N;E)=\RC_c^\infty(N)\widetilde\otimes_{\mathrm{i}}E =E\wt\otimes_{\mathrm{i}} \C[z_1,z_2,\dots,z_k],
			\ee
			as LCS by  \eqref{eq:ilimcomiten} and Lemma \ref{lem:basicsonindlim} (e),
			and $E[z_1,z_2,\dots,z_k]=E\wh\otimes_{\mathrm{i}} \C[z_1,z_2,\dots,z_k]$ is complete provided that $E$ is complete.
			Furthermore, if $E$ is a   DF space, then (see Remark \ref{rem:tildeCcNE})
			\[E\wt\otimes \C[z_1,z_2,\dots,z_k]=E[z_1,z_2,\dots,z_k]=E\wh\otimes \C[z_1,z_2,\dots,z_k].\]
		\end{exampled}
		
		\begin{exampled}\label{ex:E-valuedcptsm}
			Suppose that  $E=\C[[y_1,y_2,\dots,y_k]]$ ($k\geq 1$). Then we have the following LCS identifications by  Lemmas \ref{lem:limofcomplete} and  \ref{lem:basicsonindlim} (d), (e):
			\begin{eqnarray*}
				&& \mathrm{C}_c^\infty(N;E)=\varinjlim_{K}\mathrm{C}_K^\infty(N; E)\\&=&\varinjlim_{K}(\mathrm{C}_K^\infty(N)\widetilde\otimes E)
				=\varinjlim_{K}(\mathrm{C}_K^\infty(N)\wh\otimes E)=\varinjlim_{K}\mathrm{C}_K^\infty(N)[[y_1,y_2,\dots,y_k]]\\
				&=&\mathrm{C}_c^\infty(N)\widetilde \otimes_{\mathrm{i}} E
				=\mathrm{C}_c^\infty(N)\widehat\otimes_{\mathrm{i}} E.
			\end{eqnarray*} Here $K$ runs over all compact subsets of $N$.
			If $N$ is not compact, then the space
			$\mathrm{C}_c^\infty(N;E)
			$ is  identified with  a proper linear subspace of
			\[\mathrm{C}_c^\infty(N)\widetilde\otimes E=
			\mathrm{C}_c^\infty(N)[[y_1,y_2,\dots,y_k]].\]
			
		\end{exampled}
		
		\begin{exampled} Suppose that $N$ is secondly countable, and $E=\C[z_1,z_2,\dots,z_k]$ ($k\geq 1$).
			By Remark \ref{rem:tildeCcNE} and \eqref{eq:E[z]},  we have  that 
			\[\mathrm{C}_c^\infty(N;E)=\widetilde{\mathrm{C}}_c^\infty(N;E)=\mathrm{C}_c^\infty(N)\widetilde\otimes E=\widetilde{\mathrm{C}}_c^\infty(\BN^k;\mathrm{C}_c^\infty(N))\] 
			and 
			\[\mathrm{C}_c^\infty(\BN^k;\mathrm{C}_c^\infty(N))=\mathrm{C}_c^\infty(N)[z_1,z_2,\dots,z_k]=\mathrm{C}_c^\infty(N)\widetilde\otimes_{\mathrm{i}}E\]
			as LCS. Furthermore, the above LCS are isomorphic to each other 
			as abstract vector spaces.
			However, when $N$ is not discrete, the topology on $\mathrm{C}_c^\infty(N)\widetilde\otimes_{\mathrm{i}}E$ is 
			strictly finer than that on $\mathrm{C}_c^\infty(N)\widetilde\otimes E$ (\cf \cite[Chapitre \uppercase\expandafter{\romannumeral2}, Page 85]{Gr}).
		\end{exampled}

		\subsection{The strict approximation property}
		\label{appendixB3} Let $E$ and $F$ be two LCS. 
		In the study of  vector-valued distributions (see \cite{Sc3}),
		it is more natural to equip $\CL(E,F)$
		with the topology of convex compact convergence. 
		As in \cite{Sc3}, we denote the resulting LCS by $\CL_c(E,F)$, and also denote $E_c':=\CL_c(E,\C)$.

		In what follows, we will often identify  $E'\otimes F$  with the space of all continuous linear maps of finite rank from $E$ to $F$.
		Following \cite{Gr,Sc3}, we make the following definition. 
		\begin{dfn}\label{de:strictappro}
			An LCS $E$ is said to  have the  approximation (resp.\,strict approximation) property if
			$E'\otimes E$ is dense (resp.\,strictly dense) in $\CL_c(E,E)$. 
		\end{dfn}
		
		Note that $E$ has the  approximation (resp.\,strict approximation) property if and only if the identity operator in $\CL(E,E)$ is contained in the closure (resp.\,quasi-closure) of
		$E'\otimes E$ in $\CL_c(E,E)$ (see \cite[Pr\'eliminaires, Page 5, Remarques]{Sc3}).
		Furthermore, if $E=(E_c')_c'$ as LCS,
		then the canonical map
		\[\CL_c(E,E)\rightarrow \CL_c(E_c',E_c'),\qquad f\mapsto {}^t f\]
		is a topological linear isomorphism (see \cite[Pr\'eliminaires, Page 11, Remarques 2]{Sc3}).
		Thus, in this case, $E$ has the approximation (resp.\,strict approximation) property if and only if  $E_c'$ does.

		According to \cite{Sc3}, we have the following result, which is useful in the studying of vector-valued distributions.
		
		\begin{lemd}\label{lem:stricappro}
			Assume that
			\begin{itemize}
				\item $E_c'$ and $F$ are quasi-complete (resp.\,complete);
				\item $(E_c')_c'=E$ as LCS; and
				\item $E$ has the strict approximation (resp.\,approximation) property.
			\end{itemize}
			Then the canonical injective  linear map $E'\otimes F\rightarrow \CL(E,F)$
			induces an LCS identification \[E'_c\wt\otimes_\varepsilon F=\CL_c(E,F)\quad (\text{resp.}\ E'_c\wh\otimes_\varepsilon F=\CL_c(E,F)).\]
		\end{lemd}
		\begin{proof}
			The first  condition in the lemma implies that
			$\CL_c(E,F)$ is topologically isomorphic to the $\varepsilon$-product $E_c' \varepsilon F$ of
			$E_c'$ and $F$ (see \cite[Chapitre I, Corolliare of Poposition 5]{Sc3}).
			On the other hand, the second and the third conditions imply that
			the $\varepsilon$-product $E_c' \varepsilon F$  is topologically isomorphic to
			$E'_c\wt\otimes_\varepsilon F$ (resp.\,$E'_c\wh\otimes_\varepsilon F$) (see \cite[Chapitre I, Corolliare 1 of Poposition 11]{Sc3}).
			This proves the lemma.
		\end{proof}
		Note that if $E$  is quasi-complete and nuclear, then $\CL_c(E,F)=\CL_b(E,F)$ as LCS (see \cite[Corollary 1 of Proposition 50.2]{Tr}).
		According to \cite[\S\,21.5, Corollary 5]{Ja}, we have the following result.
		\begin{lemd}\label{lem:dualofnF}
			Assume that  $E$ is a nuclear LF space.
			Then $E'_b$ is complete and nuclear. Moreover $E'_b=E_c'$ and $E=(E_b')_b'=(E_c')_c'$ as LCS.
		\end{lemd}
		
		%If $E$ satisfies the conditions of Lemma \ref{lem:dualofnF}, then it is reflexive, and we have that  \[E=(E_b')_b'=(E_c')_c'\] as LCS.
		Using Lemma \ref{lem:dualofnF}, we have the following consequence of Lemma \ref{lem:stricappro}.

		\begin{cord}\label{cor:conmaps=tensorproduct}  Assume that $F$ is quasi-complete, and $E$ is a nuclear LF space that has the strict approximation property. Then $\CL_b(E,F)$ is quasi-complete and
			\[E_b'\,\wt\otimes F=\CL_b(E,F)\]
			as LCS. Furthermore, if $F$ is complete, then 
			$\CL_b(E,F)$ is complete and
			\[E_b'\,\wt\otimes F=\CL_b(E,F)=E_b'\,\wh\otimes F.\]
			as LCS.
		\end{cord}

		It is proved in \cite[Pr\'eliminaires, Proposition 1]{Sc3} that the LCS
		$\RC^\infty_c(\R^n), n\in \BN$
		have the strict approximation property.
		For the purpose of studying the vector-valued formal distributions and formal generalized functions,  the following generalization is necessary.
		
		We use $\mathrm{id}_X$ to denote the identity morphism of an object $X$ in a category. 
		
		\begin{lemd}\label{lem:strictapppoly} Let $n,k\in \BN$. Then the LCS
			\[
			\RC_c^\infty(\R^n)\wt\otimes_{\mathrm{i}} \C[[y_1,y_2,\dots,y_k]]\quad \text{and}\quad
			\RC^\infty_c(\R^n)\wt\otimes_{\mathrm{i}}\C[y_1,y_2,\dots,y_k]\]
			have the strict approximation property.
		\end{lemd}
		\begin{proof}
			In the proof of \cite[Pr\'eliminaires, Proposition 1]{Sc3}, Schwartz constructed  a sequence $\{L_j\}_{j\in \BN}$  of  linear
			maps  from $\RC_c^\infty(\R^n)$ to $\RC_c^\infty(\R^n)$  such that
			\begin{itemize}
				\item for every $j\in \BN$, $L_j$ is continuous and has finite rank;
				\item $\lim_{j\in \BN} L_j=\mathrm{id}_{\RC_c^\infty(\R^n)}$ in $\CL_b(\RC_c^\infty(\R^n),\RC_c^\infty(\R^n))$; and
				\item for every compact subset $K$ of $\R^n$, there is a compact subset $K^*$ of $\R^n$ such that
				$L_j(\RC^\infty_K(\R^n))\subset \RC_{K^*}^\infty(\R^n)$ for all $j\in \BN$.
			\end{itemize}
			
			For every $j\in \BN$, define the  continuous linear map
			\begin{eqnarray*}
				P_j: \RC_c^\infty(\R^n)[[y_1,y_2,\dots,y_k]]&\rightarrow& \RC_c^\infty(\R^n)[[y_1,y_2,\dots,y_k]],\\
				\sum_{I\in \BN^k} f_I y^I
				&\mapsto& \sum_{|I|\le j} L_j(f_I) y^I.
			\end{eqnarray*}
			Using Example \ref{ex:E-valuedcptsm}, we  view
			\[\RC_c^\infty(\R^n)\wt\otimes_{\mathrm{i}}\C[[y_1,y_2,\dots,y_k]]=\varinjlim_{ \text{$K$ is a compact subset of } \BR^n} \RC_K^\infty(\R^n)[[y_1,y_2,\dots,y_k]]
			\]
			as a linear subspace of $\RC_c^\infty(\R^n)[[y_1,y_2,\dots,y_k]]$.
			By restriction, $P_j$ becomes a continuous linear map  of finite rank from $\RC_c^\infty(\R^n)\wt\otimes_{\mathrm{i}}\C[[y_1,y_2,\dots,y_k]]$ to
			itself.

			Let $B$ be a bounded subset of $\RC_c^\infty(\R^n)\wt\otimes_{\mathrm{i}}\C[[y_1,y_2,\dots,y_k]]$, and $\abs{\,\cdot\,}_{\nu}$ a continuous seminorm on
			$\RC_c^\infty(\R^n)\wt\otimes_{\mathrm{i}}\C[[y_1,y_2,\dots,y_k]]$. 
			Then by Lemma \ref{lem:basicsonindlim} (c),  there is a compact subset $K$ of $\R^n$ and a family $\{B_I\}_{I\in \BN^k}$ of bounded subsets in
			$\RC_K^\infty(\R^n)$ such that
			\be \label{eq:desboundedB} B\subset\left\{\sum_{I\in \BN^k} f_I y^I\mid f_I\in B_I\ \text{for all $I\in \BN^k$}\right\}.\ee
			Since $P_j\left(\RC_K^\infty(\R^n)[[y_1,y_2,\dots,y_k]]\right)\subset \RC_{K^*}^\infty(\R^n)[[y_1,y_2,\dots,y_k]]$ for all $j\in \BN$,
			there exist finite  continuous seminorms $\abs{\,\cdot\,}_{\nu_1},\abs{\,\cdot\,}_{\nu_2},\dots,\abs{\,\cdot\,}_{\nu_r}$  on $\RC^\infty(\R^n)$ and finite $k$-tuples
			$J_1,J_2,\dots,J_r$ in $\BN^k$ such that %(see \eqref{eq:defabsBnu} and \eqref{eq:defabsnuJ})
			\begin{eqnarray*}
				\abs{P_j-\mathrm{id}_{\RC_c^\infty(\R^n)\wt\otimes_{\mathrm{i}}\BC[[y_1,y_2,\dots,y_k]]}}_{B,\abs{\,\cdot\,}_{\nu}}
				\le \sum_{i=1}^r \abs{P_j-\mathrm{id}_{\RC_c^\infty(\R^n)\wt\otimes_{\mathrm{i}}\BC[[y_1,y_2,\dots,y_k]]}}_{B,\abs{\,\cdot\,}_{\nu_i,J_i}}.
			\end{eqnarray*} Here $\abs{\,\cdot\,}_{B,\abs{\,\cdot\,}_\nu}$ is defined as in \eqref{eq:defabsBnu}, and $\abs{\,\cdot\,}_{\nu_i,J_i}$ is defined as in \eqref{eq:defabsnuJ}.
			On the other hand, for $i=1,2,\dots,r$ and $j\ge |J_i|$ we have that
			\begin{eqnarray*}
				\abs{P_j-\mathrm{id}_{\RC_c^\infty(\R^n)\wt\otimes_{\mathrm{i}}\BC[[y_1,y_2,\dots,y_k]]}}_{B,\abs{\,\cdot\,}_{\nu_i,J_i}}
				\leq \abs{L_j-\mathrm{id}_{\RC_c^\infty(\R^n)}}_{B_{J_i},\abs{\,\cdot\,}_{\nu_i}}.
			\end{eqnarray*}
			Thus the sequence $\{P_j\}_{j\in \BN}$ converges to the identity operator uniformly on the  bounded subsets of
			$\RC_c^\infty(\R^n)\wt\otimes_{\mathrm{i}}\C[[y_1,y_2,\dots,y_k]]$,
			and so $\RC_c^\infty(\R^n)\wt\otimes_{\mathrm{i}}\C[[y_1,y_2,\dots,y_k]]$ has the strict approximation property.
			
			Similarly, by Lemma \ref{lem:indtensorLF}, we view
			\begin{eqnarray*}
				\RC_c^\infty(\R^n)\wt\otimes_{\mathrm{i}}\C[y_1,y_2,\dots,y_k]
				=\varinjlim_{(K,r)\in \mathcal{C}(\R^n)} \RC_K^\infty(\R^n)[y_1,y_2,\dots,y_k]_{\le r}
			\end{eqnarray*}
			as a linear subspace of $\RC_c^\infty(\R^n)[[y_1,y_2,\dots,y_k]]$, where the directed set $\mathcal{C}(\R^n)$ is defined as in \eqref{eq:C(M)}.
			By taking the restriction, each $P_j$  becomes a continuous linear map  of finite rank from $\RC_c^\infty(\R^n)\wt\otimes_{\mathrm{i}}\C[y_1,y_2,\dots,y_k]$ to
			itself.
			Furthermore,  the sequence $\{P_j\}_{j\in \BN}$ converges to the identity operator uniformly on the bounded subsets  of $\RC_c^\infty(\R^n)\wt\otimes_{\mathrm{i}}\C[y_1,y_2,\dots,y_k]$, as required.
		\end{proof}

		\section*{Acknowledgement}
		F. Chen is supported  by the National Natural Science Foundation of China (Nos. 12131018, 12161141001) and the Fundamental Research Funds for the Central Universities (No. 20720230020).
		B. Sun is supported by  National Key R \& D Program of China (Nos. 2022YFA1005300 and 2020YFA0712600) and New Cornerstone Investigator Program. 
		The first and the third authors would like to thank Institute for Advanced Study in Mathematics, Zhejiang University. Part of this work was carried out while they were visiting the institute.

	\end{document}